\documentclass[a4paper,10pt,leqno]{amsart}
\title
{$L^2$-Euler characteristics and the Thurston norm}

\author{Stefan Friedl}
\address{Fakult\"at f\"ur Mathematik\\ Universit\"at Regensburg\\93040 Regensburg\\ Germany}
\email{sfriedl@gmail.com}

\author{Wolfgang L\"uck}
 \address{Mathematisches Institut der Universit\"at Bonn\\
 Endenicher Allee 60\\
 53115 Bonn, Germany}
 \email{wolfgang.lueck@him.uni-bonn.de}
 \urladdr{http://www.him.uni-bonn.de/lueck}
 \date{September, 2018}

\usepackage{hyperref}
\usepackage{color}
\usepackage{pdfsync}
\usepackage{calc,enumitem,stmaryrd}
\usepackage{amssymb}
\usepackage[arrow,curve,matrix,tips,2cell]{xy}
 \SelectTips{eu}{10} \UseTips
 \UseAllTwocells

\DeclareMathAlphabet\EuR{U}{eur}{m}{n}
\SetMathAlphabet\EuR{bold}{U}{eur}{b}{n}

\makeindex 

\theoremstyle{plain}
\newtheorem{theorem}{Theorem}[section]
\newtheorem{lemma}[theorem]{Lemma}

\newtheorem{corollary}[theorem]{Corollary}
\newtheorem{conjecture}[theorem]{Conjecture}

\newtheorem{question}[theorem]{Question}

\theoremstyle{definition}

\newtheorem{definition}[theorem]{Definition}
\newtheorem{example}[theorem]{Example}

\newtheorem{remark}[theorem]{Remark}
\newtheorem{notation}[theorem]{Notation}

{\catcode`@=11\global\let\c@equation=\c@theorem}

\newcommand{\comsquare}[8] 
{\begin{CD}
#1 @>#2>> #3\\
@V{#4}VV @V{#5}VV\\
#6 @>#7>> #8
\end{CD}
}

\newcommand{\xycomsquare}[8] 
{\xymatrix
{#1 \ar[r]^{#2} \ar[d]^{#4} &
#3 \ar[d]^{#5} \\
#6\ar[r]^{#7} &
#8
}
}

\newcommand{\xycomsquareminus}[8] 
{\xymatrix{#1 \ar[r]^-{#2} \ar[d]^-{#4} &
#3 \ar[d]^-{#5} \\
#6\ar[r]^-{#7} &
#8
}
}

\newcommand{\cald}{\mathcal{D}}

\newcommand{\calc}{\mathcal{C}}

\newcommand{\caln}{{\mathcal N}}
\newcommand{\calp}{{\mathcal P}}

\newcommand{\calsn}{\mathcal{SN}}
\newcommand{\calu}{\mathcal{U}}

\newcommand{\IC}{{\mathbb C}}

\newcommand{\IN}{{\mathbb N}}

\newcommand{\IP}{{\mathbb P}}
\newcommand{\IQ}{{\mathbb Q}}
\newcommand{\IR}{{\mathbb R}}

\newcommand{\IZ}{{\mathbb Z}}

\newcommand{\abel}{\operatorname{abel}}

\newcommand{\aut}{\operatorname{aut}}

\newcommand{\coker}{\operatorname{coker}}

\newcommand{\Ext}{\operatorname{Ext}}
\newcommand{\ev}{\operatorname{ev}}

\newcommand{\id}{\operatorname{id}}

\newcommand{\im}{\operatorname{im}}

\newcommand{\orb}{\operatorname{orb}}
\newcommand{\pr}{\operatorname{pr}}

\newcommand{\sn}{\operatorname{sn}}

\newcommand{\tors}{\operatorname{tors}}

\newcommand{\Hom}{\operatorname{Hom}}

\newcommand{\lmsum}[2]{\mbox{\small$\displaystyle\sum\limits_{#1}^{#2}$}}
\newcommand{\smsum}[2]{\mbox{\footnotesize$\displaystyle\sum\limits_{#1}^{#2}$}}
\newcommand{\tmsum}[2]{\mbox{$\textstyle \sum\limits_{#1}^{#2}$}}

\newcommand{\lmfrac}[2]{\mbox{\small$\displaystyle\frac{#1}{#2}$}} 
\newcommand{\smfrac}[2]{\mbox{\footnotesize$\displaystyle\frac{#1}{#2}$}} 
\newcommand{\tmfrac}[2]{\mbox{\large$\frac{#1}{#2}$}}

\newcommand{\higherlim}[3]{{\setbox1=\hbox{\rm lim}
 \setbox2=\hbox to \wd1{\leftarrowfill} \ht2=0pt \dp2=-1pt
 \mathop{\vtop{\baselineskip=5pt\box1\box2}}
 _{#1}}^{#2}#3}

\newcommand{\version}[1] 
{\begin{center} last edited on #1\\
last compiled on \today\\
name of texfile: \jobname
\end{center}
}

\newcounter{commentcounter}

\begin{document}

\typeout{---------------------------- twist.tex ----------------------------}


\typeout{------------------------------------ Abstract ----------------------------------------}

\begin{abstract}
 We assign to a finite $CW$-complex and an element in its first cohomology group a twisted
 version of the $L^2$-Euler characteristic and study its main properties. In the case of
 an irreducible orientable $3$-manifold with empty or toroidal boundary and infinite
 fundamental group we identify it with the Thurston norm. We will use the twisted $L^2$-Euler
 characteristic to address the problem whether the existence of a map inducing an
 epimorphism on fundamental groups implies an inequality of the Thurston norms. 
\end{abstract}

\maketitle


\typeout{------------------------------- Section 0: Introduction
 --------------------------------}

\setcounter{section}{-1}
\section{Introduction}


\subsection{The $(\mu,\phi)$-$L^2$-Euler characteristic}
\label{subsec:The_(mu,phi)_L2-Euler_characteristic_cw-intro}
Let $X$ be a finite connected CW-com\-plex and let $\mu\colon \pi_1(X) \to G$ and $\phi \colon G \to \IZ$ be group
homomorphisms. We say that $(\mu,\phi)$ is an \emph{$L^2$-acyclic Atiyah pair}
if the $n$th $L^2$-Betti number
$b_n^{(2)}(\overline{X};\caln(G))$ of the $G$-covering $\overline{X} \to X$ associated to
$\mu$ vanishes for all $n \ge 0$, and $G$ is torsion-free and satisfies
the Atiyah Conjecture. (We will discuss the Atiyah Conjecture in 
Section~\ref{sec:About_the_Atiyah_Conjecture}). 
Then one can define
by twisting the cellular $\IZ G$-chain complex with the infinite dimensional
$G$-representation $\phi^*\IR \IZ$ the \emph{$(\mu,\phi)$-$L^2$-Euler characteristic}
$\chi^{(2)}(X;\mu,\phi)$ which is an  integer. We will not give the precise definition of $(\mu,\phi)$-$L^2$-Euler 
characteristic in the introduction but we refer to
Section~\ref{sec:Twisting_the_L2-Euler_characteristic_with_a_cocycle_in_the_first_cohomology}
for details and a summary of the key properties.

The $(\mu,\phi)$-$L^2$-Euler characteristic can be employed in many different contexts.
For example it is used by Funke--Kielak~\cite{Funke-Kielak(2016)} to study descending HNN-extensions of 
free groups and it is at least implicitly used by the authors and Tillmann~\cite{Friedl-Lueck-Tillmann(2016)} to study one-relator groups.


\subsection{The $(\mu,\phi)$-$L^2$-Euler characteristic of 3-manifolds}
\label{subsec:The_phi_L2-Euler_characteristic_for_universal_coverings_intro}
In this paper our main application of the $(\mu,\phi)$-$L^2$-Euler characteristic  lies in the study of 3-manifolds 
to which we restrict ourselves in the remainder of the introduction.
More precisely, our main focus  will be on the following class of $3$-manifolds.

\begin{definition}[Admissible $3$-manifold]\label{def:admissible_3-manifold} 
 A $3$-manifold is called \emph{admissible} if it is connected, orientable, compact and irreducible,
 its boundary is empty or a disjoint union of tori, and its fundamental group
 is infinite.
\end{definition}

Let $M$ be an admissible $3$-manifold and let 
$\phi \colon \pi_1(M) \to \IZ$ be a group homomorphism. Then 
all the conditions listed in Section~\ref{subsec:The_(mu,phi)_L2-Euler_characteristic_cw-intro} are satisfied 
for the triple $(X,\id_{\pi_1(M)},\phi)$ and the  corresponding $L^2$-Euler characteristic 
$\chi^{(2)}(M;\id_{\pi_1(M)},\phi)$ is defined. We denote by $x_M(\phi)$ the Thurston norm of $\phi$ 
which is loosely speaking defined as the minimal complexity of a surface dual to $\phi$.
 (We recall the precise definition of the Thurston norm in Section~\ref{subsec:Brief_review_of_the_Thurston_norm}.) 
The following is one of our main theorems.

\begin{theorem}[Equality of $(\mu,\phi)$-$L^2$-Euler characteristic and the Thurston norm
 for universal coverings]%
\label{the:Equality_of_(mu,phi)-L2-Euler_characteristic_and_the_Thurston_norm_universal_covering_for_universal_coverings}
 Let $M\ne S^1\times D^2$ be an admissible $3$-manifold.

 Then we get for any $\phi\in H^1(M;\IZ)$
 \[
 -\chi^{(2)}(M;\id_{\pi_1(M)},\phi) \,\,=\,\, x_M(\phi).
 \]
\end{theorem}

If $M$ is not a closed graph manifold, then
Theorem~\ref{the:Equality_of_(mu,phi)-L2-Euler_characteristic_and_the_Thurston_norm_universal_covering_for_universal_coverings}
is a direct consequence of the subsequent
Theorem~\ref{the:Equality_of_(mu,phi)-L2-Euler_characteristic_and_the_Thurston_norm}
together with the fact that in this case the fundamental groups of $M$ satisfies the
Atiyah Conjecture, see
Theorem~\ref{the:Status_of_the_Atiyah_Conjecture}. If $M$ is a (closed) graph manifold
Theorem~\ref{the:Equality_of_(mu,phi)-L2-Euler_characteristic_and_the_Thurston_norm_universal_covering_for_universal_coverings}
follows from
Theorem~\ref{the:The_phi-L2-Euler_characteristic_and_the_Thurston_norm_for_graph_manifolds}.
\medskip

It is also interesting to consider group homomorphisms $\pi_1(M)\to G$ that are not the identity. 
For example in Section~\ref{sec:The_(mu,phi)-L2-Euler_characteristic_and_the_degree_of_higher_order_Alexander_polynomials}
we will see that if $G$ is a torsion-free elementary amenable group, then the
$(\mu,\phi)$-$L^2$-Euler characteristic $\chi^{(2)}(M;\mu,\phi)$ is basically the same as
the degrees of the non-commutative Alexander polynomials studied by
Cochran~\cite{Cochran(2004)}, Harvey~\cite{Harvey(2005)} and the first
author~\cite{Friedl(2007)}. In these three papers it was shown that the degrees of
non-commutative Alexander polynomials give lower bounds on the Thurston norm. The following
theorem can be viewed as a generalization of these results.
\\

\noindent \textbf{Theorem~\ref{the:The_Thurston_norm_ge_the_(mu,phi)-L2-Euler_characteristic}.}(The negative of the $(\mu,\phi)$-$L^2$-Euler characteristic is a lower bound for the Thurston norm) \emph{Let $M\ne S^1\times D^2$ be an admissible $3$-manifold and let $(\mu,\phi)$ be an $L^2$-acyclic Atiyah-pair.
Then $M$ is $(\mu,\phi)$-$L^2$-finite and we get
\[
-\chi^{(2)}(M;\mu,\phi) \,\,\le \,\,x_M(\phi \circ \mu).
\]}

In general the inequality of the above theorem is not an equality, for example if $M=S^3\setminus \nu K$ is the exterior of a non-trivial knot and $\mu=\phi\colon \pi_1(M)\to \IZ$ is the abelianization, then the left hand side equals 
$2\deg(\Delta_K(t))-1$ where $\Delta_K(t)$ equals  the Alexander polynomial of $K$ and the right hand side equals $2\mbox{genus}(K)-1$. But there are many knots for which the degree of the Alexander polynomial is less than twice the genus of $K$. 

But the following theorem, which will be proved in Section~\ref{subsec:Proof_of_Theorem_ref(the:Equality_of_(mu,phi)-L2-Euler_characteristic_and_the_Thurston_norm)}, shows 
that for any `sufficiently large epimorphism', the inequality does become an equality.

\begin{theorem}[Equality of $(\mu,\phi)$-$L^2$-Euler characteristic and the Thurston norm]
\label{the:Equality_of_(mu,phi)-L2-Euler_characteristic_and_the_Thurston_norm}
Let $M\ne S^1\times D^2$ be an admissible $3$-manifold which is not a closed graph manifold.

Then there exists a virtually abelian group $\Gamma$ and a factorization 
\[
\pr_M \colon \pi_1(M) \xrightarrow{\alpha} \Gamma \xrightarrow{\beta} H_1(M)_f := H_1(M)/\tors(H_1(M))
\]
of the canonical projection $\pr_M\colon \pi_1(M)\to H_1(M)_f$ into epimorphisms such that the following holds:

Given a group $G$ satisfying the Atiyah Conjecture, a factorization of $\alpha \colon
\pi_1(M) \to \Gamma$ into group homomorphisms $\pi_1(M) \xrightarrow{\mu} G \xrightarrow{\nu}
\Gamma$, and a group homomorphism $\phi \colon H_1(M)_f \to \IZ$, the pair $(\mu, \phi
\circ \beta \circ \nu)$ is an $L^2$-acyclic Atiyah-pair, and we get
\[
 -\chi^{(2)}(M;\mu,\phi \circ \beta \circ \nu)~=x_M(\phi).
\]
\end{theorem}

With a little bit of extra effort one can use
Theorem~\ref{the:Equality_of_(mu,phi)-L2-Euler_characteristic_and_the_Thurston_norm} to
show that one can use epimorphisms onto torsion-free elementary amenable groups to detect
the Thurston norm. Put differently, one can show that the aforementioned non-commutative
Alexander polynomials detect the Thurston norm. We refer to
Corollary~\ref{cor:existence_of_torsion-free_elementary_coverings_with_equality} for the
precise statement.

\subsection{Inequality of the Thurston norm}
\label{subsec:Inequality_of_the_THurston_norm}

One of the key motivations for developing the theory of $L^2$-Euler characteristics is the following question by 
Simon~\cite[Problem~1.12]{Kirby(1997)}.

\begin{question}\label{que:Simons_question}
Let $K$ and $K'$ be two knots. If there is an epimorphism from the knot group of $K$ to
the knot group of $K'$, does this imply that the genus of $K$ is greater than or equal to the
genus of $K'$?
\end{question}

Note that, as is pointed out in  \cite[p.~278]{Kirby(1997)}, it follows from classical results that the inequality holds if the degree of the Alexander polynomial of $K'$ is twice the genus of $K'$, in particular the inequality holds if $K'$ is fibered. 
The approach we take in attacking Question~\ref{que:Simons_question} owes some intellectual debt to \cite{Kitano-Suzuki-Wada(2005)} and especially \cite[Corollary~3.12]{Harvey(2006)}.

We propose the following conjecture. 

\begin{conjecture}[Inequality of the Thurston norm]\label{con:Inequality_of_the_Thurston_norm} 
 Let $f \colon M \to N$ be a map between
 admissible $3$-manifolds which is surjective on $\pi_1(N)$ and induces an isomorphism
 $f_* \colon H_n(M;\IQ) \to H_n(N;\IQ)$ for $n \ge 0$.
 
 Then we get for any $\phi \in H^1(N;\IR)$ that
 \[
 x_M(f^*\phi) \,\,\ge\,\, x_N(\phi).
 \]
\end{conjecture}

\begin{remark}
\begin{enumerate}
\item In Section~\ref{subsec:Brief_review_of_the_Thurston_norm} we will recall that the
 Thurston norm can be viewed as a generalization of the knot genus. In particular a proof
 of  Conjecture~\ref{con:Inequality_of_the_Thurston_norm} 
 would give an affirmative answer to Simon's question.
\item If $M$ and $N$ are closed 3-manifolds, then the conclusion of 
 Conjecture~\ref{con:Inequality_of_the_Thurston_norm} follows from~\cite[Corollary~6.18]{Gabai(1983)}.
\item The condition on the induced map on rational homology cannot be dropped. For
 example, suppose that $M=S^1\times \Sigma$ with $\Sigma$ a surface of genus $g\geq 2$
 with boundary. Let $N$ be the exterior of a non-trivial torus knot. Then $\pi_1(N)$ is generated by
 two elements, therefore there exists an epimorphism $\varphi\colon \pi_1(S^1\times
 \Sigma)\to \pi_1(N)$ which factors through the projection $\pi_1(S^1\times \Sigma)\to
 \pi_1(\Sigma)$. Since $N$ is aspherical there exists a map $f\colon S^1\times \Sigma\to N$ with $\varphi=f_*$. If $\phi$ is a generator of $H^1(N;\IZ)$, then $x_N(\phi)\ne 0$. On the other hand $f^*\phi$ is dual to ``vertical tori'' in $S^1\times \Sigma$, which implies that $x_M(f^*\phi)=0$. We are grateful to Yi Liu for pointing
 out this example.
\end{enumerate}
\end{remark}

Before we state our main contribution to
Conjecture~\ref{con:Inequality_of_the_Thurston_norm} we need to recall one more
definition. A group $G$ is called \emph{locally indicable} if any finitely generated
non-trivial subgroup of $G$ admits an epimorphism onto $\IZ$. For example
Howie~\cite{Howie(1982)} showed that the fundamental group of any admissible 3-manifold
with non-trivial boundary is locally indicable.
The main use of groups being locally indicable is that we can apply the Gersten-Howie-Schneebeli Theorem~\ref{thm:howie-schneebeli-gersten}.

Our main result is Theorem~\ref{the:Inequality_of_the_Thurston_norm}. \\

\noindent \textbf{Theorem~\ref{the:Inequality_of_the_Thurston_norm}.}
\emph{Let $f \colon M \to N$ be a map of admissible $3$-manifolds which is surjective on $\pi_1(N)$ 
and induces an isomorphism $f_* \colon H_n(M;\IQ) \to H_n(N;\IQ)$ for $n \ge 0$. Suppose that
$\pi_1(N)$ is residually locally indicable elementary amenable. 
Then we get for any $\phi\in H^1(N;\IR)$ that
 \[
 x_M(f^*\phi ) \,\,\ge\,\, x_N(\phi).
 \]}

By Lemma~\ref{lem:fibered-manifolds-acapl} the fundamental group of any fibered 3-manifold
is residually locally indicable elementary amenable.  Thus we have proved
Conjecture~\ref{con:Inequality_of_the_Thurston_norm} in particular in the case that $N$ is
fibered. The conclusion of Conjecture~\ref{con:Inequality_of_the_Thurston_norm} can be
proved relatively easily for \emph{fibered classes in $H^1(N;\IR)$}, but it seems to us
that if $N$ is fibered, then there is no immediate reason why the inequality should hold
for \emph{non-fibered} classes in $H^1(N;\IR)$.

We propose the following

\begin{conjecture}\label{con:res_loc_ind_elem_ama_always}
The fundamental group of any admissible $3$-manifold $M$ with $b_1(M) \ge 1$
is residually locally indicable elementary amenable. 
\end{conjecture}

A proof of Conjecture~\ref{con:res_loc_ind_elem_ama_always} together with 
Theorem~\ref{the:Inequality_of_the_Thurston_norm} implies 
Conjecture~\ref {con:Inequality_of_the_Thurston_norm} and in particular an affirmative answer to Simon's 
Question~\ref{que:Simons_question}.


\subsection{The $(\mu,\phi)$-$L^2$-Euler characteristic and the degree of the $L^2$-torsion 
function}%
\label{subsec:The_(mu,phi)-L2-Euler_characteristic_and_the_degree_of_the_L2-torsion_function_intro}

We briefly discuss a relation of the $(\mu,\phi)$-$L^2$-Euler characteristic to the degree of the $L^2$-torsion 
function in Section~\ref{sec:The_degree_of_the_L2-torsion_function}.
The $L^2$-torsion function is defined in~\cite{Dubois-Friedl-Lueck(2014Alexander)} and~\cite{Lueck(2015twisting)}.


\subsection{Methods of proof}\label{subsec:Methods_of_proof}
The Ore localization of a group ring $\IZ G$ is known to exist for torsion-free elementary amenable groups, but it  is definitely not available
if $G$ contains a non-abelian free subgroup, which is the case for most fundamental groups of 3-manifolds.
One key ingredient in this paper is therefore to replace the Ore localization of a group ring $\IZ G$  by the division closure $\cald(G)$ of
$\IZ G$ in the algebra $\calu(G)$ of operators affiliated to the group von Neumann algebra $\caln(G)$.
This is a well-defined skew field containing $\IZ G$ if and only if $G$ is torsion-free and satisfies
the Atiyah Conjecture with rational coefficients. This is known to be true in many interesting cases.

We will also take advantage of the recent proof by Agol and others of the Virtual Fibering Conjecture.



\subsection*{Acknowledgments.}
The first author gratefully acknowledges the support provided by the SFB 1085 ``Higher
Invariants'' at the University of Regensburg, funded by the Deutsche
Forschungsgemeinschaft {DFG}. The paper is financially supported by the Leibniz-Preis of the second author granted by the {DFG}
and the ERC Advanced Grant ``KL2MG-interactions'' (no. 662400) of the second author granted by the European Research Council. We wish to thank the referee for reading an earlier version thoroughly and for giving lots of helpful feedback.
We are particularly grateful to the referee for pointing out to use the proof of 
(3c) of Theorem~\ref{the:Status_of_the_Atiyah_Conjecture}.


\tableofcontents


\typeout{------------ Section 1: Brief review of the Thurston norm -------------}

\section{Brief review of the Thurston norm}
\label{subsec:Brief_review_of_the_Thurston_norm}
We recall the definition in~\cite{Thurston(1986norm)} of the \emph{Thurston norm} $x_M(\phi)$
of a compact connected orientable $3$-manifold $M$ with empty or non-empty boundary 
and an element $\phi \in H^1(M;\IZ)=\Hom(\pi_1(M),\IZ)$. It is defined as
\[
x_M(\phi)\,\,:=\,\,\min \{ \chi_-(F)\, | \, F \subset M \mbox{ properly embedded surface dual to
}\phi\},
\]
where, given a surface $F$ with connected components $F_1, F_2, \ldots , F_k$, we define
\[\chi_-(F)\,\,:=\,\,\smsum{i=1}{k} \max\{-\chi(F_i),0\}.\]

Thurston~\cite{Thurston(1986norm)} showed that this defines a seminorm on $H^1(M;\IZ)$
which can be extended to a seminorm on $H^1(M;\IR)$ which we denote by $x_M$ again.
In particular we get for $r \in \IR$ and $\phi \in H^1(M;\IR)$
\begin{eqnarray}
 x_M(r \cdot \phi) 
 & = & 
 |r| \cdot x_M(\phi).
\label{scaling_Thurston_norm}
\end{eqnarray}

If $K\subset S^3$ is a knot, then we denote by $\nu K$ an open tubular neighborhood of $K$
and we refer to $X_K=S^3\setminus \nu K$ as the exterior of $K$. We refer to the minimal
genus of a Seifert surface of $K$ as the \emph{genus $g(K)$ of $K$}. We have
$H^1(X_K;\IZ)\cong \IZ$ and an elementary exercise shows that for any generator $\phi$ of
$H^1(X_K;\IZ)\cong \IZ$ we have
\begin{equation} x_{X_K}(\phi)\,\,=\,\,\operatorname{max}\{
 2g(K)-1,0\}.\label{equ:genus-thurston-norm} \end{equation}

If $p \colon M' \to M$ is a finite covering with $n$ sheets, then Gabai~\cite[Corollary~6.13]{Gabai(1983)} showed
\begin{eqnarray}
 x_{M'}(p^*\phi) 
 & = & 
 n \cdot x_M(\phi).
\label{finite_coverings_Thurston_norm}
\end{eqnarray}
If $F \to M \xrightarrow{p} S^1$ is a fiber bundle for a compact connected orientable
$3$-manifold $M$ and compact surface $F$ and $\phi \in H^1(M;\IZ)$ is given by $H_1(p)
\colon H_1(M) \to H_1(S^1)$, then by~\cite[Chapter~3]{Thurston(1986norm)}  we have
\begin{eqnarray}
 x_M(\phi) & = & 
 \begin{cases}
 - \chi(F) & \text{if} \;\chi(F) \le 0;
 \\
 0 & \text{if} \;\chi(F) \ge 0.
 \end{cases}
\label{fiber_bundles_Thurston_norm}
\end{eqnarray}


\typeout{-- Section 2: Twisting the $L^2$-Euler characteristic with a cocycle in the first cohomology -----------}

\section{Twisting the $L^2$-Euler characteristic with a cocycle in the first cohomology}
\label{sec:Twisting_the_L2-Euler_characteristic_with_a_cocycle_in_the_first_cohomology}

In this section we introduce our main invariant on the $L^2$-side, namely certain
variations of the $L^2$-Euler characteristic which are obtained in the special case of the
universal covering $\widetilde{X} \to X$  by twisting with an element $\phi \in H^1(X;\IZ)$. More generally,
we will consider $G$-$CW$-complexes and twist with a group homomorphism $\phi \colon G \to \IZ$.


\subsection{Review of the $L^2$-Euler characteristic}
\label{subsec:Review_of_the_L2-Euler_characteristic}

Let $G$ be a group. Denote by $\caln(G)$ the group von Neumann algebra which can be
identified with the algebra $B(L^2(G),L^2(G)^G)$ of bounded left $G$-equivariant operators
$L^2(G) \to L^2(G)$. Let $C_*$ be a finitely generated based free left $\IZ G$-chain
complex. Then we can consider the chain complex of finitely generated Hilbert
$\caln(G)$-chain complexes $L^2(G) \otimes_{\IZ G} C_*$. Its $L^2$-Betti numbers
$b_n^{(2)}(L^2(G) \otimes_{\IZ G} C_*)$ are defined as the von Neumann dimension
of its $L^2$-homology, see~\cite[Section~1.1]{Lueck(2002)}.

One can also work entirely algebraically by applying $\caln(G) \otimes_{\IZ G} -$ which
yields a chain complex $\caln(G) \otimes_{\IZ G} C_*$ of $\caln(G)$-modules, where we
consider $\caln(G)$ as a ring and forget the topology. There is a dimension function
defined for all $\caln(G)$-modules, see~\cite[Section~6.1]{Lueck(2002)}. So one gets
another definition of $L^2$-Betti numbers by taking this dimension of the
$\caln(G)$-module $H_n(\caln(G) \otimes_{\IZ G} C_*)$. These two definitions agree
by~\cite[Section~6.2]{Lueck(2002)}.

The advantage of the algebraic approach is that it works and often still gives finite
$L^2$-Betti numbers also in the case where we drop the condition that $C_*$ is a finitely
generated free $\IZ G$-chain complex and consider any $\IZ G$-chain complex $C_*$. This is
explained in detail in~\cite[Chapter~6]{Lueck(2002)}. We recall that for any
chain complex of free left $\IZ G$-chain modules $C_*$ we can define its $n$th $L^2$-Betti number as
\begin{eqnarray}
b_n^{(2)}(\caln(G) \otimes_{\IZ G} C_*)
& := &
 \dim_{\caln(G)}\bigl(H_n(\caln(G) \otimes_{\IZ G} C_*)\bigr)
 \quad \in [0,\infty].
\label{b_n(2)(caln(G)_otimes_ZGC_ast)}
\end{eqnarray}
In particular given a $G$-$CW$-complex $X$ we can define its $n$th $L^2$-Betti number as 
\begin{eqnarray}
 b_n^{(2)}(X;\caln(G)) & := & 
 \dim_{\caln(G)}\bigl(H_n(\caln(G) \otimes_{\IZ G} C_*(X))\bigr)
 \quad \in [0,\infty].
\label{b_n(2)(X;caln(G))}
\end{eqnarray}

We leave the superscript in the notation $b_n^{(2)}(\caln(G) \otimes_{\IZ G} C_*)$ and $b_n^{(2)}(X;\caln(G))$, although the definition is
purely algebraic, in order to remind the reader that it is related to the classical notion
of $L^2$-Betti numbers.

\begin{definition}[$L^2$-Euler characteristic]\label{def:L2-Euler_characteristic}
 Let $X$ be a $G$-$CW$-complex. Define
 \begin{eqnarray*}
 h^{(2)}(X;\caln(G))
 & := & \lmsum{p \ge 0}{} b_p^{(2)}(X;\caln(G)) \in [0,\infty];
 \\
 \chi^{(2)}(X;\caln(G))
 & := & 
 \lmsum{p \ge 0}{} (-1)^p \cdot b_p^{(2)}(X;\caln(G))
 \in \IR, \quad \mbox{if}\; h^{(2)}(X;\caln(G)) < \infty.
 \\
 \end{eqnarray*}
 We call $\chi^{(2)}(X;\caln(G))$ the \emph{$L^2$-Euler characteristic} of $X$.
\end{definition}

The condition $h^{(2)}(X;\caln(G)) < \infty$ ensures that the sum which appears in the
definition of $\chi^{(2)}(X;\caln(G))$ converges absolutely. In the sequel we assume that
the reader is familiar with the notion of the $L^2$-Euler characteristic and its basic
properties, as presented in~\cite[Section~6.6.1]{Lueck(2002)}. Another approach to
$L^2$-Betti numbers for not necessarily finite $G$-$CW$-complexes is given by
Cheeger-Gromov~\cite{Cheeger-Gromov(1986)}.


\subsection{The $\phi$-twisted $L^2$-Euler characteristic}
\label{subsec:The_phi-twisted-L2-Euler_characteristic}

We will be interested in the following version of an $L^2$-Euler characteristic.

\begin{definition}[$\phi$-twisted $L^2$-Betti number and $L^2$-Euler characteristic]%
\label{def:phi-twisted_L2_betti_number_and_L2-Euler_characteristic}
 Let $X$ be a $G$-$CW$-complex. Let $\phi \colon G \to \IZ$ be a group homomorphism. Let
 $\phi^*\IZ[\IZ]$ be the $\IZ G$-module obtained from $\IZ[\IZ]$ regarded as module over
 itself by restriction with $\phi$. If $C_*(X)$ is the cellular $\IZ G$-chain complex,
 denote by $C_*(X) \otimes_{\IZ} \phi^*\IZ[\IZ]$ the $\IZ G$-chain complex obtained by
 the diagonal $G$-action. Here the diagonal $G$-action is defined as the action that is determined by $g\cdot (\sigma \otimes p)=g\cdot \sigma\otimes g\cdot p$. Define
 \begin{eqnarray*}
 b_n^{(2)}(X;\caln(G),\phi) & := &
 \dim_{\caln(G)}\bigl(H_n(\caln(G) \otimes_{\IZ G}(C_*(X) \otimes_{\IZ} \phi^*\IZ[\IZ]))\bigr)
 \in [0,\infty];
 \\
 h^{(2)}(X;\caln(G),\phi)
 & := & \lmsum{p \ge 0}{} b_p^{(2)}(X;\caln(G),\phi) \in [0,\infty];
 \\
 \chi^{(2)}(X;\caln(G);\phi)
 & := & 
 \lmsum{p \ge 0}{} (-1)^p \cdot b_p^{(2)}(X;\caln(G),\phi)
 \in \IR, \hspace{2mm} \mbox{if} \; h^{(2)}(X;\caln(G),\phi) < \infty.
 \\
 \end{eqnarray*}
 We say that $X$ is \emph{$\phi$-$L^2$-finite} if $h^{(2)}(X;\caln(G),\phi) < \infty$ holds. If this
 the case, we call the real number $\chi^{(2)}(X;\caln(G),\phi)$ the 
 \emph{$\phi$-twisted $L^2$-Euler characteristic} of $X$.
 \end{definition}

Notice that so far we are not requiring that the $G$-$CW$-complex $X$ is free or finite.

We collect the basic properties of this invariant.

\begin{theorem}[Basic properties of the $\phi$-twisted $L^2$-Euler characteristic]%
\label{the:Basic_properties_of_the_phi_L2-Euler_characteristic}
 Let $X$ be a $G$-$CW$-complex. Let $\phi \colon G \to \IZ$ be a group homomorphism.
 
 \begin{enumerate}[font=\normalfont]

 \item\label{the:Basic_properties_of_the_phi_L2-Euler_characteristic:homotopy_invariance}
 \emph{$G$-homotopy invariance}\\
 Let $X$ and $Y$ be $G$-$CW$-complexes which are $G$-homotopy equivalent.
Then $X$ is $\phi$-$L^2$-finite if and only if $Y$ is $\phi$-$L^2$-finite, and in this case
 we get
 \[
 \chi^{(2)}(X;\caln(G),\phi) = \chi^{(2)}(Y;\caln(G),\phi);
 \]

 \item\label{the:Basic_properties_of_the_phi_L2-Euler_characteristic:sum_formula}
 \emph{Sum formula}\\
 Consider a $G$-pushout of $G$-$CW$-complexes
 \[
 \xymatrix@C1.2cm@R0.5cm{ X_0 \ar[r] \ar[d] & X_1 \ar[d]
 \\
 X_2 \ar[r]& X }
 \]
 where the upper horizontal arrow is cellular, the left vertical arrow is an inclusion
 of $G$-$CW$-complexes and $X$ has the obvious $G$-$CW$-structure coming from the ones
 on $X_0$, $X_1$ and $X_2$. Suppose that $X_0$, $X_1$ and $X_2$ are $\phi$-$L^2$-finite.
 Then $X$ is $\phi$-$L^2$-finite and we get
 \[
 \chi^{(2)}(X;\caln(G),\phi) = \chi^{(2)}(X_1;\caln(G),\phi) + \chi^{(2)}(X_2;\caln(G),\phi) -
 \chi^{(2)}(X_0;\caln(G),\phi);
 \]

 \item\label{the:Basic_properties_of_the_phi_L2-Euler_characteristic:induction}
 \emph{Induction}\\ 
 Let $i \colon H \to G$ be the inclusion of a subgroup of $G$ and let $Y$ be an $H$-CW-complex
 Then $Y$ is $(\phi \circ i)$-$L^2$-finite if and only if $G \times_HY$ is
 $\phi$-$L^2$-finite. If this is the case, we get
 \[
 \chi^{(2)}(G \times_H Y;\caln(G),\phi) = \chi^{(2)}(Y;\caln(H),\phi \circ i);
 \]

 \item\label{the:Basic_properties_of_the_phi_L2-Euler_characteristic:restriction}
 \emph{Restriction}\\
 Let $i\colon H \to G$ be the inclusion of a subgroup $H$ of $G$ with $[G:H] < \infty$.
 Let $X$ be a $G$-$CW$-complex. Denote by $i^*X$ the $H$-$CW$-complex obtained from
 $X$ by restriction with $i$. 
 Then $i^* X$ is $(\phi \circ i)$-$L^2$-finite if and only if $X$ is $\phi$-$L^2$-finite, and in
 this case we get
 \[
 \chi^{(2)}(i^*X;\caln(H),\phi \circ i) = [G:H]\cdot \chi^{(2)}(X;\caln(G),\phi);
 \]

 \item\label{the:Basic_properties_of_the_phi_L2-Euler_characteristic:scaling_phi} \emph{Scaling $\phi$}\\
 We get for every integer $k \ge 1$ that $X$ is $\phi$-$L^2$-finite if and only if $X$
 is $(k \cdot \phi)$-$L^2$-finite, and in this case we get
 \[
 \chi^{(2)}(X;\caln(G),k \cdot \phi) = k\cdot \chi^{(2)}(X;\caln(G),\phi);
 \]

 \item\label{the:Basic_properties_of_the_phi_L2-Euler_characteristic:trivial_phi} \emph{Trivial $\phi$}\\
 Suppose that $\phi$ is trivial. Then $X$ is $\phi$-$L^2$-finite if and only if we have 
 $b_n^{(2)}(X;\caln(G)) = 0$ for all $n \ge 0$. If this is the case, then
 \[
 \chi^{(2)}(X;\caln(G),\phi) = 0.
 \]

 \end{enumerate}
\end{theorem}
\begin{proof}~\eqref{the:Basic_properties_of_the_phi_L2-Euler_characteristic:homotopy_invariance}
 This follows from the homotopy invariance of $L^2$-Betti numbers.
 \\[1mm]~\eqref{the:Basic_properties_of_the_phi_L2-Euler_characteristic:sum_formula} The
 proof is analogous to the one of~\cite[Theorem~6.80~(2) on page~277]{Lueck(2002)}, just
 replace the short split exact sequence of $\IZ G$-chain complexes $0 \to C_*(X_0) \to
 C_*(X_1) \oplus C_*(X_2) \to C_*(X) \to 0$ by the induced short split exact
 sequence of $\IZ G$-chain complexes
 \begin{multline*}
 0 \to C_*(X_0) \otimes_{\IZ} \phi^*\IZ[\IZ] 
 \to C_*(X_1) \otimes_{\IZ} \phi^*\IZ[\IZ] \oplus C_*(X_2) \otimes_{\IZ} \phi^*\IZ[\IZ] 
 \\
 \to C_*(X) \otimes_{\IZ} \phi^*\IZ[\IZ] \to 0.
 \end{multline*}
 \\[1mm]~\eqref{the:Basic_properties_of_the_phi_L2-Euler_characteristic:induction} The
 proof is analogous to the one of~\cite[Theorem~6.80~(8) on page~279]{Lueck(2002)} using
 the isomorphism of $\IZ G$-chain complexes
 \begin{multline*}
 C_*(G \times_H X) \otimes_{\IZ} \phi^*\IZ[\IZ] \cong \bigl(\IZ G \otimes_{\IZ H} C_*(X)\bigr) \otimes_{\IZ} \phi^*\IZ[\IZ]
 \\
 \cong \IZ G \otimes_{\IZ H}\bigl(C_*(X) \otimes_{\IZ} (\phi \circ i)^*\IZ[\IZ]\bigr),
 \end{multline*}
 where the second isomorphism is given by $(g \otimes u) \otimes v \mapsto g \otimes u \otimes g^{-1}v$.
 \\[1mm]~\eqref{the:Basic_properties_of_the_phi_L2-Euler_characteristic:restriction} The
 proof is analogous to the one of~\cite[Theorem~6.80~(7) on page~279]{Lueck(2002)} using
 the obvious identification of $\IZ H$-chain complexes $i^*\bigl(C_*(X) \otimes_{\IZ} \phi^*\IZ[\IZ]\bigr)
 = C_*(i^*X) \otimes_{\IZ} i^*\phi^*\IZ[\IZ]$.
 \\[1mm]~\eqref{the:Basic_properties_of_the_phi_L2-Euler_characteristic:scaling_phi}
 Since there is an obvious isomorphism of $\IZ G$-modules
 $(k \cdot \phi)^* \IZ[\IZ] \cong \bigoplus_{i = 1}^k \phi^* \IZ[\IZ]$,
 we get
\begin{multline*}
b_n^{(2)}\bigl(\caln(G) \otimes_{\IZ G} (C_*(X) \otimes_{\IZ} (k \cdot \phi)^*\IZ[\IZ])\bigr)
\\
= k \cdot b_n^{(2)}\bigl(\caln(G) \otimes_{\IZ G} (C_*(X) \otimes_{\IZ} \phi^* \IZ[\IZ])\bigr).
\end{multline*}
\\[1mm]~\eqref{the:Basic_properties_of_the_phi_L2-Euler_characteristic:trivial_phi} 
Since the triviality of $\phi$ implies
that $C_*(X) \otimes_{\IZ} \phi^*\IZ[\IZ]$ is $\IZ G$-isomorphic to $\bigoplus_{\IZ} C_*(X)$, we get
\begin{eqnarray*}
b_n^{(2)}\bigl(\caln(G) \otimes_{\IZ G} (C_*(X) \otimes_{\IZ} \phi^*\IZ[\IZ])\bigr)
 = 
\begin{cases}
0 & \text{if} \; b_n^{(2)}\bigl(\caln(G) \otimes_{\IZ G} C_*(X)\bigr) = 0;
\\
\infty & \text{otherwise.}
\end{cases}
\end{eqnarray*}
This finishes the proof of Theorem~\ref{the:Basic_properties_of_the_phi_L2-Euler_characteristic}.
 \end{proof}

We can interpret the $\phi$-twisted $L^2$-Euler characteristic also as an $L^2$-Euler characteristic 
for surjective $\phi$ as follows.

\begin{lemma}\label{lem:phi-twisted_as_ordinary_L2_Euler_characteristic}
 Let $X$ be a $G$-$CW$-complex. Let $\phi \colon G \to \IZ$ be a surjective group homomorphism.
 Denote by $K$ the kernel of $\phi$ and by $i \colon K \to G$ the inclusion. 
 
 Then $X$ is $\phi$-$L^2$-finite if and only if $b_n^{(2)}(i^*X;\caln(K)) < \infty$
 holds for all $n\in \IN_0$. If this is the case, then
 \[
 \chi^{(2)}(X;\caln(G),\phi)\,\, =\,\, \chi^{(2)}(i^*X;\caln(K)).
 \]

\end{lemma}
\begin{proof} 

We have the isomorphism of $\IZ G$-chain complexes
\[
\IZ G \otimes_{\IZ K} i^*C_*(X) \xrightarrow{\cong} C_*(X) \otimes_{\IZ} \phi^*\IZ[\IZ],
\quad g \otimes x \mapsto gx \otimes \phi(g).
\]
The inverse sends $y\otimes q$ to $g \otimes g^{-1} y$ for any choice of $g \in
\phi^{-1}(q)$. Since $\caln(G)$ is flat as an $\caln(K)$-module
by~\cite[Theorem~6.29~(1) on page~253]{Lueck(2002)}, we obtain a sequence of
obvious isomorphisms of $\caln(G)$-modules
\[
\begin{array}{ll}
\phantom{\cong }\caln(G) \otimes_{\caln(K)} H_n(\caln(K) \otimes_{\IZ K} C_*(i^*X))
\hspace{-0.3cm}& \cong
\caln(G) \otimes_{\caln(K)} H_n(\caln(K) \otimes_{\IZ K} i^*C_*(X))
\\
 \cong 
H_n\bigl(\caln(G) \otimes_{\caln(K)} \caln(K) \otimes_{\IZ K} i^*C_*(X)\bigr)

\hspace{-0.3cm}& \cong
H_n\bigl(\caln(G) \otimes_{\IZ K} i^*C_*(X)\bigr)
\\
 \cong 
H_n\bigl(\caln(G) \otimes_{\IZ G} \IZ G \otimes_{\IZ K} i^*C_*(X)\bigr)

\hspace{-0.3cm}& \cong 
H_n\bigl(\caln(G) \otimes_{\IZ G} (C_*(X) \otimes_{\IZ} \phi^*\IZ[\IZ])\bigr).
\end{array}\]
Since $\dim_{\caln(G)}(\caln(G) \otimes_{\caln(K)} M) = \dim_{\caln(K)}(M)$ holds for
every $\caln(K)$-module $M$ by~\cite[Theorem~6.29~(2) on page~253]{Lueck(2002)}, we
conclude for every $n \ge 0$
\[
b_n^{(2)}\bigl(\caln(K) \otimes_{\IZ K} C_*(i^*X);\caln(K)\bigr)
=
b_n^{(2)}\bigl(\caln(G) \otimes_{\IZ G} (C_*(X) \otimes_{\IZ} \phi^*\IZ[\IZ]);\caln(G)\bigr).
\]
\end{proof}


\subsection{The $(\mu,\phi)$-$L^2$-Euler characteristic}
\label{subsec:The_(mu,phi)-L2-Euler_characteristic}

 We will be interested in this paper mainly in the case where the $G$-$CW$-complex $X$ is
free. If we put $Y = X/G$, then $X$ is the disjoint union of the preimages of the
components of $Y$. Therefore it suffices to study a connected $CW$-complex $Y$ and
$G$-coverings $\overline{Y} \to Y$. Any such $G$-covering is obtained from the universal
covering $\widetilde{Y} \to Y$ and a group homomorphism $\mu \colon \pi = \pi_1(Y) \to G$
as the projection $G \times_{\mu} \widetilde{Y} \to Y$. Therefore we introduce the
following notation:

 \begin{definition}[$(\mu,\phi)$-$L^2$-Euler characteristic]%
\label{def:mu-phi-L2-Euler_characteristic}
 Let $X$ be a connected $CW$-complex. Let $\mu \colon \pi_1(X) \to G$ and $\phi \colon G \to \IZ$ be a group homomorphisms.
 Let $\overline{X} \to X$ be the $G$-covering associated to $\mu$. We call $X$ \emph{$(\mu,\phi)$-$L^2$-finite}
 if $\overline{X}$ is $\phi$-$L^2$-finite, and in this case we define the \emph{$(\mu,\phi)$-$L^2$-Euler characteristic}
 $\chi^{(2)}(X;\mu,\phi)$ to be $\chi^{(2)}(\overline{X};\caln(G);\phi)$, see 
 Definition~\ref{def:phi-twisted_L2_betti_number_and_L2-Euler_characteristic}.
 \end{definition}

 The next lemma essentially reduces the general case $(\mu,\phi)$ to the special case,
 where $\mu$ and $\phi$ are surjective or $\phi \circ \mu$ is trivial.

 \begin{lemma}\label{lem:reduction_to_surjective_mu} Let $X$ be a connected
 $CW$-complex. Let $\mu \colon \pi_1(X) \to G$ and $\phi \colon G \to \IZ$ be group
 homomorphisms. Let $G'$ be the image of $\mu$. Let $\mu' \colon \pi_1(X) \to G'$ be the
 epimorphism induced by $\mu$ and let $\phi' \colon G' \to \IZ$ be obtained by
 restricting $\phi$ to $G'$.

 \begin{enumerate}[font=\normalfont]

 \item\label{lem:reduction_to_surjective_mu:mu'} 
 Then $X$ is $(\mu,\phi)$-$L^2$-finite if and only if $X$ is $(\mu',\phi')$-$L^2$-finite.
 If this is the case, we get 
 \[\chi^{(2)}(X;\mu,\phi)\,\, =\,\, \chi^{(2)}(X;\mu',\phi');\]

 \item\label{lem:reduction_to_surjective_mu:mu_circ_phi_is_not_trivial} 

 Suppose that $\phi\circ \mu  \not= 0$. Let
 $k\ge 1 $ be the natural number such that the image of $\phi'$ is $k \cdot \IZ$ and let $\phi''
 \colon G' \to \IZ$ be the epimorphism uniquely determined by $k \cdot \phi'' = \phi'$. 
 Then $X$ is $(\mu,\phi)$-$L^2$-finite, if and only if $X$ is
 $(\mu',\phi'')$-$L^2$-finite. If this is the case, we get
 \[
 \chi^{(2)}(X;\mu,\phi) \,\,=\,\, \smfrac{1}{k} \cdot \chi^{(2)}(X;\mu',\phi'');
 \]
 \item\label{lem:reduction_to_surjective_mu:mu_circ_phi_is_trivial} 
 Suppose that $\phi \circ \mu = 0$. Then
 $X$ is $(\mu,\phi)$-$L^2$-finite if and only if $b_n^{(2)}(\overline{X};\caln(G))$ vanishes 
 for the $G$-covering $\overline{X}$ associated to
 $\mu$ and every $n \ge 0$. If this is the case, then
 \[\chi^{(2)}(X;\mu,\phi)\,\, =\,\, 0.
 \]
\item Let $p\colon Z\to X$ be a finite $d$-sheeted covering. Then $Z$ is
 $(p^*\mu,p^*\phi)$-$L^2$ finite if and only if $X$ is $(\mu,\phi)$-$L^2$ finite and 
\[ \chi^{(2)}(Z;p^*\mu,p^*\phi)\,\, =\,\,d \cdot \chi^{(2)}(X;\mu,\phi)\,\, =\,\, 0.\]
 \end{enumerate}
 \end{lemma}
 
 \begin{proof} The first statement follows from Theorem~\ref{the:Basic_properties_of_the_phi_L2-Euler_characteristic} (3).
 The second and third statement follow from (1) and Theorem~\ref{the:Basic_properties_of_the_phi_L2-Euler_characteristic} (5) and (6).
 The last statement follows from 
 Theorem~\ref{the:Basic_properties_of_the_phi_L2-Euler_characteristic} (4).
\end{proof}

\begin{example}[Mapping torus]\label{exa_mapping_torus} Let $Y$ be a connected finite
 $CW$-complex and $f \colon Y \to Y$ be a self-map. Let $T_f$ be its
 mapping torus. Consider any factorization $\pi_1(T_f) \xrightarrow{\mu} G
 \xrightarrow{\phi} \IZ$ of the epimorphism $\pi_1(T_f) \to \pi_1(S^1) = \IZ$ induced by
 the obvious projection $T_f \to S^1$. Then $T_f$ is $(\mu,\phi)$-$L^2$-finite and we get
 \[
 \chi^{(2)}(T_f;\mu,\phi) \,\,= \,\,\chi(Y)
 \]
 by the following argument. Let $\overline{T_f}$ be the $G$-covering
 associated to $\mu \colon \pi_1(T_f) \to G$. Let $K$ be the kernel of 
 $\phi$ and $i \colon K \to G$ be the inclusion. The
 image of the composite $\pi_1(Y) \to \pi_1(T_f) \xrightarrow{\mu} G$ is contained in
 $K$ and we can consider the $K$-covering $\widehat{Y} \to Y$ associated to it. The
 $K$-$CW$-complex $i^*\overline{T_f}$ is $K$-homotopy equivalent to the $K$-$CW$-complex
 $\widehat{Y}$, see~\cite[Section~2]{Lueck(1994b)}. Hence we conclude 
 $\chi^{(2)}(T_f;\mu,\phi) = \chi^{(2)}(\widehat{Y};\caln(K))$ 
 from Lemma~\ref{lem:phi-twisted_as_ordinary_L2_Euler_characteristic}
 and the $K$-homotopy invariance of $L^2$-Betti numbers. Since $Y$ is a finite 
 $CW$-complex, we have $\chi^{(2)}(\widehat{Y};\caln(K)) = \chi(Y)$. 

 Notice that the $L^2$-Betti numbers $b_n^{(2)}(\overline{T_f};\caln(G))$ are all trivial
 by~\cite[Theorem~2.1]{Lueck(1994b)} and hence the $L^2$-Euler characteristic
 $\chi^{(2)}(\overline{T_f};\caln(G))$ is trivial. So the passage to the subgroup of
 infinite index $K$ or, equivalently, the twisting with the $\IZ G$-module $\phi^*\IZ[\IZ]$ 
 which is not finitely generated as an abelian group,
 ensures that we get an interesting invariant by the
 $(\mu,\phi)$-$L^2$-Euler characteristic.
 \end{example}

\begin{lemma}\label{lem:tori} 
 Let $T^n$ be the $n$-dimensional torus for $n \ge 1$. Consider homomorphisms $\mu \colon
 \pi_1(T^n) \to G$ and $\phi \colon G \to \IZ$ such that the image of $\mu$ is infinite.
 
Then $T^n$ is $(\mu,\phi)$-$L^2$-finite and we get
\[
\chi^{(2)}(T^n;\mu,\phi) = 
\begin{cases} [\IZ : \im(\phi \circ \mu)] 
& 
\text{if}\; n = 1, \phi \circ \mu \not = 0;
\\
0 
& 
\text{otherwise.}
\end{cases}
\]
\end{lemma}
\begin{proof}
 Because of
 Lemma~\ref{lem:reduction_to_surjective_mu}~\eqref{lem:reduction_to_surjective_mu:mu'} we
 can assume without loss of generality that $\mu$ is surjective. Suppose that $\phi \circ \mu$ is
 non-trivial. Because of
 Lemma~\ref{lem:reduction_to_surjective_mu}~\eqref{lem:reduction_to_surjective_mu:mu_circ_phi_is_not_trivial}
 it suffices to consider the case where $\mu$ and $\phi$ are surjective.
 Then the claim follows from Example~\ref{exa_mapping_torus} since there is a
 homeomorphism $h \colon T^n \xrightarrow{\cong} T^{n-1} \times S^1$ such that the composite of the map
 of $\pi_1(T^{n-1} \times S^1) \to \pi_1(S^1)$ induced by the projection onto $S^1$
 composed with $\pi_1(h)$ is $\phi$. Suppose that $\phi$ is trivial. Since $\mu$ has infinite image,
 one can show using~\cite[Theorem~2.1]{Lueck(1994b)} that 
 $b_m^{(2)}(\overline{T^n};\caln(G))$ vanishes for the $G$-covering
 $\overline{T^n} \to T^n$ associated to $\mu$ for all $m \ge 0$. Hence the claim follows from
 Lemma~\ref{lem:reduction_to_surjective_mu}~\eqref{lem:reduction_to_surjective_mu:mu_circ_phi_is_trivial}.
\end{proof}

\begin{theorem}[The $(\mu,\phi)$-$L^2$-Euler characteristic for $S^1$-$CW$-complexes] 
\label{the:The_(mu,phi)-L2-Euler_characteristic_for_S1-CW-complexes}
Let $X$ be a connected finite $S^1$-$CW$-complex. Let $\mu \colon \pi_1(X) \to G$ 
and $\phi \colon G \to \IZ$ be group homomorphisms. Suppose that for one and hence all $x \in X$ the composite
\[
\eta \colon \pi_1(S^1,1) \xrightarrow{\pi_1(\ev_x,1)} \pi_1(X,x) \xrightarrow{\mu} G \xrightarrow{\phi} \IZ
\]
is injective, where $\ev_x \colon S^1 \to X$ sends $z$ to $z \cdot x$. 
Define the $S^1$-orbifold Euler characteristic of $X$ by
\[
\chi^{S^1}_{\orb}(X)\,\,=\,\, \smsum{n \ge 0}{} (-1)^n \cdot \smsum{e \in I_n}{} \smfrac{1}{|S^1_e|},
\]
where $I_n$ is the set of open $n$-dimensional $S^1$-cells of $X$ and for 
$e \in I_n $ we denote by $S^1_e$ the isotropy group of any point in $e$. 
Then $X$ is $(\mu,\phi)$-$L^2$-finite and we get
\[
\chi^{(2)}(X;\mu,\phi)\,\, =\,\, 
\chi^{S^1}_{\orb}(X) \cdot [\IZ : \im(\eta)].
\]
\end{theorem}
\begin{proof}
 The strategy of proof is the same as the the one of~\cite[Theorem~1.40 on page~43]{Lueck(2002)}, 
 where one considers more general finite $S^1$-$CW$-complex $Y$ together with a
 $S^1$-map $f \colon Y \to X$ and does induction over the number of $S^1$-equivariant cells.
 A basic ingredient is the additivity of the two terms appearing in the desired equation in 
 Theorem~\ref{the:The_(mu,phi)-L2-Euler_characteristic_for_S1-CW-complexes}.
\end{proof}

We mention that the condition about the injectivity of the map $\eta$ appearing in
 Theorem~\ref{the:The_(mu,phi)-L2-Euler_characteristic_for_S1-CW-complexes} is necessary.

For the reader's convenience we record the next result 
which we will not need in this paper and whose proof is a variation of the one of
Theorem~\ref{the:The_(mu,phi)-L2-Euler_characteristic_for_S1-CW-complexes}.

\begin{theorem}[The $(\mu,\phi)$-$L^2$-Euler characteristic for fibrations]%
\label{the:(mu,phi)-Euler_characteristic_for_fibrations}
 Let $F \xrightarrow{i} E \xrightarrow{p} B$ be a fibration of connected $CW$-complexes. Suppose
 that $B$ is a finite $CW$-complex. Consider group homomorphisms
 $\mu \colon \pi_1(E) \to G$ and $\phi \colon G \to \IZ$. Suppose that
 ${F}$ is $(\mu \circ \pi_1(i),\phi)$-$L^2$-finite.

 Then ${E}$ is $(\mu,\phi)$-$L^2$-finite and we get
\[
\chi^{(2)}(E;\mu,\phi) =
\chi(B) \cdot \chi^{(2)}(F,\mu \circ \pi_1(i), \phi). 
\]
\end{theorem}

If $M$ is a compact connected orientable $3$-manifold with proper $S^1$-action, then $M$
is a compact connected orientable Seifert manifold. The converse is not true in 
general. Associated to a compact connected orientable Seifert manifold is an orbifold $X$
and $X$ has a orbifold Euler characteristic $\chi_{\orb}(X)$. For a basic introduction to
these notions we refer for instance to~\cite{Scott(1983)}. If $M$ is a compact
$3$-manifold with proper $S^1$-action, then $X$ is given by $M/S^1$ and $\chi_{\orb}(X)$
is the $S^1$-orbifold Euler characteristic $\chi^{S^1}_{\orb}(M)$. We omit the proof of
the next result since it is essentially a variation of the one of
Theorem~\ref{the:The_(mu,phi)-L2-Euler_characteristic_for_S1-CW-complexes}, the role of
the cells $S^1/H \times D^n$ in
Theorem~\ref{the:The_(mu,phi)-L2-Euler_characteristic_for_S1-CW-complexes} is now played
by the typical neighborhoods of the Seifert fibers given by solid tori.

\begin{theorem}[The $(\mu,\phi)$-$L^2$-Euler characteristic for Seifert manifolds] 
\label{the:The_(mu,phi)-L2-Euler_characteristic_for_Seifert_manifolds}
Let $M$ be a compact connected orientable Seifert manifold. Let $\mu \colon \pi_1(M) \to G$ 
and $\phi \colon G \to \IZ$ be group homomorphisms. Suppose that for one (and hence all) $x \in M$
\[
\eta \colon \pi_1(S^1,1) \xrightarrow{\pi_1(\ev,1)} \pi_1(M,x) \xrightarrow{\mu} G \xrightarrow{\phi} \IZ
\]
is injective, where $\ev \colon S^1 \to M$ is the inclusion of a regular fiber.
Let $X$ be the associated orbifold of $M$ and denote by $\chi_{\orb}(X)$ its orbifold Euler characteristic.
Then $M$ is $(\mu,\phi)$-$L^2$-finite and we get
\[
\chi^{(2)}(M;\mu,\phi)\,\, = \,\, 
\chi_{\orb}(X) \cdot [\IZ : \im(\eta)].
\]
\end{theorem}

\begin{lemma}\label{lem:JSJ_decom_Euler} Let $M$ be a $3$-manifold,
which is admissible, see Definition~\ref{def:admissible_3-manifold}. Let
 $M_1$, $M_2$, \ldots, $M_r$ be its pieces in the Jaco-Shalen-Johannson decomposition. Consider
 group homomorphisms $\mu \colon \pi_1(M) \to G$ and $\phi \colon G \to \IZ$. Suppose that
 the composite of $\mu$ with $\pi_1(j) \colon \pi_1(T^2) \to \pi_1(M)$ has infinite image for the inclusion
 $j \colon T^2 \to M$ of any splitting torus appearing in the Jaco-Shalen-Johannson decomposition.
 Let $\mu_i \colon \pi_1(M_i) \to G$ be the composite of $\mu$ with the map $\pi_1(M_i) \to \pi_1(M)$
 induced by the inclusion $M_i \to M$. Suppose that $M_i$ is $(\mu_i,\phi)$-$L^2$-finite for $i = 1,2, \ldots , r$. 

 Then $M$ is $(\mu,\phi)$-$L^2$-finite and we have
\[
\chi^{(2)}(M;\mu,\phi)\,\,=\,\, \lmsum{i = 1}{r} \chi^{(2)}\big(\widetilde{M_i};\mu_i,\phi\big).
\]
\end{lemma}
\begin{proof} 
This follows from 
Theorem~\ref{the:Basic_properties_of_the_phi_L2-Euler_characteristic}~%
\eqref{the:Basic_properties_of_the_phi_L2-Euler_characteristic:sum_formula}
and Lemma~\ref{lem:tori}.
\end{proof}

\begin{theorem}[The $(\phi,\mu)$-$L^2$-Euler characteristic and the Thurston norm for graph manifolds]
\label{the:The_(mu,phi-L2-Euler_characteristic_and_the_Thurston_norm_for_graph_manifolds}
Let $M$ be an admissible $3$-manifold, which is a graph manifold and not homeomorphic to $S^1 \times D^2$.
Consider group homomorphisms $\mu \colon \pi_1(M) \to G$ and $\phi \colon G \to \IZ$. 
Suppose that for each piece $M_i$ in the
Jaco-Shalen-Johannson decomposition the map $\pi_1(S^1 ) \xrightarrow{\pi_1(\ev_i)}
\pi_1(M_i) \xrightarrow{j_i} \pi_1(M) \xrightarrow{\mu} G \xrightarrow{\phi} \IZ$ is
injective, where $\ev_i \colon S^1 \to M_i$ is the inclusion of the regular fiber 
and $j_i \colon M_i \to M$ is the inclusion.

Then $M$ is $(\mu,\phi)$-$L^2$-finite and we get
\[
- \chi^{(2)}(M;\mu,\phi) \,\,=\,\, x_M(\phi\circ \mu).
\]
\end{theorem}
\begin{proof} In the situation and notation of Lemma~\ref{lem:JSJ_decom_Euler} 
we conclude from~\cite[Proposition~3.5 on page~33]{Eisenbud-Neumann(1985)}
\[
x_M(\phi)\,\, =\,\, \lmsum{i = 1}{r} x_{M_i}(\phi_i)
\]
if $\phi_i \in H^1(M_i;\IZ)$ is the restriction of $\phi$ to $M_i$.
Moreover, we get from Theorem~\ref{the:The_(mu,phi)-L2-Euler_characteristic_for_Seifert_manifolds}
and from~\cite[Lemma~A]{Herrmann(2016)}
for $i = 1,2 \ldots, r$ 
\[
\chi^{(2)}(M_i;\mu_i,\phi) \,\,=\,\, -x_M(\phi_i),
\]
if $\mu_i$ is the composite of $\mu$ with the homomorphism $\pi_1(M_i) \to \pi_1(M)$ induced
by the inclusion. 
Now the claim follows from Lemma~\ref{lem:JSJ_decom_Euler}
since any splitting torus appearing in the Jaco-Shalen-Johannson decomposition contains a regular fiber
of one of the pieces $M_i$.
\end{proof}


\subsection{The $\phi$-$L^2$-Euler characteristic for universal coverings}
\label{subsec:The_phi-L2-Euler_characteristic_for_universal_coverings}

In this section we consider the special case of the universal covering and of a group
homomorphism $\phi \colon \pi_1(X) \to \IZ$. This is in some sense the most canonical and
important covering and in this case the formulations of the main results simplify in a
convenient way.

\begin{definition}[The $\phi$-$L^2$-Euler characteristic for $\phi \in H^1(X;\IZ)$]
\label{def:The_phi-L2-Euler_characteristic_for_phi_in_H1(X;Z)}
Let $X$ be a connected $CW$-complex with fundamental group $\pi$.
Let $\phi$ be an element in $H^1(X;\IZ)$, or, equivalently, 
let $\phi \colon \pi \to \IZ$ be a group homomorphism. We say that the universal covering
 $\widetilde{X}$ of $X$ is \emph{$\phi$-$L^2$-finite}, if $X$ is $(\id_{\pi},\phi)$-$L^2$-finite 
 in the sense of Definition~\ref{def:mu-phi-L2-Euler_characteristic}.
 If this is the case, we define its \emph{$\phi$-$L^2$-Euler characteristic}
 \[
 \chi^{(2)}(\widetilde{X};\phi) 
 \,\, := \,\,
 \chi^{(2)}(X;\id_{\pi},\phi)
 \]
 where $\chi^{(2)}(X;\id_{\pi},\phi)$ has been introduced in Definition~\ref{def:mu-phi-L2-Euler_characteristic}.
 
 If $X$ is a (not necessarily connected) finite $CW$-complex and $\phi \in H^1(X;\IZ)$, we say that
 $\widetilde{X}$ is $\phi$-$L^2$-finite if for each component $C \in X$ the universal
 covering $\widetilde{C} \to C$ is $\phi|_C$-$L^2$-finite and we put
 \[
 \chi^{(2)}(\widetilde{X};\phi) \,\,= \,\,\lmsum{C \in \pi_0(X)}{} \chi^{(2)}(\widetilde{C};\phi|_C).
 \]
\end{definition}

For the reader's convenience we record the basic properties of the $\phi$-$L^2$-Euler characteristic.

\begin{theorem}[Basic properties of the $\phi$-$L^2$-Euler characteristic for universal coverings]
\label{the:Basic_properties_of_the_phi_L2-Euler_characteristic_for_universal_coverings}
\

\begin{enumerate}[font=\normalfont]

\item \emph{Homotopy invariance}\\[1mm]
\label{the:Basic_properties_of_the_phi_L2-Euler_characteristic_for_universal_coverings:homotopy_invariance}
Let $f \colon X \to Y $ be a homotopy equivalence of $CW$-complexes. Consider $\phi \in H^1(Y;\IZ)$.
Let $f^*\phi \in H^1(X;\IZ)$ be its pullback with $f$.

Then $\widetilde{X}$ is $f^*\phi$-$L^2$-finite if and only if $\widetilde{Y}$ is $\phi$-$L^2$-finite, and in this
 case we get
 \[
 \chi^{(2)}(\widetilde{X};f^*\phi)\,\, =\,\, \chi^{(2)}(\widetilde{Y};\phi);
 \]

\item\label{the:Basic_properties_of_the_phi_L2-Euler_characteristic_for_universal_coverings:sum_formula}
 \emph{Sum formula}\\
 Consider a pushout of $CW$-complexes
 \[
 \xymatrix@C1.2cm@R0.5cm{ X_0 \ar[r] \ar[d] & X_1 \ar[d]
 \\
 X_2 \ar[r]& X }
 \]
 where the upper horizontal arrow is cellular, the left vertical arrow is an inclusion of
 $CW$-complexes and $X$ has the obvious $CW$-structure coming from the ones on $X_0$,
 $X_1$ and $X_2$. Consider $\phi \in H^1(X;\IZ)$. For every $i \in \{0,1,2\}$ suppose
 that for each base point $x_i \in X_i$ the map $\pi_1(j_i,x_i) \colon \pi_1(X_i,x_i) \to \pi_1(X,j_i(x_i))$ 
 induced by the inclusion $j_i \colon X_i \to X$ is injective and that
 $\widetilde{X_i}$ is $j_i^*\phi$-$L^2$-finite.

 Then $\widetilde{X}$ is $\phi$-$L^2$-finite and we get
 \[
 \hspace{1cm} \chi^{(2)}(\widetilde{X};\phi) \,\,=\,\, \chi^{(2)}(\widetilde{X_1};j_1^*\phi) + \chi^{(2)}(\widetilde{X_2};j_2^*\phi) -
 \chi^{(2)}(\widetilde{X_0};j_0^*\phi);
 \]
 
 \item\label{the:Basic_properties_of_the_phi_L2-Euler_characteristic_for_universal_coverings:finite_coverings}
 \emph{Finite coverings}\\
 Let $p \colon X \to Y$ be a finite $d$-sheeted covering of connected $CW$-complexes, $\phi$
be an element in $H^1(Y;\IZ)$ and $p^*\phi \in H^1(X;\IZ)$ be its pullback with $p$.

Then $\widetilde{Y}$ is $\phi$-$L^2$-finite if and only if $\widetilde{X}$ is $p^*\phi$-$L^2$-finite, and in this case
\[
\chi^{(2)}(\widetilde{X};p^*\phi) \,\,=\,\, d \cdot \chi^{(2)}(\widetilde{Y};\phi);
\]

\item\label{the:Basic_properties_of_the_phi_L2-Euler_characteristic_for_universal_coverings:scaling_phi}
 \emph{Scaling $\phi$}\\
Let $X$ be a $CW$-complex and $\phi$ be an element in $H^1(X;\IZ)$
Consider any integer $k \not= 0$. 
Then $\widetilde{X}$ is $\phi$-$L^2$-finite if and only if $\widetilde{X}$ is $(k \cdot \phi)$-$L^2$-finite, 
and in this case we get
\[
\chi^{(2)}(\widetilde{X};k \cdot \phi)\,\, =\,\, |k| \cdot \chi^{(2)}(\widetilde{X};\phi).
\]

\item\label{the:Basic_properties_of_the_phi_L2-Euler_characteristic_for_universal_coverings:trivial_phi}
 \emph{Trivial $\phi$}\\
Let $X$ be a $CW$-complex. Let $\phi$ be trivial.
Then $\widetilde{X}$ is $\phi$-$L^2$-finite if and only if $b_n^{(2)}(\widetilde{X};\caln(\pi_1(X))) = 0$
holds for all $n \ge 0$. If this is the case, we get
\[
\chi^{(2)}(\widetilde{X};\phi)\,\, =\,\, 0;
\]

\item\label{the:Basic_properties_of_the_phi_L2-Euler_characteristic_for_universal_coverings:tori}
 \emph{Tori}\\
Let $T^n$ be the $n$-dimensional torus for $n \ge 1$. Consider any $\phi \in H^1(T^n;\IZ)$. Then
$\widetilde{T^n}$ is $\phi$-$L^2$-finite and we get
\[
\chi^{(2)}(\widetilde{T^n};\phi) \,\,=\,\, 
\begin{cases} [\IZ : \im(\phi)] 
& 
\text{if}\; n = 1, \phi \not = 0;
\\
0 
& 
\text{otherwise.}
\end{cases}
\]
\end{enumerate}
\end{theorem}

\begin{proof}
 The second statement follows from
 Theorem~\ref{the:Basic_properties_of_the_phi_L2-Euler_characteristic}~%
 \eqref{the:Basic_properties_of_the_phi_L2-Euler_characteristic:sum_formula}
 and~\eqref{the:Basic_properties_of_the_phi_L2-Euler_characteristic:induction} using the
 fact that for every $i \in \{0,1,2\}$ and every base point $x_i \in X_i$ the total space
 of the pullback of the universal covering of the component $D$ of $X$ containing
 $j_i(x_i)$ with $j_i|_C \colon C \to D$ for $C$ the component of $X_i$ containing $x_i$
 is $\pi_1(D)$-homeomorphic to $\pi_1(D) \times_{\pi_1(j_i)} \widetilde{C}$ for the
 universal covering $\pi_1(C,x_i) \to \widetilde{C} \to C$ of $C$.

 The last statement is a special case of Lemma~\ref{lem:tori}. Finally all other
 statements follow from the corresponding statements
 of~Theorem~\ref{the:Basic_properties_of_the_phi_L2-Euler_characteristic}.
\end{proof}

\begin{lemma}\label{lem:JSJ_decom_Euler_universal} 
 Let $M$ be an admissible $3$-manifold. Let
 $M_1, M_2, \ldots, M_r$ be its pieces in the Jaco-Shalen-Johannson decomposition. Consider
 $\phi \in H^1(M;\IZ)$. Let $\phi_i \in H^1(M_i;\IZ)$ be the pullback of $\phi$ with the
 inclusion $M_i \to M$ for $i = 1,2 , \ldots,r$.

 Then $M_i$ is $\phi_i$-$L^2$-finite for $i = 1,2, \ldots ,r$ and $M$ is
 $\phi$-$L^2$-finite and we have
\[
\chi^{(2)}(\widetilde{M};\phi)\,\, = \,\,\lmsum{i = 1}{r} \chi^{(2)}(\widetilde{M_i};\phi_i).
\]
\end{lemma}
\begin{proof} 
If $M_i$ is Seifert, then it is $\phi_i$-$L^2$-finite by 
Theorem~\ref{the:The_(mu,phi)-L2-Euler_characteristic_for_Seifert_manifolds}.
If $M_i$ is hyperbolic, $b_n^{(2)}(\widetilde{M_i})$ vanishes for all $n \ge 0$
by~\cite[Theorem~0.1]{Lott-Lueck(1995)}, and hence $M_i$ is $\phi_i$-$L^2$-finite by 
Theorem~\ref{the:Status_of_the_Atiyah_Conjecture}~%
~\eqref{the:Status_of_the_Atiyah_Conjecture:3-manifold_not_graph}
and Theorem~\ref{the:Atiyah_and_(mu,phi)-L2-Euler_characteristic}. Now the claim follows from
Theorem~\ref{the:Basic_properties_of_the_phi_L2-Euler_characteristic_for_universal_coverings}~%
\eqref{the:Basic_properties_of_the_phi_L2-Euler_characteristic_for_universal_coverings:sum_formula}
and~\eqref{the:Basic_properties_of_the_phi_L2-Euler_characteristic_for_universal_coverings:tori}
using the fact that the splitting tori in the Jaco-Shalen-Johannson decomposition are incompressible.
\end{proof}

\begin{theorem}[The $\phi$-$L^2$-Euler characteristic and the Thurston norm for graph manifolds]
\label{the:The_phi-L2-Euler_characteristic_and_the_Thurston_norm_for_graph_manifolds}
Let $M$ be an admissible $3$-manifold, which is a graph manifold and not homeomorphic to $S^1 \times D^2$. 
Consider $\phi \in H^1(M;\IZ)$. Then $\widetilde{M}$ is $\phi$-$L^2$-finite and we get
\[
- \chi^{(2)}(\widetilde{M};\phi) \,\,=\,\, x_M(\phi).
\]
\end{theorem}
\begin{proof} This follows from 
Theorem~\ref{the:The_(mu,phi-L2-Euler_characteristic_and_the_Thurston_norm_for_graph_manifolds}.
\end{proof}

\begin{example}[$S^1\times D^2$ and $S^1 \times S^2$]
\label{exa:S1_timesD2_and_S1_times_S2}
Consider a homomorphism $\phi \colon H_1(S^1 \times D^2) \xrightarrow{\cong} \IZ$.
Let $k$ be the index $[\IZ : \im(\phi)]$ if $\phi$ is non-trivial, and let $k = 0$ if $\phi$ is trivial.
Then we conclude from~\eqref{scaling_Thurston_norm},~\eqref{fiber_bundles_Thurston_norm},
 Lemma~\ref{lem:reduction_to_surjective_mu} and Example~\ref{exa_mapping_torus} that
\[
x_{S^1 \times D^2}(\phi) \,\, = \,\, 0\quad \mbox{ and }\quad 
-\chi^{(2)}(\widetilde{S^1 \times D^2};\phi) 
\,\, = \,\, k.\]
Hence we have to exclude $S^1 \times D^2$ in 
Theorem~\ref{the:The_phi-L2-Euler_characteristic_and_the_Thurston_norm_for_graph_manifolds} and thus also in Theorem~\ref{the:Equality_of_(mu,phi)-L2-Euler_characteristic_and_the_Thurston_norm_universal_covering_for_universal_coverings}.
Analogously we get 
\[
x_{S^1 \times S^2}(\phi) \,\, = \,\, 0\quad \mbox{ and } \quad
-\chi^{(2)}(\widetilde{S^1 \times S^2};\phi) 
\,\, = \,\, 2 \cdot k,\]
so that we cannot replace ``irreducible'' by ``prime'' in 
Theorem~\ref{the:Equality_of_(mu,phi)-L2-Euler_characteristic_and_the_Thurston_norm_universal_covering_for_universal_coverings}.
\end{example}


 \typeout{-------------------------- Section 3: About the Atiyah Conjecture -----------------------}

\section{About the Atiyah Conjecture}
\label{sec:About_the_Atiyah_Conjecture}

So far the definition and the analysis of the $\phi$-twisted $L^2$-Euler characteristic
has been performed on an abstract level. In order to ensure that the condition
$(\mu,\phi)$-$L^2$-finite is satisfied and that the $(\mu,\phi)$-$L^2$-Euler
characteristic contains interesting information, we will need further input, namely, the
following Atiyah Conjecture.


\subsection{The Atiyah Conjecture}
\label{subsec:The_Atiyah_Conjecture}

\begin{definition}[Atiyah Conjecture]
\label{def:Atiyah_Conjecture}
We say that a torsion-free group $G$ satisfies the \emph{Atiyah Conjecture} if for any
matrix $A \in M_{m,n}(\IQ G)$ the von Neumann dimension $\dim_{\caln(G)}(\ker(r_A))$
of the kernel of the $\caln(G)$-homomorphism
$r_A\colon \caln(G)^m \to \caln(G)^n$ given by right multiplication with $A$ is an integer.
\end{definition}

The Atiyah Conjecture can also be formulated for any field $F$ with $\IQ \subseteq F \subseteq \IC$
and matrices $A \in M_{m,n}(FG)$ and for any group with a bound on the order of its finite subgroups. 
However, we only need and therefore consider in this paper the case, where $F = \IQ$ and $G$ is torsion-free.

\begin{theorem}[Status of the Atiyah Conjecture]
\label{the:Status_of_the_Atiyah_Conjecture}
\begin{enumerate}[font=\normalfont]

\item\label{the:Status_of_the_Atiyah_Conjecture:subgroups}
If the torsion-free group $G$ satisfies the Atiyah Conjecture,
then also each of its subgroups satisfies the Atiyah Conjecture;

\item\label{the:Status_of_the_Atiyah_Conjecture:Linnell} Let $\calc$ be the smallest
 class of groups which contains all free groups and is closed under directed unions and
 extensions with elementary amenable quotients. Suppose that $G$ is a torsion-free group
 which belongs to $\calc$.

Then $G$ satisfies the Atiyah Conjecture;

\item\label{the:Status_of_the_Atiyah_Conjecture:3-manifold_not_graph} If $G$ is an infinite group which is the
 fundamental group of an admissible $3$-manifold $M$,
 then $G$ is torsion-free and belongs to $\calc$. In particular $G$ satisfies the Atiyah Conjecture.

\item\label{the:Status_of_the_Atiyah_Conjecture:approx}
Let $\cald$ be the smallest class of groups such that
\begin{itemize}

\item The trivial group belongs to $\cald$;

\item If $p\colon G \to A$ is an epimorphism of a torsion-free group $G$ onto an
elementary amenable group $A$ and if $p^{-1}(B) \in \cald$ for every finite group
$B \subset A$, then $G \in \cald$;

\item $\cald$ is closed under taking subgroups;

\item $\cald$ is closed under colimits and inverse limits over directed systems.

\end{itemize}

If the group $G$ belongs to $\cald$,
then $G$ is torsion-free and the Atiyah Conjecture holds for $G$.

The class $\cald$ is closed under direct sums, direct products and free
products. Every residually torsion-free elementary amenable group
belongs to $\cald$.
\end{enumerate}
\end{theorem}

\begin{proof}~\eqref{the:Status_of_the_Atiyah_Conjecture:subgroups} 
 This follows from~\cite[Theorem~6.29~(2) on page~253]{Lueck(2002)}.
 \\[1mm]~\eqref{the:Status_of_the_Atiyah_Conjecture:Linnell} This is due to Linnell,
 see for instance~\cite{Linnell(1993)} or~\cite[Theorem~10.19 on page~378]{Lueck(2002)}.
 \\[1mm]~\eqref{the:Status_of_the_Atiyah_Conjecture:3-manifold_not_graph} 
 We know from \cite[(C.3)]{AFW(2015)} that $G$ is torsion-free. Thus it suffices to
 show that $G = \pi_1(M)$ belongs to the class $\calc$ appearing in
 assertion~\eqref{the:Status_of_the_Atiyah_Conjecture:Linnell}. 
 
 We do this in several steps: 
 \begin{enumerate}
 \item[(a)] First suppose that $G$ is not a closed graph manifold.  By the proof of the Virtual Fibering Theorem due to Agol, Liu, Przytycki-Wise, and 
 Wise~\cite{Agol(2008),Agol(2013),Liu(2013),Przytycki-Wise(2012), Przytycki-Wise(2014),Wise(2012raggs),Wise(2012hierachy)}
 there exists a finite normal covering $p \colon \overline{M} \to M$ and a fiber
 bundle $F \to \overline{M} \to S^1$ for some compact connected orientable surface $F$. Hence it
 suffices to show that $\pi_1(F)$ belongs to $\calc$. If $F$ has non-empty boundary,
 this follows from the fact that $\pi_1(F)$ is free. If $M$ is closed, the commutator
 subgroup of $\pi_1(F)$ is free and hence $\pi_1(F)$ belongs to $\calc$. Now
 it follows from assertion~\eqref{the:Status_of_the_Atiyah_Conjecture:Linnell} that $G$ belongs to $\calc$.
 \item[(b)] If $M$ is finitely covered by a torus bundle, then
 it follows from assertion~\eqref{the:Status_of_the_Atiyah_Conjecture:Linnell} that $G$ belongs to $\calc$.
 \item[(c)] Now suppose that $G$ is the fundamental group of a closed graph manifold
 and that $M$ is not finitely covered by a torus bundle. 
 It follows from the arguments of \cite[(C.14), (C.15)]{AFW(2015)} that $M$ admits a finite cover $N$ that contains a non-separating torus $T$.
 We denote by $\widetilde{N}$  the infinite cyclic covering of $N$ corresponding to the Poincar\'e dual of $[T]\subset H_2(N;\IZ)\cong H^1(N;\IZ)=\hom(\pi_1(N),\IZ)$. We can write $\widetilde{N}$ as the union of a nested sequence of $\pi_1$-injective compact graph manifolds with non-empty boundary. It follows from (a) and assertion~\eqref{the:Status_of_the_Atiyah_Conjecture:Linnell} that $\pi_1(\widetilde{N})$ lies in $\calc$. But then it follows, again from 
 assertion~\eqref{the:Status_of_the_Atiyah_Conjecture:Linnell}, that $\pi_1(M)$ itself lies in $\calc$.
 \end{enumerate}
 ~\eqref{the:Status_of_the_Atiyah_Conjecture:approx} This result is due to Schick,
 see for instance~\cite{Schick(2001b)} or~\cite[Theorem~10.22 on page~379]{Lueck(2002)}.
\end{proof}


\subsection{$L^2$-acyclic Atiyah pair}
\label{subsec:L2-acyclic_Atiyah_pair}

\begin{definition}[$L^2$-acyclic Atiyah-pair]\label{def:L2-acyclic_Atiyah-pair} 
 An \emph{$L^2$-acyclic Atiyah-pair} $(\mu,\phi)$ for a finite connected $CW$-complex $X$
 consists of group homomorphisms $\mu \colon \pi_1(X) \to G$ and
 $\phi \colon G \to \IZ$ such that the $G$-covering $\overline{X} \to X$ associated to
 $\mu$ is $L^2$-acyclic, i.e., the $n$th $L^2$-Betti number
 $b_n^{(2)}(\overline{X};\caln(G))$ vanishes for every $n \ge 0$, and $G$ is torsion-free and satisfies the
 Atiyah Conjecture.
\end{definition}

Notice that the conditions appearing in Definition~\ref{def:L2-acyclic_Atiyah-pair} only
concern $G$ and $\mu$ but not $\phi$. The Atiyah Conjecture enters in this paper because
of the following theorem.

\begin{theorem}[The Atiyah Conjecture and the $(\mu,\phi)$-$L^2$-Euler characteristic]%
\label{the:Atiyah_and_(mu,phi)-L2-Euler_characteristic}
Let $X$ be a connected finite $CW$-complex. 
Suppose that $(\mu,\phi)$ is an $L^2$-acyclic Atiyah-pair.
Then $X$ is $(\mu,\phi)$-$L^2$-finite, and the $(\mu,\phi)$-$L^2$-Euler characteristic
$\chi^{(2)}(X;\mu,\phi)$ is an integer.
\end{theorem}

\begin{proof}
Let $X$ be a connected finite $CW$-complex and suppose that $(\mu,\phi)$ is an $L^2$-acyclic Atiyah-pair.
It follows from 
Theorem~\ref{the:Main_properties_of_cald(G)} (4) and Lemma~\ref{lem:phi-twisted_as_ordinary_L2_Euler_characteristic} that 
 $X$ is $(\mu,\phi)$-$L^2$-finite. Furthermore it follows from 
Theorem~\ref{the:Main_properties_of_cald(G)} (2) and (4) that 
the $(\mu,\phi)$-$L^2$-Euler characteristic
$\chi^{(2)}(X;\mu,\phi)$ is an integer.
\end{proof}

As we see, the proof of the previous theorem rests on 
Theorem~\ref{the:Main_properties_of_cald(G)}~\eqref{the:Main_properties_of_cald(G):chain},
whose formulation requires some preparation.


\subsection{The division closure $\cald(G)$ of $\IQ G$ in $\calu(G)$}
\label{subsec:The_division_closure_cald(G)_of_FG_in_calu(G)}

Let $S$ be a ring with subring $R \subset S$. The \emph{division closure} $\cald(R \subset S)$ 
is the smallest subring of $S$ which contains $R$ and is division closed, i.e., every
element in $\cald(R \subset S)$ which is a unit in $S$ is already a unit in $\cald(R \subset S)$.

\begin{notation}[$\cald(G)$]\label{not:T(Sigma_subset_R)} Let $G$ be a group.
 Denote by $\calu(G)$ the algebra of operators
 affiliated to the (complex) group von Neumann algebra $\caln(G)$,
 see~\cite[Section~8.2]{Lueck(2002)}. (This is the Ore localization of $\caln(G)$ with
 respect to the set of non-zero-divisors of $\caln(G)$, see~\cite[Theorem~8.22 on
 page~327]{Lueck(2002)}.) Denote by $\cald(G)$ the division closure of $\IQ G$ considered
 as a subring of $\calu(G)$. 
\end{notation}

On several occasions we will use the following lemma.

\begin{lemma}\label{lem:elementary_amenable_and_ore} 
Let $G$ be a torsion-free elementary amenable group.
 Then the Ore localization 
 $T^{-1} \IZ G$ for $T$ the set of non-trivial elements in $\IZ G$ exists
 and agrees with the skew field $\cald(G)$. In particular $\cald(G)$ is flat over $\IZ G$.
 Moreover, $G$ satisfies the Atiyah Conjecture and we get for every finitely generated free $\IQ G$-chain complex $C_*$ that 
 \[
 b_n^{(2)}(\caln(G) \otimes_{\IQ G} C_*) \,\, =\,\, \dim_{\cald(G)}(H_n(\cald(G) \otimes_{\IQ G} C_*)).
 \] 
\end{lemma}

 \begin{proof} This follows from 
 Theorem~\ref{the:Status_of_the_Atiyah_Conjecture}~\eqref{the:Status_of_the_Atiyah_Conjecture:Linnell},
 Theorem~\ref{the:Main_properties_of_cald(G)}~\eqref{the:Main_properties_of_cald(G):dim}
 and~\cite[ Example~8.16 on page~324 and Lemma~10.16 on page~376]{Lueck(2002)}.
 The proofs there deal only with $\IC$, but carry over without changes to any field $F$ with $\IQ \subseteq F \subseteq \IC$.
\end{proof}

\begin{notation}
Let $\cald$ be a skew field together with an automorphism 
$t\colon \cald \to \cald$ of skew fields. A Laurent polynomial over $\cald$ is 
a formal linear combination  $x = \sum_{i=m}^{n} d_i \cdot u^i$ with $d_m,\dots,d_n\in\cald$. The set of Laurent polynomials becomes a ring
$\cald_{t}[u^{\pm 1}]$ by defining $du^i\cdot eu^j=dt^{-i}(e)u^{i+j}$ for all $d,e\in \cald$ and $i,j\in \IZ$.
\end{notation}

The proof of Theorem~\ref{the:Main_properties_of_cald(G)} will be based on ideas of Peter
Linnell from~\cite{Linnell(1993)} which have been explained in detail and a little bit extended
in~\cite[Chapter~10]{Lueck(2002)} and~\cite{Reich(2006)}. 

\begin{theorem}[Main properties of $\cald(G)$]\label{the:Main_properties_of_cald(G)}
Let $G$ be a torsion-free group.

\begin{enumerate}[font=\normalfont]

\item\label{the:Main_properties_of_cald(G):skew_field}
The group $G$ satisfies the Atiyah Conjecture if and only if
$\cald(G)$ is a skew field;

\item\label{the:Main_properties_of_cald(G):dim}
Suppose that $G$ satisfies the Atiyah Conjecture.
Let $C_*$ be a projective $\IQ G$-chain complex, i.e.\ $C_*$ is a chain complex consisting of projective $($e.g.\ free$)$ $\IQ G$-modules.
Then we get for all $n \ge 0$
\[
b_n^{(2)}\bigl(\caln(G) \otimes_{\IQ G} C_*\bigr) = \dim_{\cald(G)}\bigl(H_n(\cald(G) \otimes_{\IQ G} C_*)\bigr).
\]
In particular $b_n^{(2)}\bigl(\caln(G) \otimes_{\IQ G} C_*\bigr)$ is either infinite or an integer;

\item\label{the:Main_properties_of_cald(G):cald(K)_and_(cald(G)} 
 Suppose that $G$ satisfies the Atiyah Conjecture.
 Let $\phi \colon G \to \IZ$ be a surjective group
 homomorphism. Let $K \subseteq G$ be the kernel of $\phi$. Fix an element $\gamma \in G$
 with $\phi(\gamma) = 1$. If we define the semi-direct product $K \rtimes \IZ$ with
 respect to the conjugation automorphism $c_{\gamma} \colon K \to K$ of $\gamma$ on $K$,
 we can identify $G$ with $K \rtimes \IZ$ and $\phi$ becomes the
 canonical projection $G = K \rtimes \IZ \to \IZ$. Let $\cald(K)_{t}[u^{\pm 1}]$
 be the ring of twisted Laurent polynomials with respect to the automorphism
 $t \colon \cald(K) \xrightarrow{\cong} \cald(K)$ coming
 from $c_{\gamma} \colon K \to K$. 

 Then $\cald(K_{t}[u^{\pm 1}])$ is a non-commutative principal ideal domain, i.e.,
 it has no non-trivial zero-divisor and every left ideal is a principal left ideal and
 every right ideal is a principal right ideal. Furthermore the set $T$ of non-zero elements in
 $\cald(K)_{t}[u^{\pm 1}]$ satisfies the Ore condition and there is a canonical
 isomorphism of skew fields
 \[
 T^{-1}\cald(K)_{t}[u^{\pm 1}] \xrightarrow{\cong} \cald(G);
 \]

\item\label{the:Main_properties_of_cald(G):chain} 
 Let $G$ be a torsion-free group which satisfies the
 Atiyah Conjecture. Let $\phi \colon G \to \IZ$ be a surjective group
 homomorphism. Denote by $i \colon K \to G$ the inclusion of the kernel $K$ of $\phi$.
 Let $C_*$ be a finitely generated projective $\IQ G$-chain complex such that 
 $b_n^{(2)}(\caln(G) \otimes_{\IQ G} C_*)$ vanishes for all $n \ge 0$. Denote by $i^*C_*$ 
 the restriction of the $\IQ G$-chain complex $C_*$ to a $\IQ K$-chain complex.

 Then $H_n(\cald(K) \otimes_{\IQ K} i^*C_*)$ and $H_n(\cald(K) _{t}[u^{\pm 1}] \otimes_{\IQ G} C_*)$
 are finitely generated free as $\cald(K)$-modules, $b_n^{(2)}\bigl(\caln(K) \otimes_{\IQ K} i^*C_*\bigr)$ is finite,
 and we have
\begin{eqnarray*}
b_n^{(2)}\bigl(\caln(K) \otimes_{\IQ K} i^*C_*\bigr) 
& = &
\dim_{\cald(K)}\bigl(H_n(\cald(K) _{t}[u^{\pm 1}] \otimes_{\IQ G} C_*)\bigr).
\\
& = & 
\dim_{\cald(K)}\bigl(H_n(\cald(K) \otimes_{\IQ K} i^*C_*)\bigr)
\end{eqnarray*}
for all $n \ge 0$.

\end{enumerate}
\end{theorem}
\begin{proof}~\eqref{the:Main_properties_of_cald(G):skew_field} This is proved in the
 case $F = \IC$ in~\cite[Lemma~10.39 on page~388]{Lueck(2002)}. The proof goes through
 for an arbitrary field $F$ with $\IQ \subseteq F \subseteq \IC$ without
 modifications. 
 \\[1mm]~\eqref{the:Main_properties_of_cald(G):dim} 
 We have the following commutative diagram of inclusion of rings
\[
\xymatrix@C1.2cm@R0.5cm{\IQ G \ar[r] \ar[d]
& \caln(G) \ar[d]
\\
\cald(G) \ar[r] 
&
\calu(G).
}
\]
There is a dimension function $\dim_{\calu(G)}$ for arbitrary (algebraic) $\calu(G)$-modules
such that for any $\caln(G)$-module $M$ we have 
$\dim_{\calu(G)}(\calu(G) \otimes_{\caln(G)} M) = \dim_{\caln(G)}(M)$ 
and basic features like additivity and continuity and cofinality are still satisfied, 
see~\cite[Theorem~8.29 on page~330]{Lueck(2002)}. Moreover, $\calu(G)$ is flat over $\caln(G)$,
see~\cite[Theorem~8.22~(2) on page~327]{Lueck(2002)}. 
Since $\cald(G)$ is a skew field by assertion~\eqref{the:Main_properties_of_cald(G):skew_field}, 
$\calu(G)$ is also flat as a $\cald(G)$-module
and we have for any $\cald(G)$-module $M$ the equality
$\dim_{\calu(G)}(\calu(G) \otimes_{\cald(G)} M) = \dim_{\cald(G)}(M)$. We conclude
\begin{eqnarray*}
b_n^{(2)}(\caln(G) \otimes_{\IQ G} C_*) 
& = & 
\dim_{\caln(G)} \bigl(H_n(\caln(G) \otimes_{\IQ G} C_*)\bigr)
\\
& = & 
\dim_{\calu(G)} \bigl(\calu(G) \otimes_{\caln(G)} H_n(\caln(G) \otimes_{\IQ G} C_*)\bigr)
\\
& = & 
\dim_{\calu(G)} \bigl(H_n(\calu(G) \otimes_{\caln(G)} \caln(G) \otimes_{\IQ G} C_*)\bigr)
\\
& = & 
\dim_{\calu(G)} \bigl(H_n(\calu(G) \otimes_{\IQ G} C_*)\bigr)
\\
& = & 
\dim_{\calu(G)} \bigl(H_n(\calu(G) \otimes_{\cald(G)} \cald(G) \otimes_{\IQ G} C_*)\bigr)
\\
& = & 
\dim_{\calu(G)} \bigl(\calu(G) \otimes_{\cald(G)} H_n(\cald(G) \otimes_{\IQ G} C_*)\bigr)
\\
& = & 
\dim_{\cald(G)} \bigl(H_n(\cald(G) \otimes_{\IQ G} C_*)\bigr).
\end{eqnarray*}
This finishes the proof of assertion~\eqref{the:Main_properties_of_cald(G):dim}.
\\[1mm]~\eqref{the:Main_properties_of_cald(G):cald(K)_and_(cald(G)} Since $G
= K \rtimes \IZ$ satisfies the Atiyah Conjecture by assumption, the
same is true for $K$ by
Theorem~\ref{the:Status_of_the_Atiyah_Conjecture}~\eqref{the:Status_of_the_Atiyah_Conjecture:subgroups}.
We know already from assertion~\eqref{the:Main_properties_of_cald(G):skew_field} that
$\cald(K)$ and $\cald(G)$ are skew fields. The ring $\cald(K)_{t}[u^{\pm 1}]$ is
a non-commutative principal ideal domain, see~\cite[2.1.1 on page~49]{Cohn(1995)}
or~\cite[Proposition~4.5]{Cochran(2004)}. The claim that the Ore localization
$T^{-1}\cald(K)_{t}[u^{\pm 1}]$ exists and is isomorphic to $\cald(G)$ is proved in
the case $F = \IC$ in~\cite[Lemma~10.60 on page~399]{Lueck(2002)}. The proof goes through
for an arbitrary field $F$ with $\IQ \subseteq F \subseteq \IC$ without
modifications. 
\\[1mm]~\eqref{the:Main_properties_of_cald(G):chain} 
We write the group ring $\IQ G$ as the ring $\IQ K_{t}[u^{\pm 1}]$ of twisted Laurent
polynomials with coefficients in $\IQ K$. We get a commutative diagram of inclusions of rings,
where $\cald(K)_{t}[u^{\pm 1}]$ is a (non-commutative) principal ideal domain and
$\cald(K)$ and $\cald(G)$ are skew fields:

\[\xymatrix@C1.2cm@R0.5cm{\IQ K \ar[r] \ar[d]
&
\IQ G = \IQ K_{t}[u^{\pm 1}] \ar[dd] \ar[ldd]
\\
\cald(K) \ar[d]
&
\\
\cald(K)_{t}[u^{\pm 1}] \ar[r] \ar[d]
&
\cald(G) \ar[d]^{\id}
\\
T^{-1}\cald(K)_{t}[u^{\pm 1}] \ar[r]^-{\cong}
&
\cald(G).
}
\]
Since $C_*$ is a finitely generated projective $\IQ G$-chain complex by assumption,
the $\cald(K)_{t}[u^{\pm 1}]$-chain complex $\cald(K)_{t}[u^{\pm 1}] \otimes_{ \IQ G}C_*$ 
is finitely generated projective. Since $\cald(K \rtimes \IZ)$ is a
(non-commutative) principal ideal domain, it follows from~\cite[p.~494]{Cohn(1985)},
that there exist integers $r,s \ge 0$ and non-zero elements $p_1, p_2, \ldots , p_s \in \cald(K)_{t}[u^{\pm 1}]$ 
such that we get an isomorphism of $\cald(K)_{t}[u^{\pm 1}]$-modules
\[
H_n\bigl(\cald(K)_{t}[u^{\pm 1}] \otimes_{\IQ G} C_*\bigr) \cong \cald(K)_{t}[u^{\pm 1}]^r \oplus \bigoplus_{i = 1}^s 
\cald(K)_{t}[u^{\pm 1}]/(p_i).
\]
Since $\cald(G) = T^{-1}\cald(K)_{t}[u^{\pm 1}]$ is flat over
$\cald(K)_{t}[u^{\pm 1}]$, we conclude using
assertion~\eqref{the:Main_properties_of_cald(G):dim} that
\begin{eqnarray*}
r 
& = &
\dim_{\cald(G)}\bigl(\cald(G)\otimes_{\cald(K)_{t}[u^{\pm 1}]} 
H_n(\cald(K)_{t}[u^{\pm 1}] \otimes_{\IQ G} C_*)\bigr)
\\
& = &
\dim_{\cald(G)}\bigl(
H_n\bigl(\cald(G) \otimes_{\cald(K)_{t}[u^{\pm 1}]} \cald(K)_{t}[u^{\pm 1}] \otimes_{\IQ G} C_*\bigr)\bigr)
\\
& = &
\dim_{\cald(G)}\bigl(
H_n\bigl(\cald(G) \otimes_{\IQ G} C_*\bigr)\bigr)
\\
& = & 
b_n^{(2)}(\caln(G) \otimes_{\IQ G} C_*).
\end{eqnarray*}
Since by assumption $b_n^{(2)}(\caln(G) \otimes_{\IQ G} C_*) = 0$ holds, we conclude
\[
H_n(\cald(K)_{t}[u^{\pm 1}] \otimes_{\IQ G} C_*)\,\, \cong \,\, \bigoplus_{i = 1}^s 
\cald(K)_{t}[u^{\pm 1}]/(p_i).
\]
Lemma~\ref{lem:degree_and_rank} implies that $\cald(K)_{t}[u^{\pm 1}]/(p_i)$ considered as
$\cald(K)$-module is finitely generated free. 
This implies that $H_n\bigl(\cald(K)_{t}[u^{\pm 1}] \otimes_{\IQ G} C_*\bigr)$
considered as $\cald(K)$-module is finitely generated free.
Assertion~\eqref{the:Main_properties_of_cald(G):dim} 
applied to $K$ instead of $G$ implies
\[
b_n^{(2)}\bigl(\caln(K) \otimes_{FK} i^*C_*\bigr) 
=
\dim_{\cald(K)}\bigl(H_n\bigl(\cald(K) \otimes_{\IZ K} i^* C_*\bigr)\bigr). 
\]
There is an obvious isomorphism of $\cald(K)$-chain complexes
\[
\cald(K) \otimes_{\IZ K} i^*C_* \xrightarrow{\cong} \cald(K)_{t}[u^{\pm 1}] \otimes_{\IZ G} C_*,
\quad 
x \otimes_{\IZ K} y \mapsto x \otimes_{\IZ G} y
\]
which induces an isomorphism of $\cald(K)$-modules
\[
H_n\bigl(\cald(K) \otimes_{\IQ K} i^*C_*\bigr) \xrightarrow{\cong} H_n\bigl(\cald(K)_{t}[u^{\pm 1}] \otimes_{\IQ G} C_*\bigr).
\]
Hence we get
\[
\dim_{\cald(K)}\bigl(H_n\bigl(\cald(K) \otimes_{FK} i^*C_*\bigr)\bigr) 
\,\,=\,\,
\dim_{\cald(K)}\bigl(H_n\bigl(\cald(K)_{t}[u^{\pm 1}] \otimes_{\IQ G} C_*\bigr)\bigr). 
\]
This finishes the proof of Theorem~\ref{the:Main_properties_of_cald(G)}
and hence also of Theorem~\ref{the:Atiyah_and_(mu,phi)-L2-Euler_characteristic}.
\end{proof}


 \typeout{-- Section 4: The $\phi$-$L^2$-Euler characteristic is a lower bound for the Thurston norm ---------}

\section{The negative of the $(\mu,\phi)$-$L^2$-Euler characteristic is a lower bound for the Thurston norm}
\label{subsec:The_(mu,phi)-L2-Euler_characteristic_is_a_lower_bound_for_the_Thurston_norm}

\begin{theorem}[The negative of the $(\mu,\phi)$-$L^2$-Euler characteristic is a lower bound for the Thurston norm]
\label{the:The_Thurston_norm_ge_the_(mu,phi)-L2-Euler_characteristic}
Let $M\ne S^1\times D^2$ be an admissible $3$-manifold and let $(\mu,\phi)$ be an $L^2$-acyclic Atiyah-pair.
Then $M$ is $(\mu,\phi)$-$L^2$-finite and we get
\[
-\chi^{(2)}(M;\mu,\phi) \,\,\le \,\,x_M(\phi \circ \mu).
\]
\end{theorem}

Its proof needs some preparation.


\subsection{Some preliminaries about twisted Laurent polynomials over skew fields}
\label{subsec:Some_preliminaries_about_twisted_Laurent_polynomials_over_skew_fields}
In this subsection we consider a skew field $\cald$ together with an automorphism 
$t\colon \cald \to \cald$ of skew fields. Let $\cald_{t}[u^{\pm 1}]$ be the ring of twisted Laurent
polynomials over $\cald$. For a non-trivial element $x = \sum_{i \in \IZ} d_i \cdot u^i$
in $\cald_{t}[u^{\pm 1}]$ we define its degree to be the natural number
 \begin{eqnarray}
 \deg(x) & := & n_+ - n_-
\label{degree_of_a_Laurent_polynomials}
 \end{eqnarray}
 where $n_-$ the smallest integer such that $d_{n_-}$ does not vanish, and $n_+$ is largest
 integer such that $d_{n_+}$ does not vanish.

 \begin{lemma}\label{lem:degree_and_rank} 
 Consider a non-trivial element $x$ in
 $\cald_{t}[u^{\pm 1}]$. Then the $\cald_{t}[u^{\pm 1}]$-homomorphism $r_x
 \colon \cald_{t}[u^{\pm 1}] \to \cald_{t}[u^{\pm 1}]$ given by right
 multiplication with $x$ is injective and its cokernel has finite dimension over
 $\cald$, namely,
 \[\dim_{\cald}(\coker(r_x)) = \deg(x).
 \]
 \end{lemma}
 \begin{proof} Notice that for two non-trivial elements $x$ and $y$ we have $n_-(xy) = n_-(x) + n_-(y)$,
 $n_+(xy) = n_+(x) + n_+(y)$, and $\deg(xy) = \deg(x) + \deg(y)$. Now one easily checks that
 $r_x$ is injective and that a $\cald$-basis for the cokernel of $r_x$ is given by the image of the subset
 $\{u^0, u^1, \ldots , u^{\deg(x)-1}\}$ of $\cald_{t}[u^{\pm 1}]$ under the canonical projection
 $\cald_{t}[u^{\pm 1}] \to \coker(r_x)$.

\end{proof}

\begin{lemma}\label{rank_estimate_for_a_plus_t_I_k} 
Consider integers $k,n$ with $0 \le k$, $1 \le n$ and $k \le n$. Denote by
 $I_k^n$ the $(n,n)$-matrix whose first $k$ entries on the diagonal are $1$ and all of
 whose other entries are zero. Let $A$ be an $(n,n)$-matrix over $\cald$. Suppose that
 the $\cald_{t}[u^{\pm 1}]$-map 
 \[
 r_{A + u \cdot I_k^n} \colon
 \cald_{t}[u^{\pm 1}]^n \to \cald_{t}[u^{\pm 1}]^n
 \] 
 given by right multiplication with $A + u \cdot I_k^n$ is injective. 

 Then the dimension over $\cald$ of its cokernel is finite and satisfies
\[
\dim_{\cald}\bigl(\coker\bigl(r_{A + u \cdot I_k^n} \colon 
\cald_{t}[u^{\pm 1}]^n \to \cald_{t}[u^{\pm 1}]^n\bigr)\bigr) \,\,\le\,\, k.
\]
\end{lemma}
\begin{proof}
 We use induction over the size $n$. If $ k = n$, the claim has already been proved
 in~\cite[Proposition~9.1]{Harvey(2005)}. So we can assume in the sequel $k < n$. We
 perform certain row and column operations on matrices $B \in
 M_{n,n}(\cald_{t}[u^{\pm 1}])$ and it will be obvious that they will respect the
 property that $r_B \colon \cald_{t}[u^{\pm 1}]^n \to \cald_{t}[u^{\pm 1}]^n$ is
 injective and will not change $\dim_{\cald}\bigl(\coker\bigl(r_{B} \colon
 \cald_{t}[u^{\pm 1}]^n \to \cald_{t}[u^{\pm 1}]^n\bigr)\bigr)$. With the help
 of these operations we will reduce the size of $A$ by $1$ and this will finish the
 induction step. To keep the exposition easy, we explain the induction step from $(n-1)$
 to $n$ in the special case $k = 3$ and $n = 5$. The general induction step is
 completely analogous.

So we start with 
\[
A + u \cdot I_3^5 =\footnotesize{\begin{pmatrix}
u + * & * & * & * & * 
\\
* & u+ * &* & * & *
\\
* & * & u+* & * & *
\\
* & * & * & * & *
\\
* & * & * & * & *
\end{pmatrix}}
\]
where here and in the sequel $*$ denotes some element in $\cald$.
We first treat the case, where the $(2,2)$-submatrix sitting in the right lower corner is non-trivial.
By interchanging rows and columns and right multiplying a row with a non-trivial element in $\cald$, we can achieve 

\[
\footnotesize{\begin{pmatrix}
u+* & * & * & * & * 
\\
* & u+* &* & * & *
\\
* & * & u+* & * & *
\\
* & * & * & * & *
\\
* & * & * & * & 1
\end{pmatrix}}
\]
By subtracting appropriate right $\cald$-multiplies of the lowermost row from the other rows and 
subtracting appropriate left $\cald$-multiplies of the rightmost column from the other columns,
we achieve
\[
\footnotesize{
\begin{pmatrix}
u+* & * & * & * & 0
\\
* & u+* &* & * & 0
\\
* & * & u+* & * & 0
\\
* & * & * & * & 0
\\
0 & 0 & 0 & 0 & 1
\end{pmatrix}}
\]
For this matrix the claim follows from the induction hypothesis applied to the $(4,4)$-matrix
obtained by deleting the lowermost row and the rightmost column. 

It remains to treat the case, where the matrix looks like
\[
\footnotesize{
\begin{pmatrix}
u+* & * & * & * & * 
\\
* & u+* &* & * & *
\\
* & * & u+* & * & *
\\
* & * & * & 0 & 0
\\
* & * & * & 0 & 0
\end{pmatrix}}
\]
At least one of the entries in the lowermost row must be non-trivial since the map induced
by right multiplication with it is assumed to be injective. By interchanging rows and
columns and right multiplying a row with a non-trivial element in $\cald$, we can achieve
\[
\footnotesize{
\begin{pmatrix}
u+* & * & * & * & * 
\\
* & u+* &* & * & *
\\
* & * & u+* & * & *
\\
* & * & * & 0 & 0
\\
1 & * & * & 0 & 0
\end{pmatrix}}
\]
By subtracting appropriate $\cald$-multiples of first column from the second and third column, we
can arrange
\[
\footnotesize{
\begin{pmatrix}
u+* & * \cdot u+* & *\cdot u+* & * & * 
\\
* & u+* &* & * & *
\\
* & * & u+* & * & *
\\
* & * & * & 0 & 0
\\
1 & 0 & 0 & 0 & 0
\end{pmatrix}}
\]
By subtracting appropriate right $\cald$-multiples of last row from the other rows we can achieve
\[
\footnotesize{
\begin{pmatrix}
0 & * \cdot u+* & * \cdot u+* & * & * 
\\
0 & u+* &* & * & *
\\
0 & * & u+* & * & *
\\
0 & * & * & 0 & 0
\\
1 & 0 & 0 & 0 & 0
\end{pmatrix}}
\]
By subtracting right $\cald$-multiples of the second and the third row from the first row and then interchanging rows
we can arrange
\[
\footnotesize{
\begin{pmatrix}
0 & u+* & * & * & * 
\\
0 & * & u+* & * & *
\\
0 & * & * & * & *
\\
0 & * & * & 0 & 0
\\
1 & 0 & 0 & 0 & 0
\end{pmatrix}}
\]
Since the induction hypothesis applies to the matrix obtained by deleting the first column and the lowermost row,
we have finished the induction step, and hence the proof of Lemma~\ref{rank_estimate_for_a_plus_t_I_k}.
\end{proof}


\subsection{Proof of Theorem~\ref{the:The_Thurston_norm_ge_the_(mu,phi)-L2-Euler_characteristic}}
\label{subsec:Proof_of_Theorem_ref_the:The_Thurston_norm_ge_the_(mu,phi)-L2-Euler_characteristic)}

\begin{proof}[Proof of Theorem~\ref{the:The_Thurston_norm_ge_the_(mu,phi)-L2-Euler_characteristic}]
We have already shown in Theorem~\ref{the:Atiyah_and_(mu,phi)-L2-Euler_characteristic}
that $M$ is $(\mu,\phi)$-$L^2$-finite.

Because of~\eqref{scaling_Thurston_norm} and Lemma~\ref{lem:reduction_to_surjective_mu} 
we can assume without loss of generality that $\phi$ is surjective. 
Let $K$ be the kernel of $\phi$ and $\gamma$ be an element in $G$ with $\phi(\gamma) = 1$.
 It follows easily
 from~\cite[Lemma~1.2]{Turaev(2002)} that we can find an oriented surface $\Sigma\subset M$ with
 components $\Sigma_1,\dots,\Sigma_l$ and non-zero $r_1,\dots,r_l\in \IN$ with the following
 properties: 
 \begin{itemize}
\item $r_1[\Sigma_1]+\dots+r_l[\Sigma_l]$ is dual to $\phi\circ \mu$,
\item $\sum_{i=1}^l -r_i\chi(\Sigma_i)\leq x_M(\phi\circ \mu)$,
\item $M\setminus \Sigma$ is connected.
\end{itemize}
(Here we use that $M\ne S^1\times D^2$ and that $N$ is irreducible.)
For $i=1,\dots,l$ we pick disjoint oriented tubular neighborhoods $\Sigma_i\times [0,1]$
and we identify $\Sigma_i$ with $\Sigma_i\times \{0\}$. We write $M':=M\setminus
\cup_{i=1}^l \Sigma_i\times [0,1]$. We pick once and for all a base point $p$ in $M'$ and
we denote by $\widetilde{M}$ the universal cover of $M$. We write $\pi=\pi_1(M,p)$. For
$i=1,\dots,l$ we also pick a curve $\nu_i'$ based at $p$ which intersects $\Sigma_i$
precisely once in a positive direction and does not intersect any other component of
$\Sigma$. Put $\nu_i = \mu(\nu_i')$. Note that $\phi(\nu_i)=r_i$. Finally for
$i=1,\dots,l$ we put
\[
n_i :=- \chi(\Sigma_i)+2.
\]

Following~\cite[Section~4]{Friedl(2014twisted)} we now pick an appropriate $CW$-structure
for $M$ and we pick appropriate orientations and lifts of the cells to the universal
cover. The resulting boundary maps are described in
detail~\cite[Section~4]{Friedl(2014twisted)}. In order to keep the notation manageable we
now restrict to the case $l=2$, the general case is completely analogous.

Let $\overline{M} \to M$ be the $G$-covering associated to $\mu$. It follows from the
discussion in~\cite[Section~4]{Friedl(2014twisted)} and the definitions that the cellular
$\IZ G$-chain complex $C_*(\overline{M})$ of $\overline{M} $ looks like
\[
\xymatrix{0 \ar[r] 
& \IZ[G]^{4}\ar[r]^-{B_3} 
& \IZ[G]^{4+2n_1+2n_2+s}\ar[r]^-{B_2} 
& \IZ[G]^{4+2n_1+2n_2+s}\ar[r]^-{B_1} 
& \IZ[G]^4\ar[r] 
& 0}
\]
where $s$ is a natural number and the matrices $B_3,B_2,B_1$ are matrices of the form
\[
B_3
\,\, = \,\, 
\footnotesize{\begin{array}{ccccccc|c}
n_1 & n_2 & 1 & 1 & 1 & 1 & s+n_1+n_2 &
\\
\hline 
 * & 0 & 1 & -\nu_1& 0 & 0 & 0 & 1
\\ 
 0 & 0 & 1 & -z_1 & 0 & 0 & * & 1
\\
 0 & * & 0 & 0 & 1 & -\nu_2 & 0 & 1
\\ 
 0 & 0 & 0 & 0 & 1 & -z_2 & * &1
\end{array} }
\]
\[
B_2
\,\,=\,\, 
\footnotesize{
\begin{array}{ccccccccc|c}
1 &1 & n_1 & n_1 & n_2 & n_2 & 1 & 1 & s
\\
\hline
 * & 0 & \id_{n_1} & -\nu_1\id_{n_1} & 0 & 0 & 0 & 0 & 0 & n_1
\\ 
 0 & * & 0 & 0& \id_{n_2} &-\nu_2\id_{n_2} & 0 & 0 & 0 & n_2
\\ 
 0 & 0 & * & 0 & 0 & 0 & 0 & 0 & 0 & 1
\\ 
 0 & 0 & 0 & * & 0 & 0 & 0 & 0 & 0 & 1
\\ 
 0 & 0 & 0 & 0 & * & 0 & 0 & 0 & 0 & 1
\\ 
 0 & 0 & 0 & 0 & 0 & * & 0 & 0 & 0 & 1
\\ 
 0 & 0 & * & * & * & * & * & * & * & s+n_1+n_2
\end{array} }
\]
\[
B_1\,\,=\,\,
\footnotesize{
\begin{array}{cccc|c}
1 & 1 & 1 & 1
\\
\hline
 1& -\nu_1 & 0 & 0 & 1
\\ 
 0 & 0 & 1 &-\nu_2 & 1
\\ 
 * & 0 & 0 & 0 & n_1
\\
 0 & * &0 & 0 & n_1
\\
0 & 0 & * & 0 & n_2
\\
0 & 0 & 0 & * & n_2
\\
1 & -x_1 & 0 & 0 & 1 
\\
0 & 0 & 1 & -x_2 & 1 
\\ 
 * & * & * & * & s
\end{array} }
\]
where $x_1,x_2,z_1,z_2\in K$, all the entries of the matrices marked by $*$ lie in $\IZ K$
and the entries above the horizontal line and right to the vertical arrow indicate the
size of the blocks. (Note that in~\cite{Friedl(2014twisted)} we view the elements of
$\IZ[G]^n$ as column vectors whereas we now view them as row vectors.)

Define submatrices $B_i'$ of $B_i$ for $i = 1,2,3$ by 
\[
\begin{array}{rcl} 
B_3'
& = &
\footnotesize{
\begin{pmatrix} 
1 & -\nu_1 & 0 & 0 
\\ 
1 & -z_1 & 0 & 0 
\\ 
0 & 0 & 1 & -\nu_2 
\\ 
0 & 0 & 1 &-z_2
\end{pmatrix}}
\\
&
\\[-0.3cm]
B_2'
& = &
\footnotesize{
\begin{array}{ccccc|c}
 n_1 & n_1 & n_2 & n_2 & s & 
\\
\hline
\id_{n_1} & -\nu_1\id_{n_1} & 0 & 0 & 0 & n_1
\\ 
0 & 0& \id_{n_2} &-\nu_2\id_{n_2} & 0 & n_2
\\ 
* & * & * & * & * & s+n_1+n_2
\end{array} }
\\
&
\\[-0.3cm]
B_1'
& = &
\footnotesize{
\begin{pmatrix} 
1 & -\nu_1 & 0 & 0 
\\
0 & 0 & 1 & -\nu_2 
\\
1 & -x_1 & 0 & 0 
\\
0 & 0 & 1 & -x_2 
\end{pmatrix}}
\end{array}
\]
where $B_1'$ and $B_3'$ are $(4,4)$-matrices and $B_2'$ is a $(2n_1+ 2n_2 + s,2n_1+2n_2+ s)$-matrix. 
For $j = 1,2,3$ let $C_*[j]$ be the $\IZ G$-chain complex
concentrated in dimensions $j$ and $j - 1$ whose $j$-th differential is given by the
matrix $B_j'$. There is an appropriate finitely generated free $\IZ G$-chain complex
$C_*'$ concentrated in dimensions $2$, $1$ and $0$ of the shape
\[
\ldots \to 0 \to 0 \to \IZ G^{2n_1 +2n_2+s} \to \IZ G^{4 + 2n_1 +2n_2+s} \to \IZ G^4
\]
and obvious short based exact sequences of finitely generated based free $\IZ G$-chain
complexes
\[
0 \to C_*' \to C_*(\overline{M}) \to C_*[3] \to 0
\]
and
\[
0 \to C_*[1] \to C_*' \to C_*[2] \to 0.
\]
Consider $\nu' \in G$ with $\phi(\nu') \not= 0$ and $x' \in K$. The matrix
$\begin{pmatrix}1& - \nu' \\ 1 & x'\end{pmatrix}$ can be transformed by subtracting the
first row from the second row to the matrix 
$\begin{pmatrix}1& -\nu' \\ 0 & \nu' + x'\end{pmatrix}$. 
The map $r_{\nu' + x'} \colon \cald(K)_{t}[u^{\pm 1}] \to \cald(K)_{t}[u^{\pm 1}]$ 
is injective and its cokernel has dimension $|\phi(\nu')|$
over $\cald(K)$ by Lemma~\ref{lem:degree_and_rank}.
Hence the map $\cald(K)_{t}[u^{\pm 1}]^2 \to \cald(K)_{t}[u^{\pm 1}]^2$
given by right multiplication with $\begin{pmatrix}1& -\nu' \\ 0 & \nu' + x'\end{pmatrix}$
is injective and its cokernel has dimension $|\phi(\nu')|$
over $\cald(K)$. We conclude from Theorem~\ref{the:Main_properties_of_cald(G)} for $l = 1,3$
using notation~\eqref{b_n(2)(caln(G)_otimes_ZGC_ast)} and defining for a skewfield $\cald$ and 
a $\cald$-chain complex $D_*$  its Betti number $b_n(D_*)$
to be $\dim_{\cald}(H_n(D_*))$.

\begin{eqnarray*}
b_n^{(2)}\bigl(\caln(G) \otimes_{\IZ G} C_*[l]\bigr) 
& = & 
0 \quad \text{for} \; n \ge 0;
\\
b_n\bigl(\cald(K) \otimes_{\IZ K} i^*C_*[l]\bigr) 
& = & 
\begin{cases}
0 & \text{if}\; n \not= l-1;\\
r_1 + r_2 & \text{if}\; n = l-1;
\end{cases}
\\
\chi^{(2)}\bigl(\caln(K) \otimes_{\IZ K} i^*C_*[l]\bigr) 
& = & 
r_1 + r_2.
\end{eqnarray*}
Recall that  $b_n^{(2)}\bigl(\caln(G) \otimes_{\IZ G} C_*(\overline{M})\bigr) = 0$ holds
for all $n \ge 0$ by assumption. We had just seen that for $l=1,3$ we have
 $b_n^{(2)}\bigl(\caln(G) \otimes_{\IZ G} C_*[l]\bigr) =0$ for all $n\geq 0$.
It follows from the above short exact sequences that
 we get
\[
b_n^{(2)}\bigl(\caln(G) \otimes_{\IZ G} C_*[2]\bigr)
 = 
0 \quad \text{for} \; n \ge 0.
\]
We conclude from Theorem~\ref{the:Main_properties_of_cald(G)}~\eqref{the:Main_properties_of_cald(G):dim}
\[
\dim_{\cald(G)}\left(\ker\bigl(r_{B_2'} \colon \cald(G)^{2n_1 +2n_2+s} \to \cald(G)^{2n_1 +2n_2+s}\bigr)\right) = 0.
\]
Theorem~\ref{the:Main_properties_of_cald(G)}~\eqref{the:Main_properties_of_cald(G):cald(K)_and_(cald(G)} 
implies that 
\[
r_{B_2'} \colon \cald(K)_{t}[u^{\pm 1}]^{2n_1 +2n_2+s} \to \cald(K)_{t}[u^{\pm 1}]^{2n_1 +2n_2+s}
\]
is injective. We get using Theorem~\ref{the:Main_properties_of_cald(G)}~\eqref{the:Main_properties_of_cald(G):dim}
\begin{eqnarray}
\label{the:The_Thurston_norm_and_the_(mu,phi)_L2-Euler_characteristic:(3)}
&&
\\
&&\chi^{(2)}\bigl(\caln(K) \otimes_{\IZ K} i^*C_*(\overline{M})\bigr) \,\,=\,\,\lmsum{l = 1}{3} \chi^{(2)}\bigl(\caln(K) \otimes_{\IZ K} i^*C_*[l]\bigr)
\nonumber
\\
& = & 
2r_1 + 2r_2 - \chi^{(2)}\bigl(\caln(K) \otimes_{\IZ K} i^*C_*[2]\bigr)
\nonumber
\\
& = & 
2r_1 + 2r_2 - \chi\bigl(\cald(K) \otimes_{\IZ K} i^*C_*[2]\bigr)
\nonumber\\
& = &
2r_1 + 2r_2 
\nonumber\\
& & - \dim_{\cald(K)}\bigl(\coker\bigl(r_{B_2'} \colon \cald(K)_{t}[u^{\pm 1}]^{2n_1 +2n_2+s} 
\to \cald(K)_{t}[u^{\pm 1}]^{2n_1 +2n_2+s}\bigr)\bigr).
\nonumber
\end{eqnarray}

It follows from the above two short exact sequences of chain complexes and 
the previous calculations that  
$\dim_{\caln(G)}\bigl(\ker\bigl(\id_{\caln(G)} \otimes_{\IZ G} r_{B_2'} \colon \caln(G)^{2n_1 +2n_2+s}
\to \caln(G)^{2n_1 +2n_2+s}\bigr)\bigr)$ coincides 
with $b_2^{(2)}(\caln(G) \otimes_{\IZ  G} C_*')$. But $b_2^{(2)}(\caln(G) \otimes_{\IZ G} C_*')$ is zero, therefore one
of the entries in the rightmost column of $B_2'$ is
non-trivial. By switching rows and multiplying a row with a non-trivial element in $\cald(K)$, 
we can arrange that the entry in the lower right corner is $1$. By elementary row and
column operations we can arrange that the lowermost row and the rightmost column have all
entries zero except the element in the lower right corner which is still equal to $1$. By
iterating this process we can transform $B_2'$ by such row and column operations into a
matrix of the shape
\[
B_2''
\,\,=\,\,
\footnotesize{ 
\begin{array}{ccccc|c}
 n_1 & n_1 & n_2 & n_2 & s & 
\\
\hline
\id_{n_1} & -\nu_1\id_{n_1} & 0 & 0 & 0 & n_1
\\ 
0 & 0& \id_{n_2} &-\nu_2\id_{n_2} & 0 & n_2
\\ 
* & * & * & * & * & n_1+n_2
\\ 
0 & 0 & 0 & 0 & \id_s & s
\end{array} }
\]
Let $B_2'''$ be the $(2n_1+ 2n_2)$-submatrix of $B_s''$ given by
\[
B_2'''
\,\, =\,\,
\footnotesize{ 
\begin{array}{cccc|c}
 n_1 & n_1 & n_2 & n_2 & 
\\
\hline
\id_{n_1} & -\nu_1\id_{n_1} & 0 & 0 & n_1
\\ 
0 & 0& \id_{n_2} &-\nu_2\id_{n_2} & n_2
\\ 
* & * & * & * & n_1+n_2
\end{array}}
\]
An inspection of the proof of~\cite[Lemma~9.3]{Dubois-Friedl-Lueck(2014Alexander)} shows
that by elementary row and column operations and taking block sum with triangular matrices
having $\id$ on each diagonal entry,
we can transform $B_2'''$ into a matrix of the shape
\[
B_2'''' = A'''' + u \cdot I_{r_1\cdot n_1 + r_2\cdot n_2}^{2r_1\cdot n_1 +2r_2 \cdot n_2}
\]
for some matrix $A'''' \in M_{2n_1 +2n_2,2n_1 +2n_2}(\cald(K))$ and
$I_{r_1\cdot n_1 + r_2\cdot n_2}^{2r_1\cdot n_1 +2r_2 \cdot n_2}$ as introduced in Lemma~\ref{rank_estimate_for_a_plus_t_I_k}. 
Obviously we have
\begin{multline*}
\dim_{\cald(K)}\bigl(\coker\bigl(r_{B_2'} \colon \cald(K)_{t}[u^{\pm 1}]^{2n_1 +2n_2+s} 
\to \cald(K)_{t}[u^{\pm 1}]^{2n_1 +2n_2+s}\bigr)\bigr)
\\
= \dim_{\cald(K)}\bigl(\coker\bigl(r_{B_2''''} \colon \cald(K)_{t}[u^{\pm 1}]^{2n_1 +2n_2} 
\to \cald(K)_{t}[u^{\pm 1}]^{2n_1 +2n_2}\bigr)\bigr).
\end{multline*}
We conclude from Lemma~\ref{rank_estimate_for_a_plus_t_I_k} applied to $B_2''''$
\begin{multline*}
\dim_{\cald(K)}\bigl(\coker\bigl(r_{B_2''''} \colon \cald(K)_{t}[u^{\pm 1}]^{2n_1 +2n_2} 
\to \cald(K)_{t}[u^{\pm 1}]^{2n_1 +2n_2}\bigr)\bigr)
\\
\le r_1 \cdot n_1 + r_2 \cdot n_2.
\end{multline*}
Hence we get 
\begin{multline}
\dim_{\cald(K)}\bigl(\coker\bigl(r_{B_2'} \colon \cald(K)_{t}[u^{\pm 1}]^{2n_1 +2n_2+s} 
\to \cald(K)_{t}[u^{\pm 1}]^{2n_1 +2n_2+s}\bigr)\bigr)
\label{the:The_Thurston_norm_and_the_(mu,phi)_L2-Euler_characteristic:(4)}
\\
\le r_1 \cdot n_1 + r_2 \cdot n_2.
\end{multline}
We conclude from~\eqref{the:The_Thurston_norm_and_the_(mu,phi)_L2-Euler_characteristic:(3)}
and~\eqref{the:The_Thurston_norm_and_the_(mu,phi)_L2-Euler_characteristic:(4)}
\[
\chi^{(2)}\bigl(\caln(K) \otimes_{\IZ K} i^*C_*(\overline{M})\bigr) \ge 2r_1 + 2r_2 - r_1 \cdot n_1 - r_2 \cdot n_2.
\]
This together with Lemma~\ref{lem:phi-twisted_as_ordinary_L2_Euler_characteristic} implies
\[ \begin{array}{rcl}
- \chi^{(2)}(M;\mu,\phi) 
& = &
- \chi^{(2)}\big(\caln(K) \otimes_{\IZ K} i^*C_*(\overline{M})\big)\\
& \le &
 r_1 n_1 + r_2 n_2 - 2r_1 - 2r_2
\\
& = & 
 r_1 (n_1 -2) + r_2 (n_2 -2)
\,\, = \,\,
- r_1 \chi(\Sigma_1) - r_2 \chi(\Sigma_2)
\,\,\le \,\,
x_M(\phi\circ \mu).
\end{array}
\]
This finishes the proof of Theorem~\ref{the:The_Thurston_norm_ge_the_(mu,phi)-L2-Euler_characteristic}.
\end{proof}


 \typeout{---- Section 5: Fox calculus and and the $(\mu,\phi)$-$L^2$-Euler characteristic ------}

\section{Fox calculus and the $(\mu,\phi)$-$L^2$-Euler characteristic}
\label{sec:Fox_calculus_and_the_(mu,phi)-L2-Euler_characteristic}

The following calculations of the $(\mu,\phi)$-$L^2$-Euler characteristic from a
presentation of the fundamental group is adapted to the corresponding calculation for the
$L^2$-torsion function appearing in~\cite[Theorem~2.1]{Friedl-Lueck(2015l2+Thurston)} and
will be used when we will compare these two invariants and the higher order Alexander
polynomials. For information about the Fox matrix and the Fox calculus we refer for instance 
to~\cite[9B on page 123]{Burde-Zieschang(1985)}, and~\cite{Fox(1953)}.

\begin{theorem}[Calculation of the $(\mu,\phi)$-$L^2$-Euler characteristic 
from a presentation of the fundamental group]
\label{the:calculation_of_(mu,phi)_L2-Euler_characteristic_from_a_presentation} 
Let $M$ be an admissible $3$-manifold with fundamental group $\pi$. Let
\[
\pi = \langle x_1,x_2, \ldots, x_a \mid R_1,R_2, \ldots, R_b \rangle
\]
be a presentation of $\pi$. Let the $(b,a)$-matrix over $\IZ \pi$
\[
F = 
\left( \begin{array}{ccc} \frac{\partial R_1}{\partial x_1}
 & \ldots & \frac{\partial R_1}{\partial x_a} \\
 \vdots & \ddots & \vdots \\
\frac{\partial R_b}{\partial x_1} & \ldots & \frac{\partial R_b}{\partial x_a}
\end{array}\right) 
\]
be the Fox matrix of the presentation. 

Let $G$ be a torsion-free group which satisfies the Atiyah Conjecture.
Consider two group homomorphisms 
$\mu \colon \pi \to G$ and $\phi \colon G \to \IZ$.
Suppose that $\mu$ is non-trivial and that $\phi$ is surjective.

Let $i \colon K \to G$ be the inclusion of the kernel $K = \ker(\phi)$ of 
$\phi$ and let $t \colon \cald(K) \xrightarrow{\cong} \cald(K)$ be the automorphism
introduced in 
Theorem~\ref{the:Main_properties_of_cald(G)}~\eqref{the:Main_properties_of_cald(G):cald(K)_and_(cald(G)}. 
Let $\overline{M} \to M$ be the $G$-covering associated to $\mu$.

\begin{enumerate}[font=\normalfont]

\item\label{the:calculation_of_(mu,phi)_L2-Euler_characteristic_from_a_presentation:non-empty} 
Suppose that $\partial M$ is non-empty and and that $a = b + 1$. 
Choose $i \in \{1,2, \ldots, r\}$ such that  $\mu(x_i)$ has infinite order.
Let $A$ be the $(a-1,a-1)$-matrix with entries in $\IZ G$
obtained from the Fox matrix $F$ by deleting the $i$-th column
and then applying the homomorphism $M_{a-1,a-1}(\IZ \pi) \to M_{a-1,a-1}(\IZ G)$ 
induced by $\mu \colon \pi \to G$.

Then $b_n^{(2)}(\overline{M};\caln(G))$ vanishes for all $n \ge 0$ if and only if we have
$\dim_{\caln(G)}\bigl(\ker(r_A \colon \caln(G)^{a-1} \to \caln(G)^{a-1})\bigr) = 0$.
If this is the case, then $(\mu,\phi)$ is an $L^2$-acyclic Atiyah-pair in the sense of 
Definition~\ref{def:L2-acyclic_Atiyah-pair} and we get 
\begin{multline*}
\chi^{(2)}(M;\mu,\phi)
\\
= 
- \dim_{\cald(K)}\bigl(\coker\bigl(r_A \colon \cald(K)_{t}[u^{\pm 1}]^{a-1} 
\to \cald(K)_{t}[u^{\pm 1}]^{a-1}\bigr)\bigr)
\\
+ |\phi \circ \mu(x_i)|;
\end{multline*}

\item\label{the:calculation_of_(mu,phi)_L2-Euler_characteristic_from_a_presentation:empty} 
Suppose $\partial M$ is empty. We make the assumption
that the given presentation comes from a Heegaard decomposition
as described in~\cite[Proof of Theorem~5.1]{McMullen(2002)}.
Then $a = b$ and there is another set of dual generators 
$\{x_1', x_2', \ldots, x_a'\}$ coming from the Heegaard decomposition as described 
in~\cite[Proof of Theorem~5.1]{McMullen(2002)}. Choose $i,j \in \{1,2, \ldots, a\}$ 
such that  $\mu(x_i)$ and  $\mu(x_j')$ have infinite order.
Let $A$ be the $(a-1,a-1)$-matrix with entries in $\IZ G$
obtained from the Fox matrix $F$ by deleting the $i$th column and the $j$th row
and then applying the homomorphism $M_{a-1,a-1}(\IZ \pi) \to M_{a-1,a-1}(\IZ G)$
induced by $\mu \colon \pi \to G$.

Then $b_n^{(2)}(\overline{M};\caln(G))$ vanishes for all $n \ge 0$ if and only if we have
$\dim_{\caln(G)}\bigl(\ker(r_A \colon \caln(G)^{a-1} \to \caln(G)^{a-1})\bigr) = 0$.
If this is the case, then $(\mu,\phi)$ is an $L^2$-acyclic Atiyah-pair and we get 
\begin{multline*}
\chi^{(2)}(M;\mu,\phi) 
\\
= 
- \dim_{\cald(K)}\bigl(\coker\bigl(r_A \colon \cald(K)_{t}[u^{\pm 1}]^{a-1} 
\to \cald(K)_{t}[u^{\pm 1}]^{a-1}\bigr)\bigr)
\\
+ |\phi \circ \mu(x_i)| + |\phi \circ \mu(x_j')|.
\end{multline*}
\end{enumerate}
\end{theorem}
\begin{proof} 
 We only treat the case, where $\partial M$ is empty, and leave it to the reader to
 figure out the details for the case of a non-empty boundary using the proof
 of~\cite[Theorem~2.4]{Lueck(1994a)}.

 From~\cite[Proof of Theorem~5.1]{McMullen(2002)} we
 obtain a compact $3$-dimensional $CW$-complex $X$ together with a homotopy equivalence
 $f \colon X \to M$ and a set of generators $\{x_1', \ldots ,x_a'\}$ such that
 the $\IZ G$-chain complex $C_*(\overline{X})$ of $\overline{X}$ looks like 
\[
\IZ G 
\xrightarrow{\prod_{i=1}^a r_{\mu(x_i') - 1}} 
\bigoplus_{i=1}^a \IZ G 
\xrightarrow{r_{\mu(F)}} \bigoplus_{i=1}^a \IZ G 
\xrightarrow{\bigoplus_{i=1}^a r_{\mu(x_i) - 1}} \IZ G,
\]
where $\mu(F)$ is the image of $F$ under the map $M_{a,a}(\IZ \pi) \to M_{a,a}(\IZ G)$ induced 
by $\mu$ and $\overline{X} \to X$ is the pullback of $\overline{M} \to M$ with $f$.

For $g \in G$ with $g \not= 1$ let $D(g)_*$ be the $1$-dimensional
$\IZ G$-chain complexes which has as first differential
$r_{g -1} \colon \IZ G\to \IZ G$. Since $g$ generates an infinite cyclic subgroup of $G$,
we conclude $b_n^{(2)}(\caln(G) \otimes_{|\IZ G} D(g)_*) = 0$
for $n \ge 0$ from~\cite[Lemma~1.24~(4) on page~30 and Lemma~1.34~(1) on page~35]{Lueck(2002)}.
Theorem~\ref{the:Main_properties_of_cald(G)}~\eqref{the:Main_properties_of_cald(G):chain} 
and Lemma~\ref{lem:degree_and_rank} imply
\begin{eqnarray}
\chi^{(2)}\bigl(\caln(K) \otimes_{\IZ K} i^*D_*(g)) & = & |\phi(g)|.
\label{chi(2)(D(g)}
\end{eqnarray}
There is a surjective $\IZ G$-chain map 
$p_* \colon C_*(\overline{X}) \to \Sigma^2 D(\mu(x_j'))_*$ 
which is the identity in degree $3$ and the
projection onto the summand corresponding to $j$ in degree $2$, and an injective 
$\IZ G$-chain map $i_* \colon D(\mu(x_i))_* \to C_*(\overline{X})$
which is the identity in degree $0$ and the inclusion to the summand corresponding to
$i$ in degree $1$. Let $P_*$ be the kernel of $p_*$ and let $Q_*$ be the cokernel of
the injective map $j_* \colon D(\mu(x_1))_* \to P_*$ induced by 
$i_* \colon D(\mu(x_1))_* \to C_*(\overline{X})$. 
Then $Q_*$ is concentrated in dimensions $1$ and $2$ and its second differential is
$r_A \colon \IZ G^{a-1} \to \IZ G^{a-1}$. We conclude from the long exact
homology sequence of a short exact sequence of Hilbert $\caln(G)$-chain complexes,
the homotopy invariance of $L^2$-Betti numbers and the additivity of the von Neumann dimension
\[
b_n^{(2)}(\overline{X};\caln(G)) = 
\begin{cases}
\dim_{\caln(G)}\bigl(\ker(r_A \colon \caln(G)^{a-1} \to \caln(G)^{a-1})\bigr)& \text{if} \; n = 1,2;
\\
0 & \text{otherwise}.
\end{cases}
\]
This shows that $b_n^{(2)}(\overline{M};\caln(G))$ vanishes for all $n \ge 0$ if and only if we have
$\dim_{\caln(G)}\bigl(\ker(r_A \colon \caln(G)^{a-1} \to \caln(G)^{a-1})\bigr) = 0$.

Suppose that this is the case. Then $(\mu,\phi)$ is an $L^2$-acyclic Atiyah-pair.
We conclude from Theorem~\ref{the:Main_properties_of_cald(G)}
that
$r_A \colon \cald(G)^{a-1} \to \cald(G)^{a-1}$ is injective and hence
$r_A \colon \cald(K)_{t}[u^{\pm 1}]^{a-1} \to \cald(K)_{t}[u^{\pm 1}]^{a-1}$ is injective
and 
\begin{multline}
\chi^{(2)}(\caln(K) \otimes_{\IZ K} i^*Q_*)
\\
 = 
\dim_{\cald(K)}\bigl(\coker\bigl(r_A \colon \cald(K)_{t}[u^{\pm 1}]^{a-1} \to \cald(K)_{t}[u^{\pm 1}]^{a-1}\bigr)\bigr).
\label{chi(2)(caln(K)_otimes_iD(g)}
\end{multline}
We get from Lemma~\ref{lem:phi-twisted_as_ordinary_L2_Euler_characteristic}
and the chain complex version of~\cite[Theorem~6.80~(2) on page~277]{Lueck(2002)}
applied to the short exact sequences of $\caln(K)$-chain complexes obtained by applying
$\caln(K) \otimes_{\IZ K} i^*-$ to the short exact sequences of $\IZ G$-chain complexes
$0 \to P_* \to C_*(\overline{X}) \to \Sigma^2 D(\mu(x_j'))_*\to 0$
and $0 \to D(\mu(x_1))_* \to P_* \to Q_*\to 0$
\begin{eqnarray*}
\lefteqn{\chi^{(2)}(\overline{M};\mu,\phi)}
& & 
\\
& = & 
\chi^{(2)}(\caln(K) \otimes_{\IZ K} i^*C_*(\overline{X}))
\\
& = & 
\chi^{(2)}(\caln(K) \otimes_{\IZ K} i^*Q_*) + \chi^{(2)}(\caln(K) \otimes_{\IZ K} i^*D(\mu(x_j'))_*) 
\\ 
& & \quad \quad \quad \quad + 
\chi^{(2)}(\caln(K) \otimes_{\IZ K} i^*D(\mu(x_i))_*) 
\\
& \stackrel{\eqref{chi(2)(D(g)} \,\text{and} \,\eqref{chi(2)(caln(K)_otimes_iD(g)}}{=}& 
- \dim_{\cald(K)}\bigl(\coker\bigl(r_A \colon \cald(K)_{t}[u^{\pm 1}]^{a-1} \to \cald(K)_{t}[u^{\pm 1}]^{a-1}\bigr)\bigr)
\\
& & 
\quad \quad \quad \quad + |\phi \circ \mu(x_i)| + |\phi \circ \mu(x_j')|.
\end{eqnarray*}
This finishes the proof of Theorem~\ref{the:calculation_of_(mu,phi)_L2-Euler_characteristic_from_a_presentation}.
\end{proof}

\begin{lemma}\label{lem:b_1(2)_is_zero_implies_L2-acyclic}
\begin{enumerate}[font=\normalfont]
\item\label{lem:b_1(2)_is_zero_implies_L2-acyclic:b_1}
Let $M$ be an admissible $3$-manifold. Consider an infinite group $G$ 
and a $G$-covering $\overline{M} \to M$. Then we get $b_n^{(2)}(\overline{M}; \caln(G)) = 0$ for
all $n \ge 0$ if $b_1(\overline{M};\caln(G)) = 0$;

\item\label{lem:b_1(2)_is_zero_implies_L2-acyclic:universal}
If $M$ is an admissible $3$-manifold,
then we get $b_n^{(2)}(\widetilde{M};\caln(\pi)) = 0$ for all $n \ge 0$.
\end{enumerate}

\end{lemma}
\begin{proof}~\eqref{lem:b_1(2)_is_zero_implies_L2-acyclic:b_1}
Since $G$ is infinite, we have $b_0^{(2)}(\overline{M};\caln(G)) = 0$
by~\cite[Theorem~1.35~(8) on page~38]{Lueck(2002)}. Using Poincar\'e duality in the closed case,
see~\cite[Theorem~1.35~(3) on page~37]{Lueck(2002)} we conclude
$b_m^{(2)}(\overline{M};\caln(G)) = 0$ for $m \ge 3$. Since $\chi(M) = 0$, we get
$b_1^{(2)}(\overline{M};\caln(G)) = b_2^{(2)}(\overline{M};\caln(G))$. Hence the
assumption $b_1^{(2)}(\overline{M};\caln(G)) = 0$ implies that we have
$b_m^{(2)}(\overline{M};\caln(G)) = 0$ for all $m \ge 0$. 
\\[1mm]~\eqref{lem:b_1(2)_is_zero_implies_L2-acyclic:universal}
This follows from~\cite[Theorem~0.1]{Lott-Lueck(1995)}.
\end{proof}

\begin{theorem}[The $(\mu,\phi)$-$L^2$-Euler characteristic in terms of the first homology]
\label{the:The_(mu,phi)-L2-Euler_characteristic_in_terms_of_the_first_homology} 
 Let $M$ be an admissible $3$-manifold. Let $G$ be a torsion-free group which satisfies the Atiyah
 Conjecture. Consider group homomorphisms $\mu
 \colon \pi \to G$ and $\phi \colon G \to \IZ$. Let $K$ be the kernel of $\phi$ and $i
 \colon K \to G$ be the inclusion. Let $t \colon \cald(K) \xrightarrow{\cong} \cald(K)$ 
 be the automorphism introduced in 
Theorem~\ref{the:Main_properties_of_cald(G)}~\eqref{the:Main_properties_of_cald(G):cald(K)_and_(cald(G)}. 
 Assume that $\phi$ is surjective, $\phi \circ \mu$ is not trivial, and the intersection $\im(\mu) \cap \ker(\phi)$ is not trivial.
 Suppose that $b_1^{(2)}(\overline{M};\caln(G)) = 0$ holds for the $G$-covering 
 $\overline{M} \to M$ associated to $\mu$. Then:

\begin{enumerate} [font=\normalfont]

\item\label{the:The_(mu,phi)-L2-Euler_characteristic_in_terms_of_the_first_homology:acyclic}

The pair $(\mu,\phi)$ is an $L^2$-acyclic Atiyah-pair;

\item\label{the:The_(mu,phi)-L2-Euler_characteristic_in_terms_of_the_first_homology:homology} 
We have
\begin{multline*}
\hspace{12mm} \chi^{(2)}(M;\mu,\phi) = - b_1^{(2)}(i^*\overline{M};\caln(K))
\\
= - \dim_{\cald(K)}\bigl(H_1\bigl(\cald(K)_{t}[u^{\pm 1}] \otimes_{\IZ G} C_*(\overline{M})\bigr)\bigr).
\end{multline*}
In particular $\chi^{(2)}(M;\mu,\phi)$ is an integer $\ge 0$.
Moreover, we have for $m \not= 1$
\[
H_m\bigl(\cald(K)_{t}[u^{\pm 1}] \otimes_{\IZ G} C_*(\overline{M})\bigr) \,\, =\,\, 0,
\]
and
\[
\hspace{1cm}
b_m^{(2)}(i^*\overline{M};\caln(K))
\,\,=\,\, \dim_{\cald(K)}\bigl(H_m\bigl(\cald(K)_{t}[u^{\pm 1}] \otimes_{\IZ G} C_*(\overline{M})\bigr)\bigr) \,\,=\,\, 0.
\]

\end{enumerate}
\end{theorem}

\begin{proof}~\eqref{the:The_(mu,phi)-L2-Euler_characteristic_in_terms_of_the_first_homology:acyclic}
This follows from 
Lemma~\ref{lem:b_1(2)_is_zero_implies_L2-acyclic}~\eqref{lem:b_1(2)_is_zero_implies_L2-acyclic:b_1}.
\\[1mm]~\eqref{the:The_(mu,phi)-L2-Euler_characteristic_in_terms_of_the_first_homology:homology} 
Consider any presentation
\[
\pi = \langle x_1,x_2, \ldots, x_a \mid R_1,R_2, \ldots, R_b \rangle
\]
of $\pi$. We want to modify it to another presentation 
\[
\pi = \langle \widehat{x}_1, \widehat{x}_2, \ldots, \widehat{x}_{a}\mid \widehat{R}_1, \widehat{R}_2, \ldots, \widehat{R}_{b} \rangle
\]
by a sequence of Nielsen transformations,
i.e., we replace the ordered set of generators 
$x_1, x_2, \ldots, x_a$ by $x_{\sigma(1)}, x_{\sigma(2)}, \ldots, x_{\sigma(a)}$
for a permutation $\sigma \in S_a$ or we replace
 $x_1, x_2, \ldots, x_a$ by $x_1, x_2, x_{i-1}, x_j^kx_i, x_{i+1}, \ldots , x_{j-1}, x_j,x_{j+1}, \ldots, x_a$ 
for some integers $i,j,k$ with $1 \le i < j \le a$, and change the relations accordingly,
such that $\mu(\widehat{x}_1) \not= 0$ and $(\phi\circ \mu)(\widehat{x}_1) = 0$
holds, provided that $\phi$ is surjective and $\im(\phi) \cap \ker(\phi)$ is non-trivial. 
We use induction over $n = |\{i \in \{1,2, \ldots, a\} \mid \mu(x_i) \not= 0\}|$. The induction beginning 
$n = 0,1$ is obvious since the case $n = 0$ cannot occur, since $\phi \circ \mu \not= 0$ and
$\phi$ is surjective, and in the case $n = 1$ we must have $\phi \circ \mu(x_i) = 0$ for the only index
$i$ with $\mu(x_i) \not= 0$, since $\im(\mu) \cap \ker(\phi) \not= \{0\}$.
The induction step from $(n-1)$ to $n \ge 2$ is done as follows.
We use subinduction over $m = \min\{|(\phi\circ \mu)(x_i)| \mid i = 1,2, \ldots, a, \mu(x_i)\not= 0\}$. 
If $m = 0$, then $\mu(x_i) \not = 0$ and $(\phi\circ \mu)(x_i) = 0$ for some $i \in
\{1,2, \ldots, a\}$, and we can achieve our goal by changing the enumeration. The induction step from
$(m-1)$ to $ m \ge 1$ is done as follows. Choose $j \in \{1,2, \ldots, a\}$ such that
$\mu(x_j) \not= 0$ and $(\phi\circ \mu)(x_j) = m$. Since $n \ge 2$, there must be an index 
$i \in \{1,2, \ldots , a\}$ with $i \not = j$ and $\mu(x_i) \not= 0$. By changing the enumeration
we can arrange $i < j$. Choose integers $k,l$ with
$0 \le l < m$ and $(\phi\circ \mu)(x_i) = k \cdot (\phi\circ \mu)(x_j) + l$. If we
replace $x_i$ by $x_i \cdot x_j^{-k}$ and leave the other elements in $\{x_1, x_2, \ldots,x_a\}$ 
fixed, we get a new set of generators which is part of a finite presentation with
$a$ generators and $b$ relations of $\pi$ for which the induction hypothesis applies since
either $\mu(x_i \cdot x_j^{-k}) = 0$ holds or we have $\mu(x_i \cdot x_j^{-k}) \not= 0$
and $|(\phi\circ \mu)(x_i \cdot x_j^{-k})| = l \le m-1$.

We have the following equations in $\IZ \pi$ for $u\in \IZ \pi$
and integers $i,j,k$ with $1 \le i < j \le a$ and $k \ge 1$:
\begin{multline*}
r_{x_j^{k}x_i - 1}(u) 
= 
u(x_j^kx_i - 1) 
= 
ux_j^k(x_i - 1) + u(x_j^k - 1)
= 
ux_j^k(x_i - 1) + u\left(\tmsum{h = 0}{k-1} x_j^h\right) \cdot (x_j - 1)
\\
= 
ux_j^k(x_i - 1) + \left(\tmsum{h = 0}{k-1} u x_j^h\right) \cdot (x_j - 1)
= 
r_{x_i - 1}(ux_j^{k}) - r_{x_j -1} \left(\tmsum{h = 0}{k-1} u x_j^{h}\right),
\end{multline*}
and 
\begin{multline*}
r_{x_j^{-k}x_i - 1}(u) 
= 
u(x_j^{-k}x_i - 1) 
= 
ux_j^{-k}(x_i - 1) - ux_j^{-k} (x_j^k- 1)
\\
= 
ux_j^{-k}(x_i - 1) - ux_j^{-k}\left(\tmsum{h = 0}{k-1} x_j^h\right) \cdot (x_j - 1)
\\
= 
ux_j^{-k}(x_i - 1) - \left(\tmsum{h = 0}{k-1} u x_j^{h-k}\right) \cdot (x_j - 1)
= 
r_{x_i - 1}(ux_j^{-k}) - r_{x_j -1} \left(\tmsum{h = 0}{k-1} u x_j^{h-k}\right).
\end{multline*}
This implies that we can find $\IZ \pi$-isomorphisms 
$f_1 \colon \IZ \pi^a\to \IZ \pi^a$ and $f_2 \colon \IZ \pi^a\to \IZ \pi^a$ 
such that the following diagrams commute
\begin{center}
$\begin{array}{rcl}
\hspace{7mm} \xymatrix@R0.45cm@!C=8em{
\IZ \pi^a \ar[dd]_{f_1} ^{\cong} \ar[rd]^{\quad \bigoplus_{h = 1}^a r_{x_h-1}} &
\\
& \IZ \pi
\\
\IZ \pi^a \ar[ru]_{\quad \bigoplus_{h = 1}^a r_{\widehat{x}_h-1}} &
}
& \hspace{20mm} &
\xymatrix@R0.45cm@!C=8em{
& 
\IZ \pi^a \ar[dd]_{f_2}^{\cong} &
\\
\IZ \pi \ar[ru]^{\prod_{h = 1}^a r_{x_i'-1}} \ar[rd]_{\prod_{h = 1}^a r_{\widehat{x}_i'-1}} & 
\\
& 
\IZ \pi^a &
}
\end{array}$
\end{center}

Now we proceed as in the proof of 
Theorem~\ref{the:calculation_of_(mu,phi)_L2-Euler_characteristic_from_a_presentation}~%
\eqref{the:calculation_of_(mu,phi)_L2-Euler_characteristic_from_a_presentation:empty},
explaining only the case where the boundary is empty. There we had constructed a compact $3$-dimensional $CW$-complex $X$ together with a homotopy equivalence
 $f \colon X \to M$ and set of generators $\{x_1, \ldots ,x_a\}$ and $\{x_1', \ldots ,x_a'\}$ such that
 the $\IZ G$-chain complex $C_*(\overline{X})$ of $\overline{X}$ looks like 
\[
\IZ G 
\xrightarrow{\prod_{i=1}^a r_{\mu(x_i') - 1}} 
\bigoplus_{i=1}^a \IZ G 
\xrightarrow{r_{\mu(F)}} \bigoplus_{i=1}^a \IZ G 
\xrightarrow{\bigoplus_{i=1}^a r_{\mu(x_i) - 1}} \IZ G.
\]
Applying the construction above yields new set of generators 
$\{\widehat{x}_1, \ldots ,\widehat{x}_a\}$ and $\{\widehat{x}_1', \ldots ,\widehat{x}_a'\}$
such that the order of $\mu(\widehat{x}_1')$ is infinite,
and $\phi \circ \mu(\widehat{x_1}) = \phi \circ \mu(\widehat{x}_1') = 0$ 
and we can find a matrix $B \in M_{a,a}(\IZ G)$ and isomorphisms $f_2$ and $f_1$ making the following diagram
of $\IZ G$-modules commute.
\[
\xymatrix@!C=8em{
\IZ G \ar[r]^{\prod_{j=1}^a r_{\mu(x_i') - 1}} \ar[d]^{\id}
&
\bigoplus_{i=1}^a \IZ G \ar[r]^{r_{\mu(F)}} \ar[d]^{f_2}_{\cong}
&
\bigoplus_{i=1}^a \IZ G 
\ar[r]^{\bigoplus_{i=1}^a r_{\mu(x_i) - 1}} \ar[d]^{f_1}_{\cong}
&\IZ G \ar[d]^{\id}
\\
\IZ G \ar[r]_{\prod_{j=1}^a r_{\mu(\widehat{x}_i') - 1}} 
&
\bigoplus_{i=1}^a \IZ G \ar[r]_{r_{B}} 
&
\bigoplus_{i=1}^a \IZ G 
\ar[r]_{\bigoplus_{i=1}^a r_{\mu(\widehat{x}_i) - 1}} 
&\IZ G
}
\]
Now proceed as in the proof of
Theorem~\ref{the:calculation_of_(mu,phi)_L2-Euler_characteristic_from_a_presentation}~%
\eqref{the:calculation_of_(mu,phi)_L2-Euler_characteristic_from_a_presentation:empty}
using the chain complex given by the lower row instead of the chain complex 
given by the upper row in the diagram above. If $A$ is the square matrix over $\IZ G$
obtained from $B$ by deleting the first row and the first
column and then applying the ring homomorphism $\IZ \pi\to \IZ G$ induced by $\mu$
to each entry, then $A$ is invertible over $\cald(G)$ and we get because of 
$\phi \circ \mu(\widehat{x_1}) = \phi \circ \mu(\widehat{x}_1') = 0$ and Lemma~\ref{lem:degree_and_rank}
\begin{multline*}
\chi^{(2)}(M;\mu,\phi) 
= 
 - \dim_{\cald(K)}\bigl(\coker\bigl(r_A \colon \cald(K)_{t}[u^{\pm 1}]^{a-1} \to \cald(K)_{t}[u^{\pm 1}]^{a-1}\bigr)\bigr)
\\
= \dim_{\cald(K)}\bigl(H_1\bigl(\cald(K)_{t}[u^{\pm 1}] \otimes_{\IZ G} C_*(\overline{M})\bigr)\bigr),
\end{multline*}
and 
\begin{multline*}
0 = 
 \dim_{\cald(K)}\bigl(\ker\bigl(r_A \colon \cald(K)_{t}[u^{\pm 1}]^{a-1} \to \cald(K)_{t}[u^{\pm 1}]^{a-1}\bigr)\bigr).
\\
= \dim_{\cald(K)}\bigl(H_2\bigl(\cald(K)_{t}[u^{\pm 1}] \otimes_{\IZ G} C_*(\overline{M})\bigr)\bigr),
\end{multline*}
and
\[
0 = \dim_{\cald(K)}\bigl(H_m\bigl(\cald(K)_{t}[u^{\pm 1}] \otimes_{\IZ G} C_*(\overline{M})\bigr)\bigr)
\quad\text{for}\; m \not = 1,2.
\]
We conclude from Theorem~\ref{the:Main_properties_of_cald(G)}~\eqref{the:Main_properties_of_cald(G):chain} 
for all $m \ge 0$
\[
b_m^{(2)}(i^*\overline{M};\caln(K))
= \dim_{\cald(K)}\bigl(H_m\bigl(\cald(K)_{t}[u^{\pm 1}] \otimes_{\IZ G} C_*(\overline{M})\bigr)\bigr).
\]
This finishes the proof of Theorem~\ref{the:The_(mu,phi)-L2-Euler_characteristic_in_terms_of_the_first_homology}.
\end{proof}

In view of Lemma~\ref{lem:reduction_to_surjective_mu}~\eqref{lem:reduction_to_surjective_mu:mu_circ_phi_is_not_trivial},
the next lemma   covers the case $\im(\mu) \cap \ker(\phi) = \{1\}$
which is not treated in Theorem~\ref{the:The_(mu,phi)-L2-Euler_characteristic_in_terms_of_the_first_homology}.

\begin{lemma}\label{lem:infinite_cyclic_covering}
Let $M$ be an admissible $3$-manifold. Consider an epimorphism $\mu \colon \pi \to \IZ$. Let $\overline{M} \to M$ be the infinite cyclic covering associated to $\mu$.
Finally let $\phi \colon \IZ\to \IZ$ be a non-trivial homomorphism.
If $b_1^{(2)}(\overline{M};\caln(\IZ)) = 0$, then 
\[
\chi(\overline{M}) 
\,\,= \,\,
\begin{cases}
1 - \dim_{\IQ}(H_1(\overline{M};\IQ)) & \text{if}\; \partial M \not= \emptyset;
\\
2 - \dim_{\IQ}(H_1(\overline{M};\IQ)) & \text{if}\; \partial M = \emptyset.
\end{cases}
\]
\end{lemma}

\begin{proof}
Let $k$ be the index of $\im(\phi)$ in $\IZ$.
Note that  $(\mu,\phi)$ is an $L^2$-acyclic pair if and only if $b_1^{(2)}(\overline{M};\caln(\IZ)) = 0$.
This follows from Lemma~\ref{lem:b_1(2)_is_zero_implies_L2-acyclic} and the fact that $\IZ$ satisfies
the Atiyah Conjecture by 
Theorem~\ref{the:Status_of_the_Atiyah_Conjecture}~\eqref{the:Status_of_the_Atiyah_Conjecture:Linnell}.

Suppose $b_1^{(2)}(\overline{M};\caln(\IZ)) = 0$.
An Euler characteristic argument and standard facts on $L^2$-Betti numbers imply
that $b_n^{(2)}(\overline{M};\caln(\IZ)) = 0$ for all $n\geq 0$.
 Lemma~\ref{lem:phi-twisted_as_ordinary_L2_Euler_characteristic} implies that
$\dim_{\IQ}(H_n(\overline{M};\IQ)) < \infty$ holds for all $n \geq 0$.
Note that $H_n(\overline{M};\IQ) =0$ holds for $n \ge 3$ holds because $\overline{M}$ is a non-compact $3$-manifold.
Thus it remains to show the following equality:
\[
\dim_{\IQ}(H_2(\overline{M};\IQ)) 
\,\,=\,\, 
\begin{cases} 0
& \text{if}\; \partial M \not= \emptyset;
\\
1 & \text{if}\; \partial M = \emptyset.
\end{cases}
\]
We begin with the case that $M$ has
non-empty boundary. Then $M$ is homotopy equivalent to a $2$-dimensional complex $X$.
Since $H_2(\overline{M};\IQ) \cong_{\IQ[\IZ]} H_2(\overline{X};\IQ)$ is a $\IQ[\IZ]$-submodule of the free $\IQ[\IZ]$-module
$C_2(\overline{X}) \otimes_{\IZ} \IQ$ for the pull back $\overline{X} \to X$ of $\overline{M} \to M$ with a homotopy
equivalence $X \to M$, the $\IQ[\IZ]$-module $H_2(\overline{M};\IQ)$ is free.
Since $\dim_{\IQ}(H_2(\overline{M};\IQ)) < \infty$ holds, this implies $H_2(\overline{M};\IQ) = 0$.

Now suppose that $\partial M$ is empty. Then we get a Poincar\'e $\IZ \pi$-chain homotopy equivalence of
$\IZ \pi$-chain complexes $\Hom_{\IZ \pi}(C_{3-*}(\widetilde{M}),\IZ \pi) \to C_*(\widetilde{M})$.
It induces a $\IQ[\IZ]$-chain homotopy equivalence $\Hom_{\IQ[\IZ]}(\IQ \otimes_{\IZ} C_{3-*}(\overline{M}), \IQ[\IZ]) 
\to \IQ \otimes_{\IZ} C_{*}(\overline{M})$ and hence a $\IQ[\IZ]$-isomorphism
\[H^1\bigl(\Hom_{\IQ[\IZ]}(\IQ \otimes_{\IZ} C_{3-*}(\overline{M}), \IQ[\IZ])\bigr) \xrightarrow{\cong} H_2(\overline{M}).
\]
Since $\IQ[\IZ]$ is a principal ideal domain, we get from the Universal Coefficient Theorem for cohomology an exact sequence 
of $\IQ[\IZ]$-modules
\begin{multline*}
0 \to \Ext^1_{\IQ[\IZ]}\bigl(H_0(C_*(\overline{M})),\IQ[\IZ]\bigr) \to H^1\bigl(\Hom_{\IQ[\IZ]}(\IQ \otimes_{\IZ} C_{3-*}(\overline{M}), \IQ[\IZ])\bigr) 
\\
\to \Hom_{\IQ[\IZ]}\bigl(H_1(\overline{M};\IQ),\IQ[\IZ]\bigr) \to 0.
\end{multline*}
From $\dim_{\IQ}(H_1(\overline{M};\IQ)) < \infty$, we conclude $\Hom_{\IQ[\IZ]}\bigl(H_1(\overline{M};\IQ),\IQ[\IZ]\bigr) = 0$.
Since $H_0(C_*(\overline{M}))$ is the trivial $\IQ[\IZ]$-module $\IQ$, we conclude that 
$\Ext^1_{\IQ[\IZ]}\bigl(H_0(C_*(\overline{M})),\IQ[\IZ]\bigr)$ is the trivial $\IQ[\IZ]$-module $\IQ$
and $H_2(\overline{M};\IQ)\cong \IQ$ follows.
\end{proof}


\typeout{------ Section 6: Equality of $(\mu,\phi)$-$L^2$-Euler characteristic and the Thurston norm ------}

\section{Equality of $(\mu,\phi)$-$L^2$-Euler characteristic and the Thurston norm}
\label{sec:Equality_of_(mu,phi)-L2-Euler_characteristic_and_the_Thurston_norm}

The section is devoted to the proof of
Theorem~\ref{the:Equality_of_(mu,phi)-L2-Euler_characteristic_and_the_Thurston_norm}
which needs some preparation.

For the remainder of this section $G$ is a torsion-free group
satisfying the Atiyah Conjecture, $H$ is a finitely generated abelian
group and $\nu \colon G \to H$ a surjective group homomorphism.


\subsection{The non-commutative Newton polytope}
\label{subsec:The_non-commutative_Newton_polytope}

Choose a set-theoretic section $s$ of $\nu$, i.e., a map of sets $s \colon H \to G$ with $\nu \circ s = \id_H$.
Notice that we do \emph{not} require that $s$ is a group homomorphism. Let $K_{\nu}$ be
the kernel of $\nu$. Then $K_{\nu}$ also satisfies the Atiyah Conjecture
by Theorem~\ref{the:Status_of_the_Atiyah_Conjecture}~\eqref{the:Status_of_the_Atiyah_Conjecture:subgroups},
and $\cald(K_{\nu})$ and $\cald(G)$ are skew fields by
Theorem~\ref{the:Main_properties_of_cald(G)}~\eqref{the:Main_properties_of_cald(G):skew_field}.
There are obvious inclusions $\IZ K_{\nu} \subseteq \IZ G$ and $\cald(K_{\nu}) \subseteq \cald(G)$ 
coming from the inclusion $j \colon K_{\nu} \to G$. Let $c \colon H \to \aut(\cald(K_{\nu}))$ be 
the map sending $h$ to the automorphisms $c(h) \colon \cald(K_{\nu}) \xrightarrow{\cong} \cald(K_{\nu})$
induced by the conjugation automorphism $K_{\nu} \to K_{\nu}, \; k \mapsto s(h) \cdot k \cdot s(h)^{-1}$. Define
$c \colon H \to \aut(\IZ K_{\nu})$ analogously. Define $\tau \colon H \times H \to (\IZ K_{\nu})^{\times}$ 
and $\tau \colon H \times H \to \cald(K_{\nu})^{\times}$ by sending $(h,h')$
to $s(h)\cdot s(h')\cdot s(hh')^{-1}$. Let $\IZ K_{\nu} \ast_s H$ and $\cald(K_{\nu}) \ast_s H$ be
the crossed product rings associated to the pair $(c,\tau)$. Elements in $\IZ K_{\nu} \ast_s H$
and $\cald(K_{\nu}) \ast_s H$ respectively are finite formal sums $\sum_{h \in H} x_h \cdot h$ 
for $x_h$ in $\IZ K_{\nu}$ and $\cald(K_{\nu})$ respectively. Addition is given by adding the
coefficients. Multiplication is given by the formula
\[
\Big(\,\smsum{h \in H}{} x_{h} \cdot h\Big) \cdot \Big(\,\smsum{h \in H}{} y_{h} \cdot h\Big)\,\, =\,\,\smsum{h\in H}{} \Big(\,\,\smsum{\substack{h',h'' \in H,\\h'h'' = h}}{} x_{h'} c_{h'}(y_{h''})
 \tau(h',h'') \Big) \cdot h.
\]
This multiplication is uniquely determined by the properties $h\cdot x = c(h)(x)\cdot h$
for $x$ in $\IZ K_{\nu}$ and $\cald(K_{\nu})$ respectively, and $h \cdot h' = \tau(h,h') \cdot
(hh')$ for $h,h' \in H$. We obtain an isomorphism of rings
\[
j'_s \colon \IZ K_{\nu} \ast_s H \,\xrightarrow{\cong}\, \IZ G, \quad \smsum{h \in H}{} x_h \cdot h
\,\mapsto\, \smsum{h \in H}{} x_h \cdot s(h),
\]
see~\cite[Example~10.53 on page~396]{Lueck(2002)}, and an injective ring homomorphism
\[
j_s \colon \cald(K_{\nu}) \ast_s H \,\to\, \cald(G), \quad \smsum{h \in H}{} x_h \cdot h
\,\mapsto\, \smsum{h \in H}{} x_h \cdot s(h).
\]
Moreover, the set $T$ of non-trivial elements in $\cald(K_{\nu}) \ast_s H$ satisfies the Ore
condition and $j_s$ induces an isomorphism of skew fields
\begin{equation}
\widehat{j_s} \colon T^{-1} (\cald(K_{\nu}) \ast_s H) \xrightarrow{\cong} \cald(G).
\label{widehat(j_s)}
\end{equation}
This is proved for $F = \IC$ in~\cite[Lemma~10.68 on page~403]{Lueck(2002)}, the proof
carries over for any field $F$ with $\IQ \subseteq F \subseteq \IC$.

Fix an element $u = \sum_{h \in H} x_h \cdot h$ in $\cald(K_{\nu}) \ast_s H$ with $u \not= 0$.
Now we introduce a non-commutative analogue of the Newton polytope. 
A \emph{polytope} in a finite dimensional real vector space is a subset which is the convex hull of a finite subset.
An element $p$ in a polytope is called \emph{extreme} if the implication $p= \frac{q_1}{2} + \frac{q_2}{2}
\Longrightarrow q_1 = q_2 = p$ holds for all elements $q_1$ and $q_2$ in the polytope.
Denote by $\Ext(P)$ the set of extreme points of $P$. If $P$ is the convex hull of the finite set $S$,
then $\Ext(P) \subseteq S$ and $P$ is the convex hull of $\Ext(P)$. 
The \emph{Minkowski sum} of two polytopes $P_1$ and $P_2$ is defined to be the polytope
\[
P_1 + P_2 := \{p_1 +p_2 \mid p_1 \in P_1, p \in P_2\}.
\] 
It is the convex hull of the set $\{p_1 + p_2 \mid p_1 \in \Ext(P_1), p_2 \in \Ext(P_2)\}$.
Define the \emph{non-commutative Newton polytope} of $u$ 
\begin{equation}
P(u) \subseteq \IR \otimes_{\IZ} H
\label{P(u)}
\end{equation}
to be the polytope given by the convex hull of the finite set $\{1 \otimes h \mid x_h \not= 0\}$. 
We now show that the definition of $P(u)$ is independent of the choice of the section $s$ by the following argument.
Consider two set-theoretic sections $s$ and $s'$ of $\nu \colon G \to H$. We denote the corresponding polytopes by $P(u;s)$ and $P(u;s')$.  Then we get for
$u \in \cald(K_{\nu}) \ast_s H$ and $u' \in \cald(K_{\nu}) \ast_{s'} H$ 
\begin{eqnarray}
\widehat{j_s}(u) = \widehat{j_{s'}}(u')
& \implies &
P(u;s) = P(u';s')
\label{P(u;s)_independent_of_s}
\end{eqnarray} 
by the following argument. If we write $u =\sum_{h \in H} x_h \cdot h$ and $u' = \sum_{h \in H} y_h \cdot h$,
then $\widehat{j_s}(u) = \widehat{j_{s'}}(u')$ implies
\[
u\,\, =\,\, \lmsum{h \in H}{} x_h \cdot h \,\,= \,\, \lmsum{h \in H}{} \bigl(y_h\cdot s'(h) s(h)^{-1} \bigr)\cdot h,
\]
and hence $x_h \not = 0 \Leftrightarrow y_h \not= 0$.

The following lemma is well-known in the commutative setting and we explain its proof 
since we could not find a reference for it in the literature.

\begin{lemma}\label{lem:P(uv)_is_P(u)_plus_P(v)} 
For $u,v \in \cald(K_{\nu}) \ast_s H$ with $u,v \not= 0$, we have 
\[
P(uv) = P(u) + P(v).
\]
\end{lemma}
	
\begin{proof}
Consider an extreme point $p \in P(u) + P(v)$. Then we can find points 
$q_1 \in P(u)$ and $q_2 \in P(v)$ with $p = q_1 + q_2$. We want to show that $q_1$ and $q_2$ are extreme. 
Consider $q_1', q_1'' \in P(u)$ and $q_2',q_2'' \in P(v)$ with $q_1 = (q_1' +q_1'')/2$ and $q_2 = (q_2' + q_2'')/2$.
Then $(q_1' + q_2')$, $(q_1'' + q_2')$, $(q_1' + q_2'')$, and $(q_1'' + q_2'')$ belong to $P(u) + P(v)$ and satisfy
\[
p\,\, =\,\, \smfrac{q_1' + q_2'}{2} + \smfrac{q_1''+ q_2''}{2}\,\, =\,\, \smfrac{q_1' + q_2''}{2} + \smfrac{q_1''+ q_2'}{2}.
\]
Since $p \in P(u) + P(v)$ is extreme, we conclude $q_1' + q_2' = q_1''+ q_2''$ and
$q_1' + q_2'' = q_1''+ q_2'$. If we subtract the second equation from the first, we obtain
$q_2' - q_2'' = q_2'' - q_2'$ and hence $q_2' = q_2''$. This implies also $q_1' = q_1''$.
This shows that $q_1 \in P(u)$ and $q_2 \in P(v)$ are extreme.

Suppose that we have other points $q_1' \in P(u)$ and $q_2' \in P(v)$ with $p = q_1' + q_2'$.
Then $q_1 + q_2'$ and $q_1' + q_2$ belong to $P(u) + P(v)$ and satisfy
$p = \frac{q_1 + q_2'}{2} + \frac{q_1'+ q_2}{2}$. Since $p$ is extreme, this implies $p = q_1 + q_2' = q_1'+ q_2$.
Since we also have $p = q_1 + q_2 = q_1' + q_2'$, we conclude $q_1 = q_1'$ and $q_2 = q_2'$.

Now write $u = \sum_{h \in H} x_{h} \cdot h$, $v = \sum_{h \in H} y_{h} \cdot h$, and 
$uv = \sum_{h \in H} z_{h} \cdot h$. Since $p \in P(u) + P(v)$ is extreme
we can write $p$ as the sum of two extreme vertices of $P(u)$ and $P(v)$, which implies that there is $h \in H$ 
with $p = 1 \otimes h$. If we write $p = q_1 + q_2$ for $q_1 \in P(u)$ 
and $q_2 \in P(v)$, then we have already seen that $q_1$ and $q_2$ are extreme
and hence there are $h_1,h_2 \in H$ with $q_1 = 1 \otimes h_1$, $q_2 = 1 \otimes h_2$,
$x_{h_1} \not= 0$ and $y_{h_2} \not= 0$. The equation $p = q_1 + q_2$ implies
$h = h_1 + h_2$. Now consider elements $h_1', h_2' \in H$ with $h = h_1' + h_2'$, $x_{h_1'} \not= 0$ 
and $y_{h_2'} \not= 0$. Put $q_1' = 1 \otimes h_1'$ and $q_2' = 1 \otimes h_2'$. Then $q_1' \in P(u)$ 
and $q_2' \in P(v)$ and we have $p = q_1' + q_2'$. We have already explained that this implies 
$q_1 = q_1'$ and $q_2 = q_2'$ and hence $h_1 = h_1'$ and $h_2 = h_2'$. Therefore we get
$z_h = x_{h_1} c(h_1)(y_{h_2}) \tau(h_1,h_2)$. Since $x_{h_1}$ and $y_{h_2}$ are non-trivial, we
conclude $z_h \not= 0$ and hence $p \in P(uv)$. Hence every extreme point in 
$P(u) + P(v)$ belongs to $P(uv)$ which implies $P(u) + P(v) \subseteq P(uv)$.

One easily checks that any point of the shape $1 \otimes h$ for $z_h \not= 0$ belongs
to $P(u) + P(v)$ since $z_h \not= 0$ implies the existence of $h_1$ 
and $h_2$ with $x_{h_1}, y_{h_2} \not= 0$ and $h = h_1 + h_2$. We conclude $P(uv) \subseteq P(u) + P(v)$.
This finishes the proof of Lemma~\ref{lem:P(uv)_is_P(u)_plus_P(v)}.
\end{proof}


\subsection{The polytope homomorphism}
\label{subsec:The_polytope_homomorphism}

We obtain a finite dimensional real vector space $\IR \otimes_{\IZ} H$. An integral polytope in 
$\IR \otimes_{\IZ} H$ is a polytope such that $\Ext(P)$ is contained in $H$, where we consider $H$ 
as a lattice in $ \IR \otimes_{\IZ} H$ by the standard embedding 
$H \to \IR \otimes_{\IZ} H, \; h \mapsto 1 \otimes h$. 
The Minkowski sum of two integral polytopes is again an integral
polytope. Hence the integral polytopes form an abelian monoid under the Minkowski sum with
the integral polytope $\{0\}$ as neutral element.

\begin{definition}[Grothendieck group of integral polytopes]
\label{def:The_Grothendieck_group_of_integral_polytope}
Let $\calp_{\IZ}(H)$ be the abelian group given by the Grothendieck construction applied
to the abelian monoid of integral polytopes in $\IR \otimes_{\IZ} H$ under the Minkowski sum.
\end{definition}

Notice that for polytopes $P_0$, $P_1$ and $Q$ in a finite dimensional real vector space
we have the implication $P_0 + Q = P_1 + Q \Longrightarrow P_0 = P_1$, see~\cite[Lemma~2]{Radstroem(1952)}.
Hence elements in $\calp_{\IZ}(H)$ are given by formal
differences $[P] - [Q]$ for integral polytopes $P$ and $Q$ in $\IR \otimes_{\IZ} H$ and we
have $[P_0] - [Q_0] = [P_1] - [Q_1] \Longleftrightarrow P_0 + Q_1 = P_1 + Q_0$.
As a side remark we point out that the polytope group $\calp_{\IZ}(H)$ is a free abelian group by Funke~\cite{Funke(2016)}.

In the sequel we denote by $G_{\abel}$ the abelianization
$G/[G,G]$ of a group $G$. Define the \emph{polytope homomorphism} for a surjective homomorphism
$\nu \colon G \to H$ onto a finitely generated free abelian group $H$ 
\begin{eqnarray}
\IP'_{\nu} \colon \cald(G)^{\times}_{\abel} \to \calp_{\IZ}(H).
\label{P_prime_colon_cald_to_p_Z}
\end{eqnarray}
as follows. Choose a set-theoretic section $s$ of $\nu$. Consider an element $z \in \cald(G)$ with $z \not= 0$.
Choose $u,v \in \cald(K_{\nu}) \ast_s H$ such that $\widehat{j_s}(uv^{-1}) = z$,
where the isomorphism $\widehat{j_s}$ has been introduced in~\eqref{widehat(j_s)}.
Then we define the image of the class $[z]$
in $\cald(G)^{\times}_{\abel}$ represented by $z$ under $\IP'_{\nu}$ to be $[P(u)]- [P(v)]$.

We have to show that this is independent of the choices of $s$,$u$ and $v$. 
Suppose that we have another set-theoretic section $s' \colon H \to G$ of $\nu$ and 
$u',v' \in \cald(K_{\nu}) \ast_{s'} H$ with $u',v' \not= 0$ and $z = \widehat{j_{s'}}(u'v'^{-1})$.
For any $h\in H$ we have $s'(h)=s(h)\cdot k$ for some $k\in K_\nu$. It follows that 
$\im(j_s)=\im(j_{s'})$. In particular there exist unique $u'',v'' \in \cald(K_{\nu}) \ast_{s'} H$ 
with $j_s(u) = j_{s'}(u'')$ and $j_s(v) = j_{s'}(v'')$.
From $\widehat{j_{s}}(uv^{-1}) = z = \widehat{j_{s'}}(u'v'^{-1})$
we conclude $u'v'^{-1} = u''v''^{-1}$ in $T^{-1} \cald(K_{\nu}) \ast_{s'} H$. Hence there exist
$w',w'' \in \cald(K_{\nu}) \ast_{s'} H$ with $w',w'' \not= 0$, $u' w' = u'' w''$ and $v'w' = v'' w''$.
We conclude 
\begin{eqnarray*}
P(u) - P(v)
& \stackrel{\eqref{P(u;s)_independent_of_s}}{=}& 
P(u'') - P(v'')
\\
& = & 
P(u'') + P(w'') - P(w') - P(v'') - P(w'') + P(w') 
\\
& \stackrel{\textup{Lemma}~\ref{lem:P(uv)_is_P(u)_plus_P(v)}}{=} & 
P(u''w'') - P(w') - P(v''w'') + P(w') 
\\
& = & 
P(u'w') - P(w') - P(v'w') + P(w') 
\\
& \stackrel{\textup{Lemma}~\ref{lem:P(uv)_is_P(u)_plus_P(v)}}{=} & 
P(u') + P(w') - P(w') - P(v') - P(w') + P(w') 
\\
& = & 
P(u') - P(v'). 
\end{eqnarray*}
Hence $\IP'_{\nu} \colon \cald(G)^{\times} \to \calp_{\IZ}(H)$ is well-defined. 
We conclude from Lemma~\ref{lem:P(uv)_is_P(u)_plus_P(v)} that $\IP'_{\nu}$
 is a homomorphism of abelian groups.


\subsection{Semi-norms and matrices over $\cald(K_{\phi \circ \nu})_{t}[u^{\pm 1}$]}
\label{subsec:Semi-norms_and_matrices_over_cald(K_phi_circ_nu)_t(u,u(-1))}

Let $P \subseteq \IR \otimes_{\IZ} H$ be a polytope. It defines a seminorm on 
$\Hom_{\IZ}(H,\IR) = \Hom_{\IR}(\IR \otimes_{\IZ} H,\IR)$ by
\begin{eqnarray}
 \|\phi \|_P 
&:= & \tmfrac{1}{2}
\sup\{\phi(p_0) - \phi(p_1) \mid p_0, p_1 \in P\}.
\label{seminorm_of_a_polytope}
\end{eqnarray}
It is compatible with the Minkowski sum, namely, for two
integral polytopes $P,Q \subseteq \IR \otimes_{\IZ} H$ we have
\begin{eqnarray}
 \|\phi \|_{P+ Q} 
&= & 
 \|\phi \|_{P} + \|\phi \|_{Q}.
\label{seminorm_of_a_polytope_is_additive_in_polytope}
\end{eqnarray}
Put 
\begin{multline}
\calsn(H) 
 := 
\{f \colon \Hom_{\IZ}(H;\IR) \to \IR \mid \text{there exists integral polytopes}
\\
\; P \; \text{and} \; Q\; \text{in}\; \IR \otimes_{\IZ} H \;\text{with}\; f = \|\; \|_P - \|\; \|_Q\}.
\label{calsn(H)}
\end{multline}
This becomes an abelian group by $(f-g)(\phi) = f(\phi) - g(\phi)$ because 
of~\eqref{seminorm_of_a_polytope_is_additive_in_polytope}. Again because 
of~\eqref{seminorm_of_a_polytope_is_additive_in_polytope}
we obtain an epimorphism of abelian groups
\begin{equation}
\sn \colon \calp_{\IZ}(H) \to \calsn(H) 
\label{norm_homomorphism}
\end{equation}
by sending $[P] - [Q]$ for
two polytopes $P,Q \subseteq \IR \otimes_{\IZ} H$ to the function
\[
\Hom_{\IZ}(H,\IR) \to \IR, \quad \phi \mapsto \|\phi \|_P - \|\phi \|_Q.
\]

Consider any finitely generated abelian group $H$ and group homomorphisms $\nu \colon G \to H$ and
$\phi \colon H \to \IZ$ such that $\phi$ is surjective. Define a homomorphism
\begin{equation}
D_{\nu,\phi} \colon \cald(G)^{\times}_{\abel} \xrightarrow{\IP'_{\nu} }\calp_{\IZ}(H)
\xrightarrow{\sn} \calsn(H) \xrightarrow{\ev_{\phi}} \IR
\label{D_(nu,phi)}
\end{equation}
to be the composite of the homomorphism defined in~\eqref{P_prime_colon_cald_to_p_Z}
and~\eqref{norm_homomorphism} and the evaluation homomorphism $\ev_{\phi}$.

\begin{lemma}\label{lem:D_(mu,psi_circ_omega_is_D_(omega_circ_nu,psi)}
Consider finitely generated free abelian groups $H$ and $H'$ and surjective group homomorphisms
$\nu \colon G \to H$, $\omega \colon H \to H'$ and an epimorphism $\psi \colon H' \to \IZ$. 
Then we get the following equality of homomorphisms $\cald(G)^{\times}_{\abel} \to \IR$
\[
D_{\nu,\psi \circ \omega} 
= 
D_{\omega \circ \nu,\psi}.
\]
\end{lemma}

\begin{proof}
Choose a set-theoretic section $s \colon H \to G$ of $\nu$ and
a set-theoretic section $t \colon H' \to H$ of $\omega$. Then $s \circ t \colon H' \to G$
is a set-theoretic section of $\omega \circ \nu \colon G \to H'$. Let $K_{\nu} \subseteq G$ 
be the kernel of $\nu$, $K_{\omega \circ \nu} \subseteq G$ be the kernel of 
$\omega \circ \nu$ and $K_{\omega} \subseteq H$ be the kernel of $\omega$. Let 
$k \colon K_{\nu} \to K_{\omega \circ \nu}$ be the inclusion. We obtain an exact sequence 
$0 \to K_{\nu} \xrightarrow{k} K_{\omega \circ \nu} \xrightarrow{\nu|_{K_{\omega \circ \nu}}}
K_{\omega}\to 0$ of groups such that $K_{\omega}$ is finitely generated free abelian. The
section $s$ induces a section $s|_{K_{\omega}} \colon K_{\omega} \to K_{\omega \circ \nu}$ 
of $\nu|_{K_{\omega \circ \nu}} \colon K_{\omega \circ \nu} \to K_{\omega}$. We also have the exact sequence 
$0 \to K_{\omega \circ \nu} \xrightarrow{l} G \xrightarrow{\omega \circ \nu} H' \to 0$ 
for $l$ the inclusion and the set-theoretic section $s \circ t$ of $\omega \circ \nu$. 
Thus we get isomorphisms of skew fields
\begin{eqnarray*}
\widehat{j_s} \colon T^{-1} \cald(K_{\nu}) \ast_s H
& \xrightarrow{\cong} & 
\cald(G);
\\
\widehat{k_{s|_{K_{\omega}}}} \colon T^{-1} \cald(K_{\nu}) \ast_{s|_{K_{\omega}}} K_{\omega}
& \xrightarrow{\cong} & 
\cald(K_{\omega \circ \nu});
\\
\widehat{l_{s\circ t}} \colon T^{-1} \cald(K_{\omega \circ \nu}) \ast_{s \circ t} H'
& \xrightarrow{\cong} & 
\cald(G),
\end{eqnarray*}
where $T^{-1}$ always indicates the Ore localization with respect to the non-trivial elements.
Consider $u = \sum_{h \in H} x_h \cdot h$ in $\cald(K_{\nu}) \ast_s H$. 
For $h' \in H'$ define an element in $\cald(K_{\nu}) \ast_{s|_{K_{\omega}}} K_{\omega}$ by
\[
u_{h'} = \sum_{h \in K_{\omega}} \bigl(x_{h \cdot t(h')} \cdot s(h \cdot t(h') ) \cdot s \circ t(h')^{-1} \cdot s(h)^{-1}\bigr) \cdot h.
\]
It is well-defined since $s(h \cdot t(h') ) \cdot s \circ t(h')^{-1} \cdot s(h)^{-1} \in K_{\nu}$ holds.
Define an element in $\cald(K_{\omega \circ \nu}) \ast_{s \circ t} H'$ by
\[
v = \sum_{h' \in H'} \widehat{k_{s|_{\ker_{\omega}}}} (u_{h'}) \cdot h'.
\]
Then we compute in $\cald(G)$
\begin{eqnarray*}
\widehat{j_s}(u) 
& = & 
\lmsum{h \in H}{} x_h \cdot s(h)
\,\, = \,\, 
\lmsum{h' \in H'}{} \lmsum{h \in \omega^{-1}(h')}{} x_h \cdot s(h)
\\
& = & 
\lmsum{h' \in H'}{} \bigg(\,\,\lmsum{h \in \omega^{-1}(h')}{} x_h \cdot s(h) \cdot s \circ t(h')^{-1} \bigg) \cdot s \circ t(h')
\\
& = & 
\lmsum{h' \in H'}{} \bigg(\,\lmsum{h \in K_{\omega}}{} x_{h \cdot t(h')} \cdot s(h \cdot t(h') ) \cdot s \circ t(h')^{-1} \bigg) \cdot s \circ t(h')
\\
& = & 
\lmsum{h' \in H'}{} \bigg(\,\lmsum{h \in K_{\omega}}{} \bigl(x_{h \cdot t(h')} 
\cdot s(h \cdot t(h') ) \cdot s \circ t(h')^{-1} \cdot s(h)^{-1}\bigr) \cdot s(h) \bigg) \cdot s \circ t(h')
\\
& = & 
\lmsum{h' \in H'}{} \widehat{k_{s|_{\ker(\omega)}}}(u_{h'}) \cdot s \circ t(h')
\,\, = \,\,
\widehat{l_t}\bigg(\,\lmsum{h' \in H'}{} \widehat{k_{s|_{\ker(\omega)}}}(u_{h'}) \cdot h'\bigg)
\,\, = \,\,
\widehat{l_t}(v).
\end{eqnarray*}
Obviously we get for $h' \in H'$
\[u_{h'} \not= 0 \Leftrightarrow \exists h \in \omega^{-1}(h') \; \text{with}\; x_h \not= 0.
\]
This implies
\begin{multline*}
\sup\{\psi(h') - \psi(k') \mid h',k' \in H', u_{h'} \not= 0, u_{k'} \not = 0\}
\\
=
\sup\{\psi \circ \omega(h) - \psi \circ \omega (k) \mid h,k \in H, x_{h} \not= 0, x_{k} \not = 0\}.
\end{multline*}
Hence we get for the element $z \in \cald(G)$ given by $z = \widehat{j_s}(u) = \widehat{l_t}(v)$
\[
D_{\omega \circ \nu,\psi}(z) = D_{\nu,\psi \circ \omega}(z).
\]
Since $D_{\omega \circ \nu,\psi}$ and $D_{\nu,\psi \circ \omega}$ are homomorphisms, 
Lemma~\ref{lem:D_(mu,psi_circ_omega_is_D_(omega_circ_nu,psi)}
follows.
\end{proof}

We recall our setting. We have an abelian group $H$, a group homomorphism $\nu \colon G\to H$ and an epimorphism $\phi\colon H\to \IZ$. 
Recall from Theorem~\ref{the:Main_properties_of_cald(G)}~\eqref{the:Main_properties_of_cald(G):cald(K)_and_(cald(G)} 
that $\cald(G)$ is the Ore localization with respect to the set of non-zero
elements of the ring $\cald(K_{\phi \circ \nu})_{t}[u^{\pm 1}]$ of twisted Laurent polynomials in the variable $u$
with coefficients in the skew-field $\cald(K_{\phi \circ \nu})$. Hence $\cald(K_{\phi \circ
  \nu})_{t}[u^{\pm 1}]$ is contained in $\cald(G)$ and we can consider for any $x \in
\cald(K_{\phi \circ \nu})_{t}[u^{\pm 1}]$ with $x \not= 0$ its image $D_{\nu,\phi}([x])
\in \IR$ under the homomorphism $D_{\nu,\phi}$ defined in~\eqref{D_(nu,phi)}.

\begin{lemma}\label{lem:image_of_an_element_x_under_D_nu,phi}
Consider an element $x \in \cald(K_{\phi \circ \nu})_{t}[u^{\pm 1}]$ with $x \not = 0$.

Then 
\[
D_{\nu,\phi}([x])\, =\,\tmfrac{1}{2} \deg(x),
\]
where $\deg(x)$ has been defined to be $k_+ -k_-$ if we write
$x = \sum_{k = k_-}^{k_+} z_n \cdot u^k$ with $z_{k_+},z_{k_-} \not= 0$.
\end{lemma}

\begin{proof}
Recall that $\cald(K_{\phi \circ \nu})_{t}[u^{\pm 1}]$ does depend on a choice of a preimage of $1$ under $\phi \circ \nu \colon G \to \IZ$
which is the same as a choice of a group homomorphism $\gamma \colon \IZ \to G$ with $\phi \circ \nu \circ \gamma = \id_{\IZ}$.
Choose a set theoretic map $s \colon H \to G$ with $\nu \circ s = \id_H$.
One easily checks that $\cald(K_{\phi \circ \nu})_{t}[u^{\pm 1}]$ agrees with
$\cald(K_{\phi \circ \nu}) \ast_{\gamma} \IZ$. 
We conclude $D_{\nu,\phi}([x]) = D_{\nu \circ \phi,\id_{\IZ}}([x])$ from Lemma~\ref{lem:D_(mu,psi_circ_omega_is_D_(omega_circ_nu,psi)}. 
Now one easily checks $D_{\phi \circ \nu,\id_{\IZ}}([x]) = \frac{1}{2} \cdot \deg(x)$ by inspecting the definitions, since for a polynomial
$\sum_{k = k_-}^{k_+} z_k u^k$ in one variable $u$ with $z_{k_-} ,z_{k_+} \not= 0$ its Newton polytope is the interval
$[k_-,k_+] \subseteq \IR$. (The factor $1/2$ comes from the factor $1/2$ in~\eqref{seminorm_of_a_polytope}.)
\end{proof}

There is a Dieudonn\'e determinant for invertible matrices over a skew field $D$ which takes
values in the abelianization of the group of units $D^{\times}_{\abel}$ 
and induces an isomorphism, see~\cite[Corollary~4.3 in page~133]{Silvester(1981)} 
\begin{eqnarray}
{\det}_D \colon K_1(D) 
& \xrightarrow{\cong} &
D^{\times}_{\abel}
\label{K_1(K)_Dieudonne}
\end{eqnarray}
The inverse 
\begin{eqnarray}
J_D \colon D^{\times}_{\abel} & \xrightarrow{\cong} & K_1(D) 
\label{K_1(K)_Dieudonne_inverse}
\end{eqnarray}
sends the class of a unit in $D$ to the class of the corresponding $(1,1)$-matrix.

\begin{lemma}\label{matrices_overcald(K_phi_circ_nu)_t(u,u(-1))}
Let $A $ be an $n\times n$-matrix over $\cald(K_{\phi \circ \nu})_{t}[u^{\pm 1}]$ 
which becomes invertible over $\cald(G)$.
Then the composite of the homomorphisms defined in~\eqref{D_(nu,phi)} and~\eqref{K_1(K)_Dieudonne} 
\[K_1(\cald(G)) \xrightarrow{\det_{\cald(G)}} \cald(G)^{\times}_{\abel} 
\xrightarrow{D_{\nu,\phi}} \IR
\]
sends the class $[A] \in K_1(\cald(G))$ of $A$ to
\[
\tmfrac{1}{2}\dim_{\cald(K_{\phi \circ \nu})}\bigl(\coker\bigl(r_A \colon \cald(K_{\phi \circ \nu})_{t}[u^{\pm 1}]^n \to \cald(K_{\phi \circ \nu})_{t}[u^{\pm 1}]^n\bigr)\bigr).
\]
\end{lemma}

\begin{proof}
The twisted polynomial ring $\cald(K_{\phi \circ \nu})_{t}[u]$ has a Euclidean function given by
the degree and hence there is a Euclidean algorithm with respect to it. This algorithm
allows to transform $A$ to a diagonal matrix over $\cald(K_{\phi \circ \nu})_{t}[u^{\pm 1}]$ 
by the following operations
\begin{enumerate}
\item 
Permute rows or columns;

\item 
Multiply a row on the right or a column on the left
with an element of the shape $y u^m$ for some $y \in \cald(K_{\phi \circ \nu})$ with $y \not = 0$ and $m \in \IZ$;

\item 
Add a right $\cald(K_{\phi \circ \nu})_{t}[u^{\pm 1}]$-multiple of a row to another row and analogously for columns;
\end{enumerate}
These operations change the class $[A]$ of $A$ in $K_1(\cald(G))$ by adding an element of the shape $J_{\cald(G)}([yu^m])$ for 
$y \in \cald(K_{\phi \circ \nu})$ with $y \not = 0$ and $m \in \IZ$ for the homomorphism 
$J_{\cald(G)}$ of~\eqref{K_1(K)_Dieudonne_inverse}. Moreover, they do not change
neither  the 
kernel nor the cokernel of $r_A$ since $yu^m$ is unit in $\cald(K_{\phi \circ \nu})_{t}[u^{\pm 1}]$.
Since $D_{\nu,\phi}( [yu^m]) = 0$ follows from Lemma~\ref{lem:image_of_an_element_x_under_D_nu,phi}, 
it suffices to treat the special case, 
 where $A$ is a diagonal matrix over $\cald(K_{\phi \circ \nu})_{t}[u^{\pm 1}]$
with non-zero entries $d_1,\dots,d_n$ on the diagonal.

Let $x$ be the product of the diagonal entries $d_1,\dots,d_n$ of the diagonal matrix $A$. 
We get in $\cald(G)^{\times}_{\abel}$
\[
{\det}_{\cald(G)^{\times}}([A]) = [x].
\]
Next recall that  for $\cald(G)$-maps 
$f_1 \colon M_0 \to M_1 $ and $f_2 \colon M_1 \to M_2$ we have 
 the obvious exact sequence
\[
0 \to \ker(f_1) \to \ker(f_2 \circ f_1) \to \ker(f_2) \to \coker(f_1) \to \coker(f_2 \circ f_1) \to \coker(f_2) \to 0.
\]
We iteratively apply this to $f_1=\operatorname{diag}(d_1,\dots,d_i,1)$ and $f_2=\operatorname{diag}(1,\dots,d_{i+1})$ to conclude
\[
\dim_{\cald(K_{\phi \circ \nu})}(\coker(r_A)) = \dim_{\cald(K_{\phi \circ \nu})}\bigl(\coker(r_x \colon \cald(G) \to \cald(G))\bigr).
\]
We conclude from Lemma~\ref{lem:degree_and_rank} and Lemma~\ref{lem:image_of_an_element_x_under_D_nu,phi}. 
\[
\tmfrac{1}{2}\dim_{\cald(K_{\phi \circ \nu})}(\coker(r_x)) \,\,=\,\,\tmfrac{1}{2} \deg(x)\,\, =\,\, D_{\nu,\phi}([x]) \,\,=\,\, D_{\nu,\phi} \circ {\det}_{\cald(G)}([A]).
\]
This finishes the proof of Lemma~\ref{matrices_overcald(K_phi_circ_nu)_t(u,u(-1))}.
\end{proof}

\begin{lemma}\label{continuity_of(mu,phi)-L2-Euler_in_phi}
 Let $M$ be an admissible $3$-manifold. Let $G$ be a torsion-free group
 which satisfies the Atiyah Conjecture. Consider any
 factorization $\pi \xrightarrow{\mu} G \xrightarrow{\nu} H_1(M)_f$ of the canonical
 projection $\pi \to H_1(M)_f$. Assume that $b_1^{(2)}(\overline{M};\caln(G)) = 0$ holds for the $G$-covering 
 $\overline{M} \to M$ associated to $\mu$.

 Then there exist two seminorms $s_1$ and $s_2$ on $H^1(M;\IR)$
 such that we get for every $\phi \in H^1(M;\IZ) = \Hom_{\IZ}(H_1(M)_f;\IZ)$ 
 \[
 \chi^{(2)}(M;\mu,\phi \circ \nu) \,\, =\,\, s_1(\phi) - s_2(\phi).
 \]
\end{lemma}
\begin{proof}
 We treat only the case, where $\partial M$ is non-empty, the case of empty $\partial M$ is completely analogous.
 Let $x_1, x_2, \ldots , x_a$ be the element in $G$ and $A$ be the $(a-1,a-1)$-matrix over $\IZ G$ appearing
 Theorem~\ref{the:calculation_of_(mu,phi)_L2-Euler_characteristic_from_a_presentation}~%
\eqref{the:calculation_of_(mu,phi)_L2-Euler_characteristic_from_a_presentation:non-empty}. 
 (Notice that they are independent of $\phi$.) We conclude from
 Theorem~\ref{the:calculation_of_(mu,phi)_L2-Euler_characteristic_from_a_presentation}~%
\eqref{the:calculation_of_(mu,phi)_L2-Euler_characteristic_from_a_presentation:non-empty}
that for any surjective group homomorphism $\phi \colon G \to \IZ$ we have
 \begin{multline*}
 \chi^{(2)}(M;\mu,\phi\circ \nu) = |\phi \circ \nu \circ \mu(x_i)|
 \\
 - \dim_{\cald(K_{\phi \circ \nu})}\bigl(\coker\bigl(r_A \colon \cald(K_{\phi \circ \nu})_{t}[u^{\pm 1}]^n 
 \to \cald(K_{\phi \circ \nu})_{t}[u^{\pm 1}]^n\bigr)\bigr).
\end{multline*}
 Choose two seminorms $s_1$ and $s_2$ such that the image of the class $2\cdot [A]$ in $K_1(\caln(G))$  under the composite
 \[
 K_1(\cald(G)) \xrightarrow{\det_{\cald{G}}} \cald(G)^{\times}_{\abel} 
 \xrightarrow{\IP'_{\nu}} \calp_{\IZ}(H_1(M)_f) \xrightarrow{\sn} \calsn(H_1(M)_f)
 \]
 is $s_1 - s_2$. We get from the definitions that
 for any surjective group homomorphism $\phi \colon G \to \IZ$ 
  the image of $2\cdot [A]$ under the composite
 $K_1(\cald(G)) \xrightarrow{\det_{\cald{G}}} \cald(G)^{\times}_{\abel} \xrightarrow{D_{\nu,\phi}} \IR$ equals $s_1(\phi) - s_2(\phi)$.
 We conclude from Lemma~\ref{matrices_overcald(K_phi_circ_nu)_t(u,u(-1))}
 \[
 \chi^{(2)}(M;\mu,\phi\circ \nu) = s_1(\phi) + |\phi \circ \nu \circ \mu(s_i)| - s_2(\phi),
 \]
 provided that $\phi$ is surjective. 
 We conclude from Lemma~\ref{lem:reduction_to_surjective_mu} that the last equation holds for
 every group homomorphism $\phi \colon G \to \IZ$.
\end{proof}


\subsection{The quasi-fibered case}
\label{subsec:The_quasi-fibered_case}

\begin{definition}[Fibered and quasi-fibered]
\label{def:fibered_and_quasi_fibered}
 Let $M$ be a $3$-manifold and consider an element $\phi\in H^1(M;\IQ)$. We say that
 $\phi$ is \emph{fibered} if there exists a locally trivial fiber bundle $p\colon M\to S^1$ 
 and $k \in \IQ$, $k > 0$ such that the induced map
 $p_*\colon \pi_1(M)\to \pi_1(S^1)=\IZ $ coincides with $k \cdot \phi$. We call an
 element $\phi \in H^1(M;\IR)$ \emph{quasi-fibered}, if there exists a sequence of
 fibered elements $\phi_n \in H^1(M;\IQ)$ converging to $\phi$ in $H^1(M;\IR)$.
\end{definition}

\begin{theorem}[Equality of $(\mu,\phi)$-$L^2$-Euler characteristic and the
 Thurston norm in the quasi-fibered case]
\label{the:Equality_of_(mu,phi)-L2-Euler_characteristic_and_the_Thurston_norm_in_the_quasi-fibered_case}
 Let $M$ be an admissible $3$-manifold, which is not homeomorphic to 
 $S^1 \times D^2$. Let $G$ be a torsion-free group which satisfies the
 Atiyah Conjecture. Consider any factorization 
 $\pr_M \colon \pi_1(M) \xrightarrow{\mu} G \xrightarrow{\nu} H_1(M)_f$ of the canonical projection 
 $\pr_M$. Let $\phi \colon H_1(M)_f \to \IZ$ be a quasi-fibered homomorphism. 

 Then $(\mu,\phi \circ \nu)$ is an $L^2$-acyclic Atiyah pair and we get
 \[
 -\chi^{(2)}(M;\mu,\phi \circ \nu)\,\, =\,\, x_M(\phi).
 \]
\end{theorem}
\begin{proof} 
Choose a sequence of fibered elements $\phi_n \in H^1(M;\IQ)$ converging to $\phi$ in $H^1(M;\IR)$.
For each $n$ choose a locally trivial fiber bundle 
$F_n \to M \xrightarrow{p_n} S^1$ and an $k_n \in \IQ$, $k_n > 0$ such that the induced map
$(p_n)_*\colon \pi_1(M)\to \pi_1(S^1)=\IZ $ coincides with $k_n \cdot \phi_n$.
It follows from elementary 3-manifold topology that the only orientable, irreducible fibered 3-manifold with fiber of positive Euler characteristic is $S^1\times D^2$. Since $M$ is admissible and since exclude $S^1\times D^2$ we deduce that $\chi(F_n) \le 0$.
We conclude from~\eqref{fiber_bundles_Thurston_norm}, Example~\ref{exa_mapping_torus},
Theorem~\ref{the:Status_of_the_Atiyah_Conjecture}~\eqref{the:Status_of_the_Atiyah_Conjecture:3-manifold_not_graph},
and~\cite[Theorem~2.1]{Lueck(1994b)} that 
$(\mu,k_n \cdot \phi_n)$ is an $L^2$-acyclic Atiyah-pair for $M$ and
\begin{equation}
-\chi^{(2)}(M;\mu,(k_n \cdot \phi_n) \circ \nu) = -\chi(F_n) = x_{M}(k_n \cdot \phi_n).
\label{Coman}
\end{equation}
Let $s_1$ and $s_2$ be the two seminorms appearing in Lemma~\ref{continuity_of(mu,phi)-L2-Euler_in_phi}.
Recall that we have for every $\psi \in H^1(M;\IZ)$
\begin{equation}
\chi^{(2)}(M;\mu,\psi \circ \nu) \,\, =\,\, s_1(\psi) - s_2(\psi).
\label{Costa}
\end{equation}
Since any seminorm on $H^1(M;\IR)$ is continuous, we get
\begin{multline*}
x_M(\phi) 
= 
\lim_{n \to \infty} x_M(\phi_n)
= 
\lim_{n \to \infty} \smfrac{x_M(k_n \cdot \phi_n)}{k_n}
\\
\stackrel{\eqref{Coman}}{=}
\lim_{n \to \infty} \smfrac{-\chi^{(2)}(M;\mu,(k_n \cdot \phi_n) \circ \nu)}{k_n}
\stackrel{\eqref{Costa}}{=}
\lim_{n \to \infty} \smfrac{-s_1(k_n \cdot \phi_n) +s_2(k_n \cdot \phi_n)}{k_n}
\\
= 
\lim_{n \to \infty} -s_1(\phi_n) +s_2(\phi_n)
 = 
-s_1(\phi) +s_2(\phi)
\stackrel{\eqref{Costa}}{=} 
\chi^{(2)}(M;\mu,\phi\circ \nu).\qedhere
\end{multline*}
\end{proof}


\subsection{Proof of Theorem~\ref{the:Equality_of_(mu,phi)-L2-Euler_characteristic_and_the_Thurston_norm}}
\label{subsec:Proof_of_Theorem_ref(the:Equality_of_(mu,phi)-L2-Euler_characteristic_and_the_Thurston_norm)}

\begin{proof}[Proof of Theorem~\ref{the:Equality_of_(mu,phi)-L2-Euler_characteristic_and_the_Thurston_norm}]
This is a variation of the proof of~\cite[Theorem~5.1]{Friedl-Lueck(2015l2+Thurston)}. For the reader's convenience
we give some details here as well.

As explained in~\cite[Section~10]{Dubois-Friedl-Lueck(2014Alexander)}, we conclude from
 combining~\cite{Agol(2008),Agol(2013),Liu(2013),Przytycki-Wise(2012),Przytycki-Wise(2014),Wise(2012raggs),Wise(2012hierachy)}
 that there exists a finite regular
 covering $p \colon N \to M$ such that for any $\phi \in H^1(M;\IR)$ its pullback
 $p^*\phi \in H^1(N;\IR)$ is quasi-fibered. (Note that for this step we need that $M$ is not a closed graph manifold.)
 Let $k$ be the number of sheets of $p$.
 Let $\pr_N \colon \pi_1(N) \to H_1(N)_f$ be
 the canonical projection. Its kernel is a characteristic subgroup of $\pi_1(N)$. The
 regular finite covering $p$ induces an injection $\pi_1(p) \colon \pi_1(N) \to \pi_1(M)$
 such that the image of $\pi_1(p)$ is a normal subgroup of $\pi_1(M)$ of finite index.
 Let $\Gamma$ be the quotient of $\pi_1(M)$ by the normal subgroup
 $\pi_1(p)(\ker(\pr_N))$. Let $\alpha \colon \pi_1(M) \to \Gamma$ be the
 projection. Since $\pi_1(p)(\ker(\pr_N))$ is contained in the kernel of the canonical
 projection $\pr_M \colon \pi_1(M)\to H_1(M)_f$ because of 
 $H_1(p;\IZ)_f \circ \pr_N = \pr_M \circ \, \pi_1(p)$, there exists precisely one epimorphism 
 $\beta \colon \Gamma \to H_1(M)_f$ satisfying $\pr_M = \beta \circ \alpha$.
 Moreover, $\alpha \circ \pi_1(p)$ factorizes over $\pr_N$ into an injective homomorphism
 $j \colon H_1(N)_f \to \Gamma$ with finite cokernel. 
 Hence $\Gamma$ is virtually finitely generated free abelian.

 Consider a factorization of $\alpha \colon \pi_1(M) \to \Gamma$
into group homomorphisms $\pi_1(M) \xrightarrow{\mu} G \xrightarrow{\nu} \Gamma$ 
for a group $G$ which satisfies the Atiyah Conjecture. Let $G' $ be the quotient of
$\pi_1(N)$ by the normal subgroup $\pi_1(p)^{-1}(\ker(\mu))$ and
$\mu' \colon \pi_1(N) \to G'$ be the projection. Since the kernels of
$\mu'$ and of $\mu \circ \pi_1(p)$ agree, there is precisely
one injective group homomorphism $i \colon G' \to G$ satisfying 
$\mu \circ \pi_1(p) = i \circ \mu'$. The kernel of $\mu'$ is contained in the kernel of 
$\pr_N \colon \pi_1(N) \to H_1(N)_f$ since $j$ is injective and we have
$j \circ \pr_N = \nu \circ i \circ \mu'$. Hence there is precisely one group homomorphism
$\nu' \colon G' \to H_1(N)_f$ satisfying $\nu' \circ \mu' = \pr_N$. One easily checks that 
the following diagram commutes, and all vertical arrows are injective, the indices 
$[\pi_1(M): \im(\pi_1(p))]$ and $[\Gamma: \im(j)]$ are finite,
and $\mu'$, $\nu'$ and $\beta$ are surjective:
\[
\xymatrix{\pi_1(N) \ar[r]^-{\mu'} \ar[d]^{\pi_1(p)} \ar@/^{5mm}/[rr]^{\pr_N}
& 
G' \ar[r]^-{\nu'} \ar[d]^i
& 
H_1(N)_f \ar[rd]^{H_1(p)_f} \ar[d]^j
\\
\pi_1(M) \ar[r]^-{\mu} \ar@/_{5mm}/[rr]_{\alpha} \ar@/_{10mm}/[rrr]_{\pr_M}
& 
G \ar[r]^{\nu}
&
\Gamma \ar[r]^-{\beta}
& 
H_1(M)_f 
}
\]

Since $G$ satisfies the Atiyah Conjecture, the group $G'$ 
satisfies the Atiyah Conjecture by 
Theorem~\ref{the:Status_of_the_Atiyah_Conjecture}~\eqref{the:Status_of_the_Atiyah_Conjecture:subgroups}.

Since $\ker(\mu) \subseteq \ker(\alpha) \subseteq \im(\pi_1(p))$ holds, we get $[G:G'] = k$ and we conclude 
from~\eqref{finite_coverings_Thurston_norm} and from
Lemma~\ref{lem:reduction_to_surjective_mu} (4)  that
\begin{eqnarray*}
 -\chi^{(2)}(M;\mu,\phi \circ \beta \circ \nu) 
& = & 
 -\lmfrac{\chi^{(2)}(N;\mu',p^*\phi \circ \nu')}{k};
\\
x_M(\phi) 
& = & 
\lmfrac{x_N(p^*\phi)}{k}.
\end{eqnarray*}
We get from
Theorem~\ref{the:Equality_of_(mu,phi)-L2-Euler_characteristic_and_the_Thurston_norm_in_the_quasi-fibered_case}
applied to $N$, $\mu'$, $\nu'$ and $p^*\phi$
\[
 -\chi^{(2)}(N;\mu',p^*\phi \circ \nu')\,\, =\,\, x_N(p^*\phi).
\]
Hence we get 
\[
 -\chi^{(2)}(M;\mu,\phi \circ \beta \circ \nu)\,\,=\,\, x_M(\phi).
\]
This finishes the proof of Theorem~\ref{the:Equality_of_(mu,phi)-L2-Euler_characteristic_and_the_Thurston_norm}.
\end{proof}


 \typeout{---- Section 7: Epimorphism of fundamental groups and the Thurston norm}

\section{Epimorphism of fundamental groups and the Thurston norm}
\label{sec:Epimorphism_of_fundamental_groups_and_the_Thurston_norm}

\begin{theorem}\label{the:localization_down_to_Z_versus_up_to_U(G)_new} Let $G$ be a group
 which is residually a locally indicable elementary amenable group. Let $f_* \colon C_* \to
 D_*$ be a $\IZ G$-chain map of finitely generated free $\IZ G$-chain complexes.
 Suppose that $\id_{\IQ} \otimes_{\IZ G} f_* \colon \IQ\otimes_{\IZ G} C_* \to
 \IQ\otimes_{\IZ G} D_*$ induces an isomorphism on homology. Then we get for $n \ge 0$
\[
b_n^{(2)}(\caln(G) \otimes_{\IZ G} C_*)\,\, =\,\, b_n^{(2)}(\caln(G) \otimes_{\IZ G} D_*).
\]
\end{theorem}

The reader might wonder why it is so important that $G$ is locally indicable. 
This condition comes up in the following theorem 
due to Gersten and independently Howie-Schneebeli.
The theorem  is a key ingredient in the proof of Theorem~\ref{the:localization_down_to_Z_versus_up_to_U(G)_new}.

\begin{theorem}\label{thm:howie-schneebeli-gersten}
Let $G$ be a group and let $f\colon P\to Q$ be a homomorphism
of finitely generated free $\IZ G$-left modules.
If  $\id_{\IQ} \otimes_{\IZ Q} f\colon \IQ\otimes_{\IZ G}P\to \IQ\otimes_{\IZ G} Q$ is a monomorphism and if  $G$ is locally indicable, then $
\id_{\IQ Q} \otimes_{\IQ G} f\colon 
\IQ G\otimes_{\IZ G}P\to  \IQ G\otimes_{\IZ G} Q$ is also a monomorphism.
\end{theorem}

\begin{proof}
We consider the following two commutative diagrams
\[ \xymatrix@C1.2cm@R0.65cm{  \IZ\otimes_{\IZ G}P\ar[d]\ar[r]^{\id_{\IZ} \otimes f}& \IZ\otimes_{\IZ G} Q
\ar[d]\\
\IQ\otimes_{\IZ G}P\ar[r]^{\id_{\IQ} \otimes f}& \IQ\otimes_{\IZ G} Q}
\quad \mbox{ and } \quad \xymatrix@C1.2cm@R0.65cm{
\IZ G\otimes_{\IZ G} P\ar[d]\ar[r]^{\id_{\IZ G} \otimes f}& \IZ G\otimes_{\IZ G}  Q
\ar[d]\\
\IQ G\otimes_{\IZ G}P\ar[r]^{\id_{\IQ G} \otimes f}& \IQ G\otimes_{\IZ G} Q.}
\]
First we consider the diagram on the left. By our hypothesis the bottom map  is a monomorphism.
Since $P$ and $Q$ are free $\IZ[G]$-modules we see that the top map in the same diagram is  a monomorphism as well. By 
~\cite[Theorem~1]{Howie-Schneebeli(1983)} 
or~\cite[Theorem~4.1]{Gersten(1983)}
this implies that the map on the top of the diagram on the right hand side is also a monomorphism. But going from top to bottom on the right hand side is just given by tensoring a map of free $\IZ$-modules with $\IQ$. Thus we see that the map on the bottom right is also a monomorphism.
\end{proof}

\begin{proof}[Proof of Theorem~\ref{the:localization_down_to_Z_versus_up_to_U(G)_new}]
 By considering the mapping cone, it suffices to show for a finitely generated free
 $\IZ G$-chain complex $C_*$ that $b_n^{(2)}(\caln(G) \otimes_{\IZ G} C_*)$ vanishes 
for all $n \ge 0$ if $H_n(\IQ \otimes_{\IZ G} C_*)$ vanishes for all $n \ge 0$. 
 
First note that a locally indicable group is torsion-free.
Since $G$ is residually a locally indicable elementary amenable group,
there exists a sequence of epimorphisms $G\to G_i$, $i\in \IN$ onto 
 locally indicable elementary amenable groups such that the intersections of the kernels is trivial. We
conclude from~\cite{Clair(1999)},~\cite[Theorem~1.14]{Schick(2001b)} 
or~\cite[Theorem~13.3 on page~454]{Lueck(2002)}
\begin{eqnarray*}
b_n^{(2)}(\caln(G) \otimes_{\IZ G} C_*) & = & 
\lim_{i \in \IN} \;b_n^{(2)}(\caln(G_i) \otimes_{\IZ G_i} \IZ G_i \otimes_{\IZ G} C_*);
\\
\IQ \otimes_{\IZ G} C_* & \cong & \IQ \otimes_{\IZ G_i} \IZ G_i \otimes_{\IZ G} C.
\end{eqnarray*}
It follows that without loss
of generality we can assume that $G$ itself is locally indicable elementary amenable.

There is an involution of rings $\IZ G \to \IZ G$ sending $\sum_{g \in G} \lambda_g \cdot g$ 
to $\sum_{g \in G} \lambda_g \cdot g^{-1}$. In the sequel we equip each
$C_n$ with a $\IZ G$-basis. With respect to this involution and $\IZ G$-basis one can define the
combinatorial Laplace operator $\Delta_n \colon C_n \to C_n$ which is the $\IZ G$-linear
map given by $c_{n+1} \circ c_n^* + c_{n-1}^* \circ c_{n-1}$. Since the augmentation
homomorphism $\IZ G \to \IZ$ sending $\sum_{g \in G} \lambda_g \cdot g $ to 
$\sum_{g \in G} \lambda_g$ is compatible with the involution, 
$\id_{\IZ} \otimes_{\IZ G} \Delta_n \colon \id_{\IZ} \otimes_{\IZ G} C_n \to \id_{\IZ} \otimes_{\IZ G} C_n$ 
is the combinatorial Laplace operator for $\IZ \otimes_{\IZ G} C_*$. We conclude
from~\cite[Lemma~1.18 on page~24 and Theorem~6.25 on page~249]{Lueck(2002)}
\begin{eqnarray*}
b_n^{(2)}(\caln(G) \otimes_{\IZ G} C_*) & = & 
\dim_{\caln(G)}(\ker(\id_{\caln(G)} \otimes_{\IZ G} \Delta_n));
\\
\dim_{\IQ}(H_n(\IQ \otimes_{\IZ G} C_*)) & = & \dim_{\IQ}(\ker(\id_{\IQ} \otimes_{\IZ G} \Delta_n)).
\end{eqnarray*}
Since  $G$ is amenable, we conclude  as
 in~\cite[(6.74) on page~275]{Lueck(2002)}
\begin{eqnarray*}
\dim_{\caln(G)}(\ker(\id_{\caln(G)} \otimes_{\IZ G} \Delta_n))
& = & 
\dim_{\caln(G)}(\caln(G) \otimes_{\IQ G} \ker(\id_{\IQ G}\otimes_{\IZ G} \Delta_n)).
\end{eqnarray*}
Hence $b_n^{(2)}(\caln(G) \otimes_{\IZ G} C_*) $ vanishes if $\id_{\IQ G}\otimes_{\IZ G}\Delta_n$ is injective. The injectivity of $\id_{\IQ G}\otimes_{\IZ G}\Delta_n$ follows from the
injectivity of $\id_{\IQ} \otimes_{\IZ G} \Delta_n$,
the hypothesis that $G$ is locally indicable and
Theorem~\ref{thm:howie-schneebeli-gersten}.
\end{proof}

\begin{theorem}\label{the:Inequality} Let $f \colon M \to N$ be a map of admissible
 $3$-manifolds. Suppose that $\pi_1(f)$ is surjective and $f$ induces an isomorphism 
 $H_n(f;\IQ) \colon H_n(M;\IQ) \to H_n(N;\IQ)$ for $n \ge 0$. Suppose that $G$ is residually
 locally indicable elementary amenable. Let $\mu \colon \pi_1(N) \to G$, 
 $\nu \colon G \to H_1(N)_f$ and $\phi \colon H_1(N)_f \to \IZ$ be group 
 homomorphisms. Let $\overline{N} \to N$ be the $G$-covering associated to $\mu$ and
 $\overline{M} \to M$ be the $G$-covering associated to $\mu \circ \pi_1(f)$. Suppose that
 $b_n^{(2)}(\overline{N};\caln(G))$ vanishes for $n \ge 0$. 

 Then $b_n^{(2)}(\overline{M};\caln(G))$ vanishes 
 for $n \ge 0$, $M$ is $(\mu \circ \pi_1(f),\phi \circ \nu)$-$L^2$-finite, 
 $N$ is $(\mu,\phi \circ \nu)$-$L^2$-finite and we get
 \[
 - \chi^{(2)}(M;\mu \circ \pi_1(f);\phi \circ \nu) \,\,\ge\,\, - \chi^{(2)}(N;\mu,\phi \circ \nu).
 \]
\end{theorem}

\begin{proof}
Since a locally indicable group is torsion-free,
$G$ is a residually torsion-free elementary amenable group and hence satisfies the Atiyah Conjecture
 by Theorem~\ref{the:Status_of_the_Atiyah_Conjecture}~\eqref{the:Status_of_the_Atiyah_Conjecture:approx}.
 Because of Lemma~\ref{lem:reduction_to_surjective_mu} and 
 Theorem~\ref{the:Status_of_the_Atiyah_Conjecture}~\eqref{the:Status_of_the_Atiyah_Conjecture:subgroups}
 we can assume without loss of generality
 that $\mu$ and $\phi \circ \nu$ are epimorphisms.
 Theorem~\ref{the:localization_down_to_Z_versus_up_to_U(G)_new} implies that
 $b_n^{(2)}(\overline{M};\caln(G))$ vanishes for $n \ge 0$. We conclude from 
 Theorem~\ref{the:Atiyah_and_(mu,phi)-L2-Euler_characteristic}
 that $M$ is $(\mu \circ \pi_1(f),\phi \circ \nu)$-$L^2$-finite and $N$ is $(\mu,\phi\circ \nu)$-$L^2$-finite. 
 
 Since $\pi_1(f)$ is surjective and hence the $G$-map $\overline{f} \colon \overline{M} \to
 \overline{N}$ induced by $f$ is $1$-connected by the exact sequence on homotopy groups associated to 
a covering,  we get
 $b_1^{(2)}(i^*\overline{M};\caln(K)) \ge b_1^{(2)}(i^* \overline{N};\caln(K))$ for the inclusion 
 $i \colon K = \ker(\phi \circ \nu) \to G$. If $\phi \circ \nu \circ \mu = 0$, we conclude
 $\chi^{(2)}(M;\mu \circ \pi_1(f);\phi \circ \nu) = \chi^{(2)}(N;\mu,\phi \circ \nu) = 0$ from 
 Lemma~\ref{lem:reduction_to_surjective_mu}~%
\eqref{lem:reduction_to_surjective_mu:mu_circ_phi_is_trivial} and the claim follows.
 Hence we can assume without loss of generality that $\phi \circ \nu \circ \mu$ is not trivial.
 
 We begin with the case $\im(\mu) \cap \ker(\phi \circ \nu) \not= \{1\}$. Then also 
 $\im(\mu \circ \pi_1(f)) \cap \ker(\phi \circ \nu) \not= \{1\}$. We conclude from 
 Theorem~\ref{the:The_(mu,phi)-L2-Euler_characteristic_in_terms_of_the_first_homology} 
 \begin{eqnarray*}
 -\chi^{(2)}(N;\mu,\phi \circ \nu) & = & b_1^{(2)}(i^*\overline{N};\caln(K));
 \\
 -\chi^{(2)}(M;\mu \circ \pi_1(f);\phi \circ \nu) & = & b_1^{(2)}(i^*\overline{M};\caln(K)).
 \end{eqnarray*}
 and Theorem~\ref{the:Inequality} follows. 

 It remains to treat the case, where $\im(\mu) \cap \ker(\phi \circ \nu) = \{1\}$.
 Then $\phi \circ \nu \colon G \to \IZ$
 is an injection and hence a bijection, $K = \{0\}$. Since $\mu$ is assumed to be an epimorphism  we get from Lemma~\ref{lem:infinite_cyclic_covering}
 \begin{eqnarray*}
 - \chi^{(2)}(N;\mu,\phi \circ \nu) \,\,=\,\, 
 \begin{cases} 
 \dim_{\IZ}\bigl(H_1(\overline{M};\IZ)\bigr) - 1 & \text{if}\; \partial M \not= \emptyset;\\
 \dim_{\IZ}\bigl(H_1(\overline{M};\IZ)\bigr) - 2 & \text{if}\; \partial M = \emptyset;
 \end{cases}
 \\
 - \chi^{(2)}(N;\mu \circ \pi_1(f),\phi \circ \nu)\,\, =\,\, 
 \begin{cases} 
 \dim_{\IZ}\bigl(H_1(\overline{N};\IZ)\bigr) - 1 & \text{if}\; \partial N \not= \emptyset;\\
 \dim_{\IZ}\bigl(H_1(\overline{N};\IZ)\bigr) - 2 & \text{if}\; \partial N = \emptyset;
 \end{cases}
\end{eqnarray*}
 We already have shown $b_1^{(2)}(i^*\overline{N};\caln(K)) \ge b_1^{(2)}(i^* \overline{M};\caln(K))$ which boils down
 in this special case to $\dim_{\IZ}(H_1(\overline{M};\IZ))\ge \dim_{\IZ}(H_1(\overline{N};\IZ))$. We conclude
 from $H_3(M;\IZ) \cong H_3(N;\IZ)$ that $\partial M$ is empty if and only if $\partial N$ is empty. This finishes the proof
 of Theorem~\ref{the:Inequality}.
 \end{proof}

\begin{theorem}[Inequality of the Thurston norm]\label{the:Inequality_of_the_Thurston_norm}
Let $f \colon M \to N$ be a map of admissible $3$-manifolds which is surjective 
on $\pi_1(N)$ and induces an isomorphism 
$f_* \colon H_n(M;\IQ) \to H_n(N;\IQ)$ for $n \ge 0$. Suppose that $\pi_1(N)$ is residually 
locally indicable elementary amenable. 
 Then we get for any $\phi\in H^1(N;\IR)$ that
 \[
 x_M(f^*\phi ) \,\,\ge\,\, x_N(\phi).
 \]
\end{theorem}

\begin{proof}
Since seminorms are continuous and homogeneous it suffices to prove the 
statement for all primitive classes $\phi\in H^1(N;\IZ)=\Hom(\pi_1(N),\IZ)$. 
The case $N = S^1 \times D^2$ is trivial. Hence we can assume that $N\ne S^1\times D^2$. 
 We conclude from Theorem~\ref{the:Inequality} applied in the case $G = \pi_1(N)$ and $\mu= \id_{\pi_1(N)}$
 \[
 -\chi^{(2)}(M; \pi_1(f),\phi) \,\,\ge\,\, -\chi^{(2)}(\widetilde{N};\phi)
 \]
 Theorem~\ref{the:The_Thurston_norm_ge_the_(mu,phi)-L2-Euler_characteristic} implies 
 \[
 x_M(\phi \circ \pi_1(f)) \,\,\ge\,\,-\chi^{(2)}(M; \pi_1(f),\phi)
 \]
and Theorem~\ref{the:Equality_of_(mu,phi)-L2-Euler_characteristic_and_the_Thurston_norm_universal_covering_for_universal_coverings} implies that 
  \[
  -\chi^{(2)}(\widetilde{N};\phi) \,\,=\,\, x_N(\phi).
 \]
 Now Theorem~\ref{the:Inequality_of_the_Thurston_norm} follows.
\end{proof}

The following lemma shows that Theorem~\ref{the:Inequality_of_the_Thurston_norm} applies in particular 
if the manifold $N$ is fibered. Since it is well-known,  we only provide a sketch of the proof.

\begin{lemma}\label{lem:fibered-manifolds-acapl}
The fundamental group of any fibered 3-manifold is residually locally indicable elementary amenable.
\end{lemma}

\begin{proof}[Sketch of proof]
 Let $N$ be a fibered 3-manifold. Then $\pi_1(N)\cong \IZ\ltimes_\varphi \Gamma$ where
 $\Gamma$ is a free group or a surface group and where $\varphi\colon \Gamma\to \Gamma$
 is an automorphism. The derived series of $\Gamma$ is defined by $\Gamma^{(0)}=\Gamma$
 and inductively by $\Gamma^{(n+1)}=[\Gamma^{(n)},\Gamma^{(n)}]$. Since $\Gamma$ is a
 free group or a surface group, each quotient $\Gamma^{(n)}/\Gamma^{(n+1)}$ is free
 abelian and $\bigcap_{n \ge 1} \Gamma^{(n)}=\{1\}$.

 The subgroups $\Gamma^{(n)}$ are characteristic subgroups of $\Gamma$, in particular
 they are preserved by $\varphi$. Thus $\varphi$ descends to an automorphism on
 $\Gamma/\Gamma^{(n)}$. It is now straightforward to see that the epimorphisms
 $\pi_1(N)=\IZ\ltimes \Gamma\to \IZ\ltimes \Gamma/\Gamma^{(n)}$, $n\in \IN$ form a
 cofinal nested sequence of epimorphisms onto locally indicable elementary amenable
 groups.
\end{proof}


 \typeout{---- Section 8: The $\phi$-$L^2$-Euler characteristic and the degree of higher order Alexander polynomials}

\section{The $(\mu,\phi)$-$L^2$-Euler characteristic and the degree of non-commutative Alexander polynomials}
\label{sec:The_(mu,phi)-L2-Euler_characteristic_and_the_degree_of_higher_order_Alexander_polynomials}


Let $M$ be an admissible $3$-manifold. Regard group homomorphisms $\mu \colon \pi_1(M)\to G$, 
$\nu \colon G \to H_1(M)_f$ and $\phi \colon H_1(M)_f \to \IZ$ such that $\nu
\circ \mu$ is the projection $\pi_1(M) \to H_1(M)_f$ and $G$ is torsion-free elementary amenable.
For simplicity we discuss only the
case, where $\phi$ is surjective. Let $\overline{M} \to M$ be the $G$-covering associated 
to $\mu \colon \pi_1(M) \to G$. Recall the following definition from
Harvey~\cite{Harvey(2005)} which extends ideas of Cochran~\cite{Cochran(2004)}.
(Actually they consider only certain solvable quotients $G$ of $\pi_1(M)$ coming from the
rational derived series, but their constructions apply directly to torsion-free elementary
amenable groups.) Let $T$ be the set of non-trivial elements in $\IZ G$. As recorded
already in Lemma~\ref{lem:elementary_amenable_and_ore}, the Ore localization $T^{-1}\IZ G$
is defined and is a skewfield. Define a natural number
\begin{eqnarray}
r_n(M;\mu) 
:= 
\dim_{T^{-1}\IZ G}\bigl(H_n\bigl(T^{-1}\IZ G\otimes_{\IZ G} C_*(\overline{M})\bigr)\bigr).
\label{r_n(M,mu)}
 \end{eqnarray}
Let $i \colon K \to G$ be the inclusion of the kernel of the composite $\phi \circ \nu \colon G \to \IZ$. 
If $T$ is the set of non-zero elements in $\IZ K$, we can
again consider its Ore localization $T^{-1}\IZ K$ which is a skew field. We obtain an isomorphism
for an appropriate automorphism $t$ of $T^{-1}\IZ K$, which comes from the conjugation automorphism
of $K$ associated to a lift of a generator of $\IZ$ to $G$, an isomorphism of skew-fields
\begin{equation}
 (T^{-1}\IZ K)_{t}[u^{\pm 1}] \xrightarrow{\cong} T^{-1}\IZ G.
\label{Sanchez}
\end{equation}
If $r_n(M;\mu)$ vanishes for all $n \ge 0$, then we can define natural numbers
\begin{equation}
\delta_n(M;\mu, \nu, \phi) 
:= 
\dim_{T^{-1}\IZ K}\bigl(H_1(T^{-1}\IZ G \otimes_{\IZ G} C_*(\overline{M}))\bigr).
\label{delta_n(phi)_surjective_tors}
\end{equation}
This construction and the invariants above turn out to be special cases of the constructions defined in this paper.
Namely, $K$ and $G$ satisfy the Atiyah Conjecture by 
Theorem~\ref{the:Status_of_the_Atiyah_Conjecture}~\eqref{the:Status_of_the_Atiyah_Conjecture:Linnell},
and Lemma~\ref{lem:elementary_amenable_and_ore} shows that we get identifications
$T^{-1}\IZ K = \cald(K)$ and $T^{-1}\IZ G = \cald(G)$ under which the isomorphism~\eqref{Sanchez}
corresponds to the isomorphism appearing in 
Theorem~\ref{the:Main_properties_of_cald(G)}~\eqref{the:Main_properties_of_cald(G):cald(K)_and_(cald(G)}.
Moreover $r_n(M;\mu)$ agrees with $b_n^{(2)}(\overline{M};\caln(G))$
by Theorem~\ref{the:Main_properties_of_cald(G)}~\eqref{the:Main_properties_of_cald(G):dim}.
Hence Theorem~\ref{the:Main_properties_of_cald(G)}~\eqref{the:Main_properties_of_cald(G):chain} 
and Lemma~\ref{lem:b_1(2)_is_zero_implies_L2-acyclic}~\eqref{lem:b_1(2)_is_zero_implies_L2-acyclic:b_1}
imply

\begin{theorem}[The $(\mu,\phi)$-$L^2$-Euler characteristic and $\delta_n(M,\mu,\nu,\phi)$]
\label{the:The_mu,phi_L2-Euler_characteristic_and_delta_n(phi)}
 Let $M$ be an admissible $3$-manifold. Consider group homomorphisms
 $\mu \colon \pi_1(M) \to G$, $\nu \colon G \to H_1(M)_f$ and $\phi \colon H_1(M)_f \to \IZ$
such that $\nu \circ \mu$ is the projection $\pi_1(M) \to H_1(M)_f$ and $\phi$ is surjective.

 Then $r_1(M;\mu)$ vanishes if and only if $(\mu,\phi \circ \nu)$ is an $L^2$-Atiyah pair, and in this case we get
 \[ \chi^{(2)}(M;\mu,\phi \circ \nu)~=~-\delta_1(M;\mu,\nu,\phi).\]
\end{theorem}

\begin{remark}\label{rem:extending_Harveys_invariant} Another way of interpreting
 Theorem~\ref{the:The_mu,phi_L2-Euler_characteristic_and_delta_n(phi)} is to say that our
 $L^2$-Euler characteristic invariant extends the original invariant due to
 Cochran, Harvey and the first author~\cite{Cochran(2004),Harvey(2005),Friedl(2007)} to other
 coverings, in particular to the universal covering or to a $G$-covering for residually
 torsion-free elementary amenable group $G$ of an admissible $3$-manifold.
\end{remark}

The following lemma might also be of independent interest.

\begin{lemma}\label{lem:G_can_be_arranged_to_be_torsion-free_elementary_amenable} Let
 $\alpha \colon \pi_1(M) \to \Gamma$ be an epimorphism onto a group that is virtually torsion-free abelian. Then
 there exists a factorization of $\alpha$ into group homomorphisms 
 $\pi \xrightarrow{\mu} G \xrightarrow{\nu} \Gamma$ such that $G$ is torsion-free elementary amenable.
\end{lemma}

\begin{proof}
Let $\alpha \colon \pi_1(M) \to \Gamma$ be an epimorphism onto a group $\Gamma$ that admits a finite index subgroup $F$ that is free abelian. After possibly going to the core of $F$ we can assume that $F$ is normal.

Since $\alpha^{-1}(F)$ is a finite-index subgroup of $\pi_1(M)$ we conclude from~\cite[Theorem~3.4]{Friedl-Schreve-Tillmann(2017)} that there is a
 torsion-free elementary amenable group $G'$ together with an epimorphism 
 $\mu' \colon \pi_1(M) \to G'$ such that $\ker(\mu')\subset \alpha^{-1}(F)$. Define the epimorphism $\mu \colon \pi_1(M) \to G$ to be the
 projection onto the quotient $G = \pi_1(M)/\bigl(\ker(\mu') \cap \ker(\alpha)\bigr)$. Obviously there is 
 an epimorphism $\nu \colon G \to \Gamma$ such that $\nu \circ \mu = \alpha$ since $\alpha$ is by 
 construction the projection from $\pi_1(M)$ to the quotient $\Gamma = \pi_1(M)/\ker(\alpha)$.
 It remains to show that $G$ is torsion-free elementary amenable. We have the obvious exact sequence
 \[
 1 \to \ker(\mu')/\bigl(\ker(\mu') \cap \ker(\alpha)\bigr) \to G \to G' \to 1.
 \]
 and obviously $\alpha$ defines an injection
 \[\ker(\mu')/\bigl(\ker(\mu') \cap \ker(\alpha)\bigr) 
\hookrightarrow F.
\]
Since $F$ and $G'$ are torsion-free elementary amenable, the same is true for $G$.
\end{proof}

Theorems~\ref{the:Equality_of_(mu,phi)-L2-Euler_characteristic_and_the_Thurston_norm} and~\ref{the:The_mu,phi_L2-Euler_characteristic_and_delta_n(phi)}
and Lemma~\ref{lem:G_can_be_arranged_to_be_torsion-free_elementary_amenable} imply
that the non-commutative Reidemeister torsions of~\cite{Friedl(2007)} detect the Thurston norm of most 3-manifolds.

\begin{corollary}\label{cor:existence_of_torsion-free_elementary_coverings_with_equality} 
Let $M$ be a $3$-manifold, which is admissible, see Definition~\ref{def:admissible_3-manifold},
is not a closed graph manifold and is not homeomorphic to $S^1 \times D^2$. 
Then there is a torsion-free elementary amenable group $G$
and a factorization $\pr_M \colon \pi_1(M) \xrightarrow{\mu} G \xrightarrow{\nu} H_1(M)_f$
of the canonical projection $\pr_M$ into epimorphisms
such that for any group homomorphism $\phi \colon H_1(M)_f \to \IZ$, 
the pair $(\mu, \phi \circ \nu)$ is an $L^2$-acyclic Atiyah-pair and we get
\[
 \delta_1(M;\mu,\nu,\phi)~=~-\chi^{(2)}(M;\mu,\phi \circ \nu)~=~x_M(\phi).
\]
\end{corollary}

\begin{remark}
 The invariant $\delta_1(M;\mu,\nu, \phi)$ of~\cite{Friedl(2007)} is
 essentially the same as the Cochran-Harvey invariant~\cite{Cochran(2004),Harvey(2005)},
 except that Cochran--Harvey only study solvable quotients of $\pi_1(M)$. But as pointed
 out in~\cite[Example~2.3]{Cochran(2004)}, in general invariants coming from solvable
 quotients do not suffice to detect the knot genus respectively the Thurston norm. It is
 necessary to work with the extra flexibility given by torsion-free elementary amenable
 groups.
\end{remark}


 \typeout{---- Section 9: The degree of the $L^2$-torsion function}

\section{The degree of the $L^2$-torsion function}
\label{sec:The_degree_of_the_L2-torsion_function}

Extending earlier work in~\cite{Dubois-Friedl-Lueck(2014Alexander)}, in~\cite{Lueck(2015twisting)} the $\phi$-twisted $L^2$-torsion function has been
introduced and analyzed for $G$-coverings of compact connected manifolds in all
dimensions. In the sequel we will consider only $G$-coverings $\overline{M} \to M$ of
admissible $3$-manifolds $M$ for countable residually finite $G$. Then all the necessary
conditions such as $\det$-$L^2$-acyclicity for $\overline{M}$ and the $K$-theoretic
Farrell-Jones Conjecture for $\pi_1(M)$ are automatically satisfied and do not have to be
discussed anymore. One can assign to the $L^2$-torsion function by considering its asymptotic behavior at
infinity a real number called its degree and denoted by $\deg\bigl(\rho^{(2)}(M;\mu,\phi)\bigr)$.
If $G = \pi_1(M)$ and $\mu \,\,=\,\, \id_{\pi_1(M)}$, i.e., for the universal covering
$\widetilde{M}$, the equality
\[
\deg\bigl(\rho^{(2)}(M;\mu,\phi)\bigr) = x_M(\phi \circ \mu)
\]
was proved by the authors in~\cite[Theorem~0.1]{Friedl-Lueck(2015l2+Thurston)} and
independently by Liu~\cite{Liu(2015)}. Actually many more instances of $G$-coverings are
considered in~\cite[Theorem~5.1]{Friedl-Lueck(2015l2+Thurston)}, where this equality
holds.

We just mention without proof the following theorem.

\begin{theorem}[The $(\mu,\phi)$-$L^2$-Euler characteristic is a lower bound for the
 degree of the $L^2$-torsion function]
\label{the:The_(mu,phi)-L2-Euler_characteristic_is_an_lower_bound_for_the_degree_of_the_L2-torsion_function}
 Let $M$ be an admissible $3$-manifold. Let $\mu \colon \pi \to G$ be a
 homomorphism to a torsion-free, elementary amenable, residually finite, countable group $G$ and
 $\phi \colon G \to \IZ$ be a group homomorphism. Let $\overline{M} \to M$ be the $G$-covering
 associated to $\mu$. Suppose $b_1^{(2)}(\overline{M};\caln(G)) = 0$.
 Then $(\mu,\phi)$ is an $L^2$-acyclic Atiyah pair and we get
 \[
- \chi^{(2)}(M;\mu,\phi) \,\le\, \deg\bigl({\rho}^{(2)}(M;\mu,\phi)\bigr).
 \]
 \end{theorem}

\typeout{-------------------------------------- References ---------------------r------------------}


\begin{thebibliography}{10}

\bibitem{Agol(2008)}
I.~Agol.
\newblock Criteria for virtual fibering.
\newblock {\em J. Topol.}, 1(2):269--284, 2008.

\bibitem{Agol(2013)}
I.~Agol.
\newblock The virtual {H}aken conjecture.
\newblock {\em Doc. Math.}, 18:1045--1087, 2013.
\newblock With an appendix by Agol, Groves, and Manning.

\bibitem{AFW(2015)}
M. Aschenbrenner, S. Friedl and H. Wilton.
\newblock 3-manifold groups. 
\newblock EMS Series of Lectures in Mathematics. European Mathematical Society (EMS), Z\"urich, 2015.

\bibitem{Burde-Zieschang(1985)}
G.~Burde and H.~Zieschang.
\newblock {\em Knots}.
\newblock Walter de Gruyter \& Co., Berlin, 1985.

\bibitem{Cheeger-Gromov(1986)}
J.~Cheeger and M.~Gromov.
\newblock ${L}\sb 2$-cohomology and group cohomology.
\newblock {\em Topology}, 25(2):189--215, 1986.

\bibitem{Clair(1999)}
B.~Clair.
\newblock Residual amenability and the approximation of ${L}\sp 2$-invariants.
\newblock {\em Michigan Math. J.}, 46(2):331--346, 1999.

\bibitem{Cochran(2004)}
T.~D. Cochran.
\newblock Noncommutative knot theory.
\newblock {\em Algebr. Geom. Topol.}, 4:347--398, 2004.

\bibitem{Cohn(1985)}
P.~M. Cohn.
\newblock {\em Free rings and their relations}.
\newblock Academic Press Inc. [Harcourt Brace Jovanovich Publishers], London,
  second edition, 1985.

\bibitem{Cohn(1995)}
P.~M. Cohn.
\newblock {\em Skew fields}, volume~57 of {\em Encyclopedia of Mathematics and
  its Applications}.
\newblock Cambridge University Press, Cambridge, 1995.
\newblock Theory of general division rings.

\bibitem{Dubois-Friedl-Lueck(2014Alexander)}
J.~Dubois, S.~Friedl, and W.~L\"uck.
\newblock The {$L^2$}--{A}lexander torsion of 3-manifolds.
\newblock Journal of Topology 9, No. 3 (2016), 889--926. 

\bibitem{Eisenbud-Neumann(1985)}
D.~Eisenbud and W.~Neumann.
\newblock {\em Three-dimensional link theory and invariants of plane curve
  singularities}, volume 110 of {\em Annals of Mathematics Studies}.
\newblock Princeton University Press, Princeton, NJ, 1985.

\bibitem{Fox(1953)}
R.~H. Fox.
\newblock Free differential calculus. {I}. {D}erivation in the free group ring.
\newblock {\em Ann.\ of Math. (2)}, 57:547--560, 1953.

\bibitem{Friedl(2007)}
S.~Friedl.
\newblock Reidemeister torsion, the {T}hurston norm and {H}arvey's invariants.
\newblock {\em Pacific J. Math.}, 230(2):271--296, 2007.

\bibitem{Friedl(2014twisted)}
S.~Friedl.
\newblock Twisted {R}eidemeister torsion, the {T}hurston norm and fibered
  manifolds.
\newblock {\em Geom. Dedicata}, 172:135--145, 2014.

\bibitem{Friedl-Lueck(2015l2+Thurston)}
S.~Friedl and W.~L\"uck.
\newblock The {$L^2$}-torsion function and the {T}hurston norm of
  $3$-manifolds.
\newblock Preprint, arXiv:1510.00264 [math.GT], 2015.
\newblock To be published by Commentarii Mathematici Helvetici.

\bibitem{Friedl-Lueck-Tillmann(2016)}
S.~Friedl, W.~L\"uck, and S.~Tillman.
\newblock Groups and polytopes.
\newblock Preprint, arXiv:1611.01857 [math.GT], 2016.

\bibitem{Friedl-Schreve-Tillmann(2017)}
S.~Friedl, K.~Schreve, and S.~Tillmann.
\newblock {T}hurston norm via {F}ox calculus.
\newblock Geometry and Topology 21 (2017), 3759-3784. 

\bibitem{Funke(2016)}
F.~Funke.
\newblock The integral polytope group.
\newblock Preprint, arXiv:1605.01217 [math.MG], 2016.

\bibitem{Funke-Kielak(2016)}
F.~Funke and D.~Kielak.
\newblock {A}lexander and {T}hurston norms, and the {B}ier-{N}eumann-{S}trebel
  invariants for free-by-cyclic groups.
\newblock Preprint, arXiv:1605.09069 [math.RA], 2016.
\newblock To be published by Geometry and Topology.

\bibitem{Gabai(1983)}
D.~Gabai.
\newblock Foliations and the topology of {$3$}-manifolds.
\newblock {\em J. Differential Geom.}, 18(3):445--503, 1983.

\bibitem{Gersten(1983)}
S.~M. Gersten.
\newblock Conservative groups, indicability, and a conjecture of {H}owie.
\newblock {\em J. Pure Appl. Algebra}, 29(1):59--74, 1983.

\bibitem{Harvey(2005)}
S.~L. Harvey.
\newblock Higher-order polynomial invariants of 3-manifolds giving lower bounds
  for the {T}hurston norm.
\newblock {\em Topology}, 44(5):895--945, 2005.

\bibitem{Harvey(2006)}
S.~L. Harvey.
\newblock  Monotonicity of degrees of generalized Alexander polynomials of groups and 3-manifolds. 
\newblock Math. Proc. Cambridge Philos. Soc. 140 no. 3, 431--450, 2006.

\bibitem{Herrmann(2016)}
G.~Herrmann.
\newblock The {$L^2$}-{A}lexander torsion of {S}eifert fibered {$3$}-manifolds.
\newblock  Arch. Math. (Basel) 109 (2017), no. 3, 273--283.

\bibitem{Howie(1982)}
J.~Howie.
\newblock On locally indicable groups.
\newblock {\em Math. Z.}, 180(4):445--461, 1982.

\bibitem{Howie-Schneebeli(1983)}
J.~Howie and H.~R. Schneebeli.
\newblock Homological and topological properties of locally indicable groups.
\newblock {\em Manuscripta Math.}, 44(1-3):71--93, 1983.

\bibitem{Kirby(1997)}
R.~Kirby.
\newblock Problems in low dimensional manifold theory.
\newblock In {\em Algebraic and geometric topology ({P}roc. {S}ympos. {P}ure
  {M}ath., {S}tanford {U}niv., {S}tanford, {C}alif., 1976), {P}art 2}, Proc.
  Sympos. Pure Math., XXXII, pages 273--312. Amer. Math. Soc., Providence,
  R.I., 1978.

\bibitem{Kitano-Suzuki-Wada(2005)}
T. Kitano, M. Suzuki and M. Wada.
\newblock . Twisted Alexander polynomials and surjectivity of a group homomorphism.
\newblock Algebr. Geom. Topol. 5 (2005), 1315--1324.

\bibitem{Linnell(1993)}
P.~A. Linnell.
\newblock Division rings and group von {N}eumann algebras.
\newblock {\em Forum Math.}, 5(6):561--576, 1993.


\bibitem{Liu(2013)}
Y.~{Liu}.
\newblock {Virtual cubulation of nonpositively curved graph manifolds.}
\newblock {\em {J. Topol.}}, 6(4):793--822, 2013.

\bibitem{Liu(2015)}
X.~Liu.
\newblock Degree of {$L^2$}-{A}lexander torsion for 3-manifolds.
\newblock Invent. Math.
Journal Profile
207, No. 3, 981-1030 (2017). 

\bibitem{Lott-Lueck(1995)}
J.~Lott and W.~L{\"u}ck.
\newblock ${L}\sp 2$-topological invariants of $3$-manifolds.
\newblock {\em Invent. Math.}, 120(1):15--60, 1995.

\bibitem{Lueck(1994b)}
W.~L{\"u}ck.
\newblock ${L}\sp 2$-{B}etti numbers of mapping tori and groups.
\newblock {\em Topology}, 33(2):203--214, 1994.

\bibitem{Lueck(1994a)}
W.~L{\"u}ck.
\newblock ${L}\sp 2$-torsion and $3$-manifolds.
\newblock In {\em Low-dimensional topology (Knoxville, TN, 1992)}, pages
  75--107. Internat. Press, Cambridge, MA, 1994.

\bibitem{Lueck(2002)}
W.~L{\"u}ck.
\newblock {\em {$L\sp 2$}-{I}nvariants: {T}heory and {A}pplications to
  {G}eometry and \mbox{{$K$}-{T}heory}}, volume~44 of {\em Ergebnisse der
  Mathematik und ihrer Grenzgebiete. 3.~Folge. A Series of Modern Surveys in
  Mathematics.}
\newblock Springer-Verlag, Berlin, 2002.

\bibitem{Lueck(2015twisting)}
W.~L\"uck.
\newblock Twisting {$L^2$}-invariants with finite-dimensional representations.
\newblock Preprint, arXiv:1510.00057 [math.GT], 2015.
\newblock To be published by the Journal of Topology and Analysis.

\bibitem{McMullen(2002)}
C.~T. McMullen.
\newblock The {A}lexander polynomial of a 3-manifold and the {T}hurston norm on
  cohomology.
\newblock {\em Ann. Sci. \'Ecole Norm. Sup. (4)}, 35(2):153--171, 2002.

\bibitem{Przytycki-Wise(2012)}
P.~Przytycki and D.~T. Wise.
\newblock Mixed 3-manifolds are virtually special.
\newblock  J. Amer. Math. Soc. 31 (2018), no. 2, 319--347.

\bibitem{Przytycki-Wise(2014)}
P.~Przytycki and D.~T. Wise.
\newblock Graph manifolds with boundary are virtually special.
\newblock {\em Journal of Topology}, 7:419--435, 2014.

\bibitem{Radstroem(1952)}
H.~R{\aa}dstr{\"o}m.
\newblock An embedding theorem for spaces of convex sets.
\newblock {\em Proc. Amer. Math. Soc.}, 3:165--169, 1952.

\bibitem{Reich(2006)}
H.~Reich.
\newblock {$L\sp 2$}-{B}etti numbers, isomorphism conjectures and
  noncommutative localization.
\newblock In {\em Non-commutative localization in algebra and topology}, volume
  330 of {\em London Math. Soc. Lecture Note Ser.}, pages 103--142. Cambridge
  Univ. Press, Cambridge, 2006.

\bibitem{Schick(2001b)}
T.~Schick.
\newblock ${L}\sp 2$-determinant class and approximation of ${L}\sp 2$-{B}etti
  numbers.
\newblock {\em Trans. Amer. Math. Soc.}, 353(8):3247--3265 (electronic), 2001.

\bibitem{Scott(1983)}
P.~Scott.
\newblock The geometries of $3$-manifolds.
\newblock {\em Bull. London Math. Soc.}, 15(5):401--487, 1983.

\bibitem{Silvester(1981)}
J.~R. Silvester.
\newblock {\em Introduction to algebraic ${K}$-theory}.
\newblock Chapman \&\ Hall, London, 1981.
\newblock Chapman and Hall Mathematics Series.

\bibitem{Thurston(1986norm)}
W.~P. Thurston.
\newblock A norm for the homology of {$3$}-manifolds.
\newblock {\em Mem. Amer. Math. Soc.}, 59(339):i--vi and 99--130, 1986.

\bibitem{Turaev(2002)}
V.~Turaev.
\newblock A homological estimate for the Thurston norm.
\newblock Unpublished manuscript, arXiv:math/0207267 [math.GT], 2002.

\bibitem{Wise(2012raggs)}
D.~T. Wise.
\newblock {\em From riches to raags: 3-manifolds, right-angled {A}rtin groups,
  and cubical geometry}, volume 117 of {\em CBMS Regional Conference Series in
  Mathematics}.
\newblock Published for the Conference Board of the Mathematical Sciences,
  Washington, DC, 2012.

\bibitem{Wise(2012hierachy)}
D.~T. Wise.
\newblock The structure of groups with a quasiconvex hierarchy.
\newblock Preprint, \texttt{http://www.math.mcgill.ca/wise/papers.html}, 2012.

\end{thebibliography}


\end{document}